\documentclass[11pt]{amsart}

\usepackage[colorlinks = true,
            linkcolor = red,
            urlcolor  = magenta,
            citecolor = black,
            anchorcolor = blue]{hyperref}
\usepackage{color}
\usepackage[shortalphabetic,initials,nobysame]{amsrefs}
\usepackage{upgreek}
\usepackage{tikz}
\usepackage{youngtab}
\usepackage[applemac]{inputenc} 
\usetikzlibrary{arrows}
\usepackage{xcolor}
\usepackage[all]{xy}
\usepackage{graphicx}
\usepackage{wrapfig}
\usepackage{yfonts}
\usepackage{graphicx}
\usepackage{amssymb}
\usepackage{epstopdf}
\usepackage{mathrsfs}
\usepackage[mathcal]{eucal}
\usepackage{bm}

\renewcommand{\eprint}[1]{\href{https://arxiv.org/abs/#1}{#1}}

\DeclareMathOperator{\Hom}{Hom}

\DeclareMathOperator{\Lie}{Lie}

\newcommand{\hcL}{\hat{\cL}}

\newcommand{\mfa}{\mathfrak{a}}
\newcommand{\mfp}{\mathfrak{p}}

\BibSpec{article}{%
    +{}  {\PrintAuthors}                {author}
    +{,} { \textit}                     {title}
     +{,} {}                            {note}
    +{.} { }                            {part}
    +{:} { \textit}                     {subtitle}
    +{,} { \PrintContributions}         {contribution}
    +{.} { \PrintPartials}              {partial}
    +{,} { }                            {journal}
    +{}  { \textbf}                     {volume}
    +{}  { \PrintDate}                {date}
    +{,} { \issuetext}                  {number}
    +{,} { \eprintpages}                {pages}
    +{,} { }                            {status}
    +{,} { \DOI}                   {doi}
    +{,} { \eprint}        {eprint}
      +{,} {\publisher}                {publisher}
    +{,} { \address}                   {address}
    +{}  { \parenthesize}               {language}
    +{}  { \PrintTranslation}           {translation}
    +{;} { \PrintReprint}               {reprint}
    +{.} {}                             {transition}
    +{}  {\SentenceSpace \PrintReviews} {review}
}

\newtheorem{Thm}{Theorem}[section]
\newtheorem{Lem}[Thm]{Lemma}
\newtheorem{Prop}[Thm]{Proposition}
\newtheorem{Cor}[Thm]{Corollary}

\theoremstyle{definition}
\newtheorem{Def}[Thm]{Definition}
\theoremstyle{remark}
\newtheorem{Rem}[Thm]{Remark}

\newtheoremstyle{named}{}{}{\itshape}{}{\bfseries}{.}{.5em}{#1 #3}
\theoremstyle{named}

\def\C{\mathbb{C}}
\def\Z{\mathbb{Z}}

\def\P{\mathbb{P}}

\def\fb{\mathfrak{b}}
\def\g{\mathfrak{g}}
\def\Frenkel:2013uda{\mathfrak{h}}

\def\sl{\mathfrak{sl}}

\def\cD{\mathcal{D}}

\def\cL{\mathcal{L}}

\def\cO{\mathcal{O}}

\def\cV{\mathcal{V}}
\def\cW{\mathcal{W}}

\def\a{\alpha}
\def\b{\beta}

\def\bo{\textbf{o}}

\def\=>{\Longrightarrow}

\def\to{\longrightarrow}

\def\o+{\oplus}
\def\bo+{\bigoplus}

\def\<{\langle}
\def\>{\rangle}
\def\({\left(}
\def\){\right)}

\def\^{\wedge}
\def\+{\dagger}

\def\half{\frac{1}{2}}

\def\dd[#1,#2]{\frac{d#1}{d#2}}
\def\del[#1,#2]{\frac{\partial #1}{\partial #2}}
\def\over[#1]{\overline{#1}}
\def\vec[#1]{\overrightarrow{#1}}

\makeatletter
\def\mr@ignsp#1 {\ifx\:#1\@empty\else #1\expandafter\mr@ignsp\fi}%
\newcommand{\upxiltiref}[1]{\begingroup
\xdef\mr@no@sparg{\expandafter\mr@ignsp#1 \: }%
\def\mr@comma{}%
\@for\mr@refs:=\mr@no@sparg\do{\mr@comma\def\mr@comma{,}\ref{\mr@refs}}%
\endgroup}
\makeatother

\newcommand{\hypref}[2]{\ifx\href\asklFrenkel:2013udaas #2\else\href{#1}{#2}\fi}

\newcommand{\figref}[1]{Fig.~\upxiltiref{#1}}

\tikzset{->-/.style={decoration={
  markings,
  mark=at position .5 with {\arrow{latex}}},postaction={decorate}}}
\tikzset{
    >=latex
    }

\newcommand{\wt}{\widetilde}

\newcommand{\nc}{\newcommand}
\nc{\on}{\operatorname}
\nc{\la}{\lambda}
\nc{\wh}{\widehat}
\nc{\ghat}{\wh\g}
\nc{\mb}{\mathbf}

\begin{document}
\title{3d Mirror Symmetry for Instanton Moduli Spaces}

\author[P. Koroteev]{Peter Koroteev}
\address[Peter Koroteev]{\newline
Department of Mathematics,\newline
University of California,\newline
Berkeley, CA 94720, USA, and\newline
Rutgers University, Piscataway, NJ, 08854, USA \newline
Email: pkoroteev@berkeley.edu\newline
\href{https://math.berkeley.edu/\textasciitilde pkoroteev/}{https://math.berkeley.edu/\textasciitilde pkoroteev/}
}

\author[A. M. Zeitlin]{Anton M. Zeitlin}
\address[Anton M. Zeitlin]{\newline
          Department of Mathematics, \newline
          Louisiana State University, \newline
          Baton Rouge, LA, USA; \newline
          Email: zeitlin@lsu.edu
          \newline \href{http://math.lsu.edu/\textasciitilde zeitlin}{http://math.lsu.edu/\textasciitilde zeitlin}}

\date{\today}

\numberwithin{equation}{section}

\begin{abstract}
We prove that the Hilbert scheme of $k$ points on $\mathbb{C}^2$ (Hilb$^k[\mathbb{C}^2]$) is self-dual under three-dimensional mirror symmetry using methods of geometry and integrability. Namely, we demonstrate that the corresponding quantum equivariant K-theory is invariant upon interchanging its K\"ahler and equivariant parameters as well as inverting the weight of the $\mathbb{C}^\times_\hbar$-action. First, we 
find a two-parameter family $X_{k,l}$ of self-mirror quiver varieties of type A and study their quantum K-theory algebras. 
The desired quantum K-theory of Hilb$^k[\mathbb{C}^2]$ is obtained via direct limit $l\to\infty$ and by imposing certain periodic boundary conditions on the quiver data. Throughout the proof, we employ the quantum/classical (q-Langlands) correspondence between XXZ Bethe Ansatz equations and spaces of twisted $\hbar$-opers.
In the end, we propose the 3d mirror dual for the moduli spaces of torsion-free rank-$N$ sheaves on $\mathbb{P}^2$ with the help of a different (three-parametric) family of type A quiver varieties with known mirror dual.
\end{abstract}

\maketitle

\setcounter{tocdepth}{1}
\tableofcontents


\section{Introduction}

\subsection{3d Mirror Symmetry}
3d mirror symmetry is a relatively new phenomenon in mathematics.
Like the well-studied case of 2d mirror symmetry, it assigns to a given space $X$ its {\it mirror dual} $X^!$ with certain relations to the geometric data of $X$. The 2d Mirror symmetry turned out to be very fruitful, first in attacking the problems in theoretical physics and enumerative geometry via Gromov-Witten theory; later it expanded its value to other parts of mathematics through several possible formulations. Similar success is expected from the 3d counterpart.

Both 2d and 3d mirror symmetries originate in physics. While 2d mirror symmetry has its roots in the topological string theory \cite{Hori:2003ic}, 3d mirror symmetry phenomenon relates $\mathcal{N}=4$ supersymmetric 3d gauge theories \cite{Intriligator:1996ex,deBoer:1996ck}. The moduli space of supersymmetric vacua of such 3d theories can in general be represented as a union of two hyper-K\"ahler cones which are called {\it Higgs branch} and {\it Coulomb branch}. The resolution of the former is described as a holomorphic symplectic quotient and for a wide variety of theories are known as Nakajima quiver varieties \cite{nakajima1994}. They are parameterized by vacuum expectation values of 3d hypermultiplets. The latter spaces describe the physics of 3d vector multiplets and, unlike Higgs branches, they receive quantum corrections. In the past several years a rigorous mathematical treatment of 3d $\mathcal{N}=4$ Coulomb branches has been initiated by \cite{Nakajima:2015txa,Braverman:2016aa}.  They usually do not have a nice formulation as a holomorphic symplectic quotient except for several rare but important situations such as quiver varieties of type A, hypertoric varieties, Hilbert scheme of points in $\mathbb{C}^2$ and its instanton generalizations, etc.

Both Higgs and Coulomb branches admit deformations by two families of complex-valued parameters --  masses and Fayet-Iliopoulos (FI) parameters that play the role of 3d Abelian coupling constants. The 3d mirror symmetry is understood as the transformation between two theories where Higgs and Coulomb branches are exchanged (together with the corresponding parameters). In other words, if theories A and B are 3d mirror dual to each other then the Coulomb branch of theory A is isomorphic to the Higgs branch of theory B as hyper-K\"ahler spaces and vice versa.

There are several formulations of what quantities one would like to compare for 3d mirror dual spaces $X$ and $X^!$. 
One may follow the ``enumerative" path of 2d Mirror symmetry, which relates enumerative invariants counting rational curves of $X$ and $X^!$ through Gromov-Witten theory. In this paper, we follow the enumerative approach applied to the three-dimensional case. Intrinsically it is the closest to the physical interpretation of 3d mirror symmetry, stating that partition functions of the dual theories agree. 

Assuming that  $X$ and $X^!$ are both Nakajima quiver varieties, mathematically, 3d mirror symmetry corresponds to the relations between {\it vertex functions} of the equivariant quasimaps K-theoretic counts to Nakajima varieties.  Given such variety $X$ the FI parameters $\{z_i\}$, in this case, are known as {\it K\"ahler parameters} which play a role of quantum parameters in quantum K-theory, the mass parameters $\{a_i\}$ are known as the {\it equivariant parameters} and correspond to the characters of the maximal torus $A$, acting on the framing of $X$. Upon 3d mirror symmetry K\"ahler and equivariant parameters interchange.
 
For a given Nakajima variety $X$, vertex functions are analytic in $\{z_i\}$. Moreover, it is known that they are the solutions to the difference equations in both K\"ahler and equivariant parameters, known as quantum Knizhnik-Zamolodchikov (qKZ) \cite{Frenkel:92qkz} and dynamical equations correspondingly \cite{Tarasov_2002, Tarasov_2005}. Upon 3d mirror symmetry, the equations are interchanged and the vertex functions for $X^!$ are analytic in $\{a_i\}$. Thus these two solutions, analytic in $\{z_i\}$ and $\{a_i\}$-variables are known as $z$ and $a$ solutions of the qKZ equations. The transition matrix between these solutions introduced geometrically by Aganagic and Okounkov \cite{Aganagic:2016aa} is known as an elliptic stable envelope and their comparison for $X$ and $X^!$ is crucial for the establishment of the full enumerative setup of 3d mirror symmetry. For a recent study of elliptic stable envelopes in the context of 3d mirror symmetry and its flavors see \cite{Kononov:2020aa,Rimanyi:2019aa,Rimanyi:2019ab,Smirnov:2020aa,Dinkins:2020aa,Dinkins:2020ab,korzeit,BRADEN20102002,Bullimore:2016nji,braden2016quantizations}.

\subsection{Quantum K-theory Rings and 3d Mirror Symmetry.}
The vertex functions are not the only objects that one could compare for mirror pairs of Nakajima varieties. Like in Gromov-Witten theory in the 2d mirror symmetry case, where the counts of rational curves produced the deformation of the cohomology ring of $X$, known as quantum cohomology ring, the K-theoretic counts of quasimaps produce similar deformation in the K-theoretic setting, namely the quantum equivariant K-theory ring $K^q_T(X)$, where $T=\mathbb{C}^{\times}_{\hbar}\times A$ is an extension of the torus $A$ by the algebraic torus $\mathbb{C}^{\times}_{\hbar}$ scaling cotangent directions in algebraic symplectic reduction producing Nakajima variety. One obtains a classical equivariant K-theory ring of $X$, merely by sending all K\"ahler parameters to zero\footnote{We want to emphasize that we are working with the equivariant \textit{quasimap} quantum K-threory which treats $X$ as a GIT quotient and is motivated by calculations in 3d SUSY gauge theories in which the corresponding gauge linear sigma models as well-defined.}.

As it will be shortly explained below, the relations in the quantum K-theory ring are defined by the asymptotic behavior of the solutions to qKZ equations and the statement of the mirror symmetry on the level of quantum K-theory rings is simply the isomorphism of algebras
\begin{eqnarray}\label{isomk}
K^q_T(X)=K^q_{T'}(X^!)\,.
\end{eqnarray}
In this paper, we shall prove this theorem in most generality when $X$ and $X^!$ are quivers of type A. 

The most important set of quiver varieties for us will be the two-parameter family $k,l\in\mathbb{N}$ of quivers which we call $X_{k,l}$  \figref{fig:AmodelNew}.
\begin{figure}[h]
\includegraphics[scale=0.28]{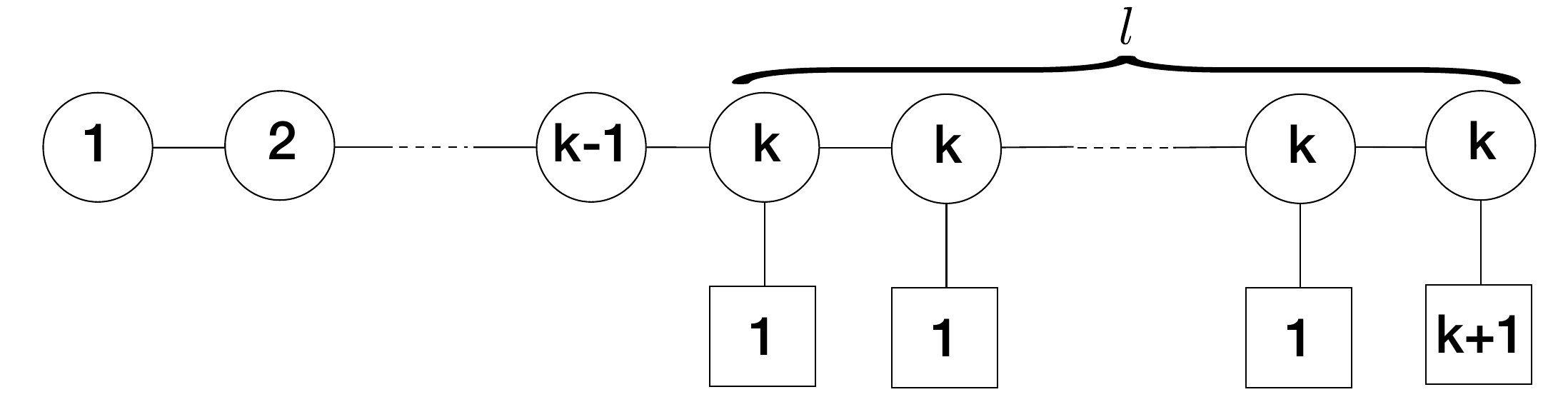}
\caption{Quiver family $X_{k,l}$.}
\label{fig:AmodelNew}
\end{figure}
Each representative of this family is 3d self-mirror dual. This fact alone is not surprising -- each quiver has $k+l$ equivariant and $k+l$ K\"ahler variables which are exchanged under 3d mirror symmetry. 

Our main theorem proves that the Hilbert scheme of $k$ points on $\mathbb{C}^2$ is self-dual in the sense of \eqref{isomk}, viewed as the simplest quiver variety with one loop. The idea of the proof is presented in \figref{fig:idea}. First, we take the direct limit $l\to\infty$ in the $X_{k,l}$ family to produce a one-parameter family of self-mirror quivers $X_{k,\infty}$ shown on the left in \eqref{isomk}. Second, after imposing certain periodic conditions on the quiver data we connect its K-theory with that of the quiver on the right of the figure. As a consequence, the latter quiver variety is also a self-mirror dual.
\begin{figure}[h]
\includegraphics[scale=0.32]{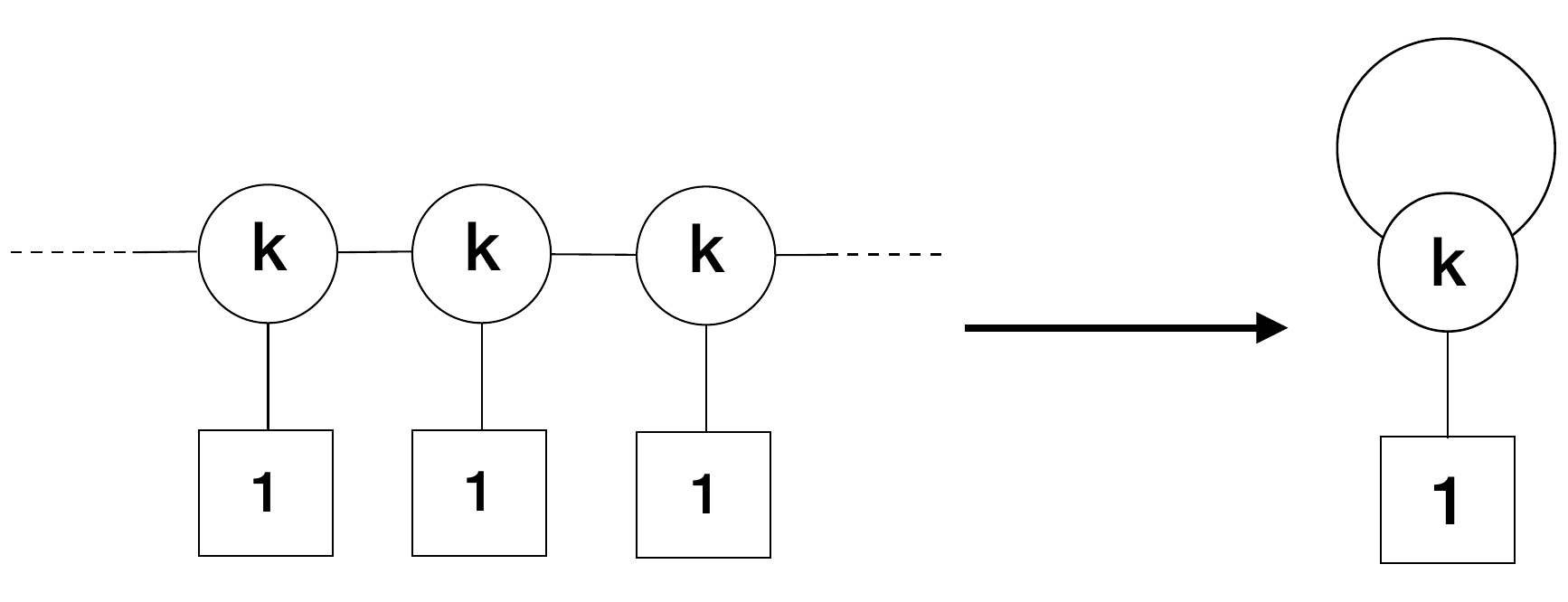}
\caption{Main idea of the proof of 3d self-mirror duality for ADHM spaces.}
\label{fig:idea}
\end{figure}

In the last Section, we also indicate how to generalize the mirror map to the Atiyah-Drinfeld-Hitchin-Manin (ADHM) instanton moduli spaces.

\subsection{Quantum K-theory rings of Nakajima Varieties and Bethe Algebra}
The difference equations, which govern the quasimap enumerative geometry of quiver varieties,  
have an intriguing connection with representation theory. Namely, the localized equivariant K-theory can be identified with the weight subspace of the tensor product of certain representations of the quantum algebra associated with a given quiver. 
In the case of quivers of ADE type, these representations are the finite-dimensional representations of the corresponding quantum affine Lie algebras. For quivers with one loop these algebras are known as {\it toroidal quantum algebras} \cite{Feigin:2009aa,Schiffmann_2013}. 
The aforementioned identification comes from the natural action of the quantum algebra in this space, the so-called R-matrix in particular. The appearance of the qKZ equations governing the vertex functions naturally emerges from this geometric realization of the R-matrix. Upon this identification, the equivariant parameters correspond to the evaluation parameters for quantum affine algebras and K\"ahler parameters parametrize the Cartan twist of the R-matrix. The deformation parameter of the quantum group $\hbar$ can be identified with the character of the algebraic torus scaling the cotangent directions in the symplectic reduction that produces the quiver variety.

At the same time, this representation space of the quantum algebra can be viewed as the {\it physical space} of the integrable model, known as XXZ spin chain. The essence of integrability is captured in the family of mutually commuting operators, known as {\it Bethe algebra} acting on this space, which is constructed using quantum algebra generators. A particular set of such operators naturally emerge from qKZ equations in the asymptotic which sends the difference parameter to zero and converts the qKZ equation into the eigenvalue problem for these operators. It is the question of simultaneous diagonalization of these operators' integrable family, which physicists address as finding the spectrum of the integrable model. In the 1980s physicists invented a procedure called {\it algebraic Bethe Ansatz} to solve this problem explicitly and diagonalize the elements of Bethe algebra simultaneously. The eigenvalues were parametrized by the solutions of the algebraic system of equations known as {\it Bethe Ansatz equations}. From the qKZ perspective, these equations could be read off the asymptotics 
of its solutions, which is defined by the so-called Yang-Yang function $\mathcal{Y}(\{a_i\}, \{z_i\},\hbar)$ \cite{Nekrasov:2011bc,Nekrasov:2009uh,Nekrasov:2009ui,Nekrasov:2013aa}, which, in its turn depends on K\"ahler and equivariant parameters. The Bethe equations emerge as the critical conditions $a_i\partial_{a_i}\mathcal{Y}=0$.

From a geometric perspective, Bethe algebra can be identified with the equivariant quantum K-theory ring of a Nakajima variety, which we discussed above. The natural generators of the quantum K-theory, the exterior powers of the tautological bundles are the natural generators of the Bethe algebra as well -- for each such tautological bundle the generating function for its exterior powers is known as a Baxter Q-operator in the theory of quantum integrable models \cite{Pushkar:2016qvw,Koroteev:2017aa}.

That leads us to another point of view on Bethe Ansatz equations, which is crucial for the following. Each such Q-operator associated with the tautological bundle has a counterpart, which geometrically can be characterized as the generator of the exterior powers of a ``dual'' bundle. Together they satisfy nonlinear difference equations for its eigenvalues, known as the $QQ$-system. Under mild non-degeneracy conditions, the space of solutions of the $QQ$-system is equivalent to the space of solutions of the system of Bethe Ansatz equations.

The $QQ$-system for semisimple Lie algebras can be understood from a representation-theoretic perspective: these are the relations between generators of the Grothendieck ring of the extended category of finite-dimensional modules, obtained by adding certain infinite-dimensional representations (known as prefundamenal representations) of Borel subalgebra only (see \cite{Frenkel:2016} for details). The $QQ$-systems in various contexts has been studied in \cite{Bazhanov:1998dq,MVflag,Mukhin_2003,Masoero_2016_SL,Masoero:2018rel}.

This $QQ$-system has another geometric realization from the difference version of certain connections on $\mathbb{P}^1$ known as $Z$-twisted $\hbar-$opers, which we discuss in the next subsection.

\subsection{The $(G,\hbar)$-opers and $QQ$-systems}
The $Z$-twisted $(G,\hbar)$-opers (here $G$ stands for a complex semisimple Lie group) with regular singularities, producing the $QQ$-systems for ADE type as well as for cyclic quivers, can be viewed as a deformation of the oper connections for principal G-bundles on $\mathbb{P}^1$, playing an important role in the geometric Langlands correspondence. In this example, studied in detail by E. Frenkel and his collaborators \cite{Frenkel:aa,Feigin:2006xs}, the opers with regular singularities and trivial monodromies were identified with the solutions of Bethe Ansatz equations of the Gaudin model. 
The latter is a certain scaling limit of the XXZ model. It turns out that  $(G,\hbar)$-opers, which are certain $\hbar$-connections in the sense of Baranovsky and Ginzburg \cite{Baranovsky:1996}, produce a correct $\hbar$-deformation of this correspondence replacing Gaudin model with XXZ model. The $\hbar$-connection is related to $\hbar$ multiplicative action on $\mathbb{P}^1$: $z\rightarrow \hbar z$, so that locally it is understood as a group-valued object $A(z)\in G(z)$, with the $\hbar$-gauge transformations $A(z)\to g(\hbar z)A(z)g^{-1}(z)$, $g(z)\in G(z)$. The oper condition produces certain constraints on $A(z)$ related to Borel reductions of the corresponding G-bundle over $\mathbb{P}^1$, where this connection acts. In the case of $G=SL(r+1)$ one can pass to the associated bundles and replace the reduction by Borel subgroup with the full flag of vector bundles. The $Z$-twisted condition is the statement that the oper is gauge equivalent to a regular semisimple element $Z\in H\subset G$. Of course, it is defined up to the action of the Weyl group. To fix the element $Z$, one requires the corresponding $\hbar$-connection to preserve another Borel reduction of G-bundle (in the case of $SL(r+1)$ another full flag of vector bundles). The resulting object consisting of the following elements: i) principal G-bundle over $P^1$, ii) its reduction to Borel subgroup, producing $(G,\hbar)$-oper condition, iii) the other Borel reduction which this  $(G,\hbar)$-oper preserves, and finally iv) the $Z\in H$ twist, is known as Miura $(G,\hbar)$-oper.

As it was proved in \cite{Frenkel:2020}, there is a one-to-one correspondence between $Z$-twisted Miura $(G,\hbar)$-oper with certain non-degeneracy conditions and solutions of the equations of the $QQ$-system/XXZ Bethe Ansatz equations. 

The {\it framed quiver}, which defines the quiver variety 
together with the equivariant data, can as well define a class of $Z$-twisted Miura $(G,\hbar)$-opers: we identify parameters of framing with regular singularities, the K\"ahler parameters $\{z_i\}$ with the root data of 
$Z=\prod_iz_i^{\check{\alpha}_i}$ and the degrees of the polynomials in the $QQ$-system with the ranks of the tautological bundles of quiver variety.

Using the geometric interpretation, one can identify the algebra of functions on the space of Miura $(G,\hbar)$-opers associated with a given quiver with the algebra $K^q_T(X)$ for a given quiver variety $X$. 

An interesting feature of the introduced formalism is the following. In \cite{Frenkel:2020} we developed the transformations, which we called quantum B\"acklund transformations, between various Miura $(G,\hbar)$-opers associated to a given $(G,\hbar)$-oper. 
These various Miura $(G,\hbar)$-opers correspond to quiver varieties (provided that they exist) with isomorphic quantum K-theory algebras. The quantum B\"acklund transformations correspond to the change of stability parameter. Thus functions on $(G,\hbar)$-oper produce a family of isomorphic algebras to a given $(G,\hbar)$-oper.

We note, that the difference oper construction can be used similarly with additive action on $\mathbb{P}^1$. Then the related $QQ$-system produces the Bethe equations for $XXX$ integrable model, which serves as an intermediate model between XXZ and Gaudin integrable models. Instead of the quantum group, it is based on Yangians. Geometrically, in the context of quiver varieties, $QQ$-system produces relations for the quantum cohomology ring.

Some of our primary targets in this paper are the $A_r$-quivers and their generalizations, $A_{\infty}$-quivers, closely related to ADHM instanton spaces. In \cite{Koroteev:2020mxs} we introduced a notion of $Z$-twisted Miura $(\overline{GL}(\infty), \hbar)$-opers, based on a completion of $GL(\infty)$ group, described by the infinite $QQ$-systems. One can impose periodic conditions on them according to the natural translational symmetry of the $A_{\infty}$ Dynkin diagram, thereby leaving us with a finite number of parameters. We refer to such 
$(\overline{GL}(\infty), \hbar)$-opers as {\it toroidal} $\hbar$-opers. The resulting algebra of functions on toroidal $\hbar$-opers coincides with the equivariant quantum K-theory of ADHM instanton spaces treated as cyclic quiver varieties.

In \cite{KSZ} we introduced another system of parameters for an $(SL(r+1),\hbar)$-oper, which corresponds to the roots of the polynomials defining the section of the line bundle from the flag of bundles defining the Miura $(SL(r+1),\hbar)$-oper connection (cf. the standard Bethe parameters corresponding to the roots of polynomials of the $QQ$-system).

These parameters, which we refer to as momenta of {\it oper magnetic frame}, allow an interesting geometric formulation. For example, \cite{Gaiotto:2013bwa,Koroteev:2017aa}, in the case of quiver variety corresponding to cotangent bundle to full flag variety, momenta of magnetic frame accompanied with K\"ahler parameters understood as the values of coordinates in symplectic phase space, provide the interpretation of the algebra of functions on the space of $Z$-twisted $(SL(r+1),\hbar)$-Miura opers as algebra of functions on the intersection of two Lagrangian subvarieties. One of those subvarieties is given by setting coordinates equal to the K\"ahler parameters and another one is defined by setting Hamiltonians of {\it trigonometric Ruijsenaars-Schneider (tRS) integrable many-body system} to be equal to symmetric functions of the equivariant parameters. One can generalize that interpretation to other quivers by using the degeneration of K\"ahler parameters in the case of a full flag (which means we will consider degenerate Miura $\hbar$-opers for the full flag). The resulting momenta variables, in this case, will be referred to as momenta of the {\it true magnetic frame}. They are in one-to-one correspondence with the oper magnetic frame.

This surprising relation between quantum spin chain integrable models of XXZ, XXX, and Gaudin type, as well as some other variations, and multiparticle systems which involve tRS systems and their various limits like Calogero model, is known as {\it quantum/classical duality} \cite{Gaiotto:2013bwa}. In the following, we will describe other incarnations of the quantum/classical duality, their relations to 3d mirror symmetry, and string theory.

\subsection{Quantum/Classical Duality and 3d Mirror Symmetry as Bispectral Duality}
On the level of integrable systems this puts 3d mirror symmetry in the network of the so-called {\it bispectral dualities} (in fact, terms 3d mirror and bispectral can be used interchangeably in this context). These dualities have been studied extensively in the literature on integrable systems and representation theory \cite{Feher:2009uw,Mukhin:2009uh,Feher:2010va,Feher:2012wc,Mukhin:wg}. More recently Gaiotto and the first author widely generalized these results by extending the network of dualities (see Section 4 in \cite{Gaiotto:2013bwa}). One can study bispectral dualities among tRS models which live on the classical side of the quantum/classical duality as well as among XXZ spin chains that belong to the quantum side. Thus at the trigonometric level (without invoking elliptic functions) bispectral dualities and quantum/classical dualities together comprise a closed network of dualities (see \figref{fig:network}). 
\begin{figure}[h]
\includegraphics[scale=0.4]{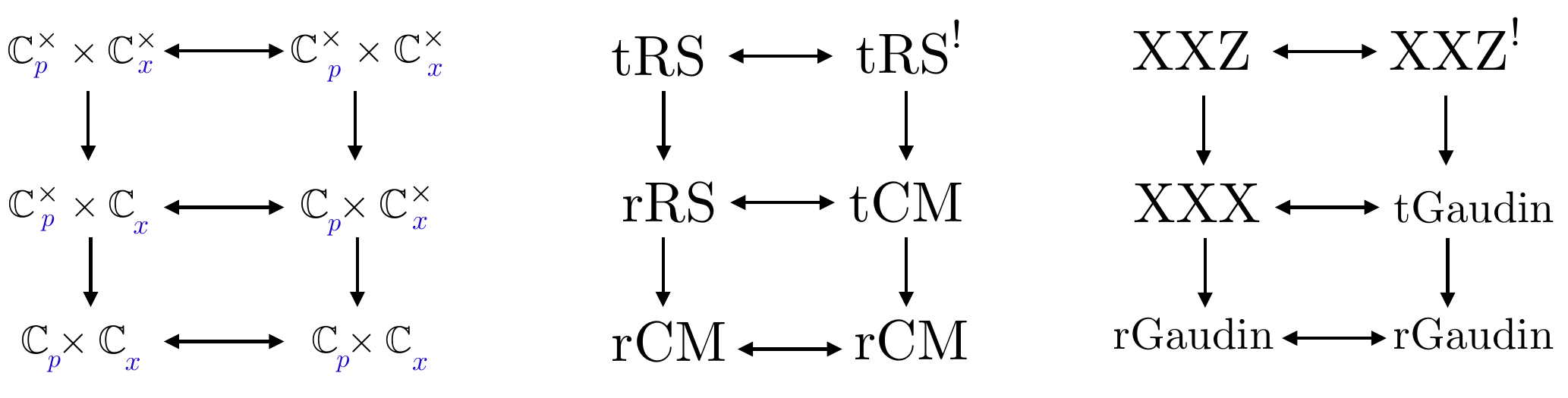}
\caption{The network of bispectral and quantum/classical dualities with additive $\mathbb{C}$ and multiplicative $\mathbb{C}^\times$ dependence on momenta $p$ and coordinates $x$ between classical many-body systems (middle) and quantum spin chains (right).}
\label{fig:network}
\end{figure}

The acronyms in the figure are the following. rCM, tCM -- rational and trigonometric Calogero-Moser systems; rRS, tRS -- rational and trigonometric Ruijsenaars-Schneider systems; rGaudin, tGaudin -- rational and trigonometric Gaudin models, XXX - Heisenberg spin chain. The horizontal arrows are the bispectral dualities that interchange momenta $p$ and coordinates $x$ in the corresponding classical systems in the middle of \figref{fig:network}. For instance, the Hamiltonians of the rRS model are trigonometric in momenta and rational in coordinates, and in the tCM model, they are the opposite. The vertical downward arrows show limits in the parameters \cite{Gaiotto:2013bwa}. For instance, if we send the speed of light $c$ to infinity the tRS$^!$ model becomes the tCM model. In the former, the momenta enter its Hamiltonians in the form $\exp(p/c)$ so they become rational in the nonrelativistic limit.

In the right column of \figref{fig:network} we see a similar pattern of dualities and limits for the dual spin chains. As explained above there exists a quantum-classical duality between tRS and XXZ chains. As part of the dictionary tRS coordinates become XXZ twists, tRS energy levels translate into XXZ anisotropies (shifts of the spectral parameter), and the tRS coupling constant is identified with the deformation parameter $\hbar$ on the spin chain. All these variables take values in $\mathbb{C}^\times$. We can now take identical limits on these parameters and express them on the spin chain side. Thusly XXZ can degenerate into XXX or into a trigonometric Gaudin model depending on which parameter is downgraded from $\mathbb{C}^\times$ to $\mathbb{C}$ and so on.

The aspects of the bispectral duality for the toroidal setting were studied previously in \cite{Feigin:2017,Feigin_2019}. It was shown that the Bethe algebras for quantum toroidal  $(\mathfrak{gl}_n,\mathfrak{gl}_m)$, constructed from two sets of parameters -- twists and evaluations parameters, commute with each other. This provides a hint at the existence of the duality between the solutions of the corresponding XXZ Bethe equations.

\vskip.1in

Let us describe the quantum/classical duality here for the case of XXZ Bethe equations for the cotangent bundle to the partial flag variety $T^*\mathbb{F}l_n$ (more details to follow in Section \ref{Sec:tRSElectric}). 
On the classical side, we have a pair of dual tRS models.
The Hamiltonians of the tRS model depend on coordinates $\chi_i$ and momenta $p_i\,,i=1,\dots,n$ by taking traces of powers of the tRS Lax matrix $T$. This Lax matrix together with another diagonal matrix $M$ built out of $\chi_i$ satisfy the defining relation for the Calogero-Moser space \cite{Oblomkov2004} which treats $M$ and $T$ on equal footing up to rescaling $\hbar\to\hbar^{-1}$ \eqref{eq:FlatConNew} (for the tRS model $\hbar$ plays the role of the coupling constant). 
The bispectral duality in this language consists in merely exchanging $M$ with $T$ and $\hbar$ with $\hbar^{-1}$. The dual model, tRS$^!$ is described by a similar relation where $T$ is diagonal and $M$ plays the role of the tRS Lax matrix.

The explicit map from the tRS model to its dual tRS$^!$ can be implemented via the Lagrangian correspondence. This also enables the connection with the dual XXZ chain living on the quantum side. A Lagrangian subvariety $\mathcal{L}\subset \mathcal{M}\times \mathcal{M}^!$ in the product of the phase spaces of tRS and tRS$^!$ respectively coincides with space of solutions of the XXZ Bethe ansatz equations for $T^*\mathbb{F}l_n$ where $\hbar$ plays the role of the Planck's constant. Moreover, this Lagrangian admits two distinct 3d mirror dual descriptions in terms of a pair of $T^*\mathbb{F}l_n$ varieties -- in one description $\chi_i$ are the equivariant parameters, in the other, they are the K\"ahler parameters.

By combining the magnetic frame and the oper frame formulations of the algebra of functions on the space of $Z$-twisted $(SL(r+1),\hbar)$-Miura opers/quantum equivariant K-theory, one obtains 3d mirror self-duality of the full flag variety $T^*\mathbb{F}l_n$. 

This construction can be generalized to other quivers by degenerations of equivariant parameters as we will show in this paper. One needs to consider restricted tRS models and their duals instead where certain constraints on their coordinates are implemented.
For instance, in the case of the cotangent bundle to partial flag variety, the algebra of functions on the space of $Z$-twisted $(SL(r+1),\hbar)$-Miura opers as algebra of functions on the intersection of two Lagrangian subvarieties, one of which is given by tRS Hamiltonians being equal to symmetric functions of K\"ahler parameters and another one is obtained by setting coordinates equal to equivariant parameters \cite{Zabrodin:}.

Another important version of this duality for the XXZ spin chain is as follows. We already discussed Baxter Q-operators and their relation to the generators of Bethe algebra. There is, however, another system of generators, known as {\it fundamental Hamiltonians} of the spin chain, emerging as coefficients in the polar decomposition of {\it transfer matrices}, which are related to the Q-operators via a rational transformation. 
The eigenvalues of the transfer matrix naturally emerge in the $\hbar$-oper formulation as well. Namely, they serve as generating functions for the coordinates on $Z$-twisted $(SL(r+1),\hbar)$-opers, and the rational transformation from the Q-operators to transfer matrices is known as $\hbar$-Miura transform \cite{Frenkel:ls}. 

\subsection{String Theory Origins}
The set of ideas to relate quantum spin chains and classical integrable systems in the way presented in this paper originates from string theory. Physics interpretation of the quantum/classical duality between the XXZ spin chain and tRS model was developed in \cite{Gaiotto:2013bwa} using Type IIB brane configuration which we shall briefly review in Section \ref{Sec:KThe}. 

The first brane configuration describes (in the field theory limit) the ground state of a 3d $\mathcal{N}=2^*$ quiver gauge theory whose color and flavor labels precisely correspond to quiver data of $X$ from \eqref{isomk}. More specifically, the equations for the supersymmetric vacua of this gauge theory are in one-to-one correspondence with the XXZ Bethe Ansatz equations such that Bethe roots play the role of vacuum expectation values of twisted chiral superfields,  the equivariant parameters are mapped to the masses of hypermultiplets, and twists are the K\"ahler parameters of $X$ \cite{Nekrasov:2009uh}.

The second brane configuration, which can be obtained from the first one by continuous transformation, describes the moduli space of vacua of a four-dimensional $\mathcal{N}=2^*$ theory compactified on a segment with BPS boundary condition on the left and the right. This moduli space is described by the moduli space of flat connections on punctured torus \cite{Donagi:1995cf}, which in this work we shall treat algebraically as the Calogero-Moser space \cite{cite-key}. The latter space is described by a relation that involves two matrices $M$ and $T$ whose size is equal to the number of particles in the dual tRS model.
The left boundary conditions of the gauge fields yield a Lagrangian submanifold described above which fixes the eigenvalues of $M$ in terms of the masses and imposes the so-called Slodowy structure on $T$. The right boundary conditions give another Lagrangian submanifold which fixes the eigenvalues of $T$ in terms of the FI parameters and imposes a Slodowy structure on $M$. The vacua of the quiver gauge theory described in the previous paragraph are the intersection points of the two Lagrangian submanifolds.

\subsection{Plan of the Paper and Main Results}
Section \ref{Sec:OpersBethe} is devoted to the structure of nondegenerate $Z$-twisted $(SL(r+1), \hbar)$-opers. It is mostly a review, with the results adopted to the needs of this paper (for a more detailed account see \cite{Koroteev:2020mxs,KSZ,Frenkel:2020}). We discuss the aforementioned notions of Z-twisted Miura $\hbar$ and their relations to the $QQ$-system/Bethe equations, quantum B\"acklund transformations, and the $\hbar$-Wronskian formulation which leads to the discussed above magnetic oper frame.

In Section \ref{Sec:QQA} we discuss the relation between $QQ$-systems and various $A_r$-quivers. We discuss some concrete examples, including an important family of $X_{k,l}$ quivers. They generalize the quiver corresponding to the cotangent bundle of the full flag variety by inserting a chain of extra $l$ vertices with 1-dimensional framings. An important feature of this example is that the oper magnetic frame coincides with the true magnetic frame as in the case $l=0$. That allows us to give a formulation of the algebra of functions on the related space of $Z$-twisted Miura $(SL(r+1), \hbar)$-opers as the algebra of functions on the intersection of two Lagrangian subvarieties as discussed above, one of which is given by setting tRS Hamiltonians to be equal to certain symmetric functions of the regular singularities.

In Section \ref{Sec:tRSElectric} we discuss the tRS system in more detail and introduce the notion of the electric frame. We start from a quiver corresponding to a cotangent bundle of the partial flag and then introduce the recursive procedure of how to degenerate it to obtain electric frame formulation for other quivers, in particular for $X_{k,l}$ family.

In Section \ref{Sec:KThe} we first discuss $A_r$-quiver varieties and the interpretation of quantum $K$-theory rings in terms of algebras of functions on $(SL(r+1), \hbar)$-opers corresponding to given quivers and the notion of 3d mirror symmetry on the level of those algebras. The second part of that section is devoted to the ``D-brane arithmetic": the string theory interpretation of quantum K-theories on quiver varieties and the prescription on how to calculate its 3d Mirror.

 Section \ref{Sec:MirrorElMag} provides a general prescription for proving the 3d mirror symmetry in the case of $A_r$ quivers using the quantum/classical duality, i.e. identifying electric and magnetic frames of 3d Mirror objects. We prove the 3d mirror theorem in type A in full generality and then discuss quiver  3d mirror self-dual family $X_{k,l}$.
 
 Section \ref{Sec:DirectLimit} is devoted to the properly computed direct limit of the equivariant quantum K-theory rings of $X_{k,l}$ family, taken with respect to $l$. As a result, we obtain the 3d mirror self-dual $A_{\infty}$-quiver with every vertex having a 1-dimensional framing. Imposing a natural periodicity condition (so that a mirror map is equivariant with respect to it), the resulting quantum equivariant K-theory ring turns out to be one of the Hilbert space of points on $\mathbb{C}^2$. This we prove 3d-mirror 
 self-duality of the moduli spaces of torsion-free sheaves of rank 1 on $\mathbb{C}^2$.

In section \ref{Sec:MirrorADHM} we propose the 3d Mirror dual varieties realized as Nakajima varieties associated to cyclic quiver) to torsion-free sheaves of rank $N$, which on the level of quantum equivariant K-theory emerge from $A_{\infty}$-quivers with periodicity conditions. We show how the same method, using finite rank $A_r$-quivers will allow us to prove the proposed duality.

\subsection{Acknowledgements}
We thank A. Smirnov for discussions at the early stage of the project. P.K. is partially supported by AMS Simons grant. A.M.Z. is partially supported by Simons Collaboration Grant 578501 and NSF grant
DMS-2203823.

\section{$(SL(r+1),\hbar)$-opers and Bethe Ansatz}\label{Sec:OpersBethe}

\subsection{Group-theoretic data and notations.}
Consider $SL(r+1)$ be the simple algebraic group of invertible $(r+1)\times (r+1)$
matrices over $\mathbb{C}$.  We fix a Borel subgroup $B_-$ with unipotent
radical $N_-=[B_-,B_-]$ of lower triangular matrices and strictly lower triangular matrices correspondingly. The maximal torus is the corresponding set of diagonal matrices $H\subset B_-$.  Let $B_+$ be the opposite Borel subgroup containing $H$.  
Let $\{
\alpha_1,\dots,\alpha_{r} \}$ be the set of positive simple roots for
the pair $H\subset B_+$.  Let $\{ \check\alpha_1,\dots,\check\alpha_{r}
\}$ be the corresponding coroots. Then the elements of the Cartan
matrix of the Lie algebra $\mathfrak{sl}(r+1)$ of $G$ are given by $a_{ij}=\langle
\alpha_j,\check{\alpha}_i\rangle$. The Lie algebra $\mathfrak{sl}(r+1)$ has Chevalley
generators $\{e_i, f_i, \check{\alpha}_i\}_{i=1, \dots, r}$, so
that $\fb_-=\Lie(B_-)$ is generated by the $f_i$'s and the
$\check{\alpha}_i$'s and $\fb_+=\Lie(B_+)$ is generated by the $e_i$'s 
and the $\check{\alpha}_i$'s. 
In the defining representation $\check{\alpha}_i\equiv E_{ii}-E_{i+1,i+1}$, $e_i\equiv E_{i,i+1}$, $f_i\equiv E_{i-1,i}$, where $E_{ij}$ stand for the matrix  with the only nonzero element 1 at ij-th place. 
The fundamental weights $\omega_1,\dots\omega_r$ are defined by the condition $\langle \omega_i,
\check{\alpha}_j\rangle=\delta_{ij}$. 

\subsection{Definition of Miura $(SL(r+1),\hbar)$-oper}

Let's consider the automorphism $M_\hbar: \P^1 \to \P^1$ sending $z \mapsto \hbar z$, where
$\hbar\in\C^\times$ is {\em not} a root of unity. For any bundle $F$ over $\P^1$, we denote $F^{\hbar}$ its pull-back under $M_{\hbar}$.

\begin{Def}    \label{qopflag}
  A meromorphic $(GL(r+1),\hbar)$-{\em oper}  on
  $\mathbb{P}^1$ is a triple $(A,E, \mathcal{L}_{\bullet})$, where $E$ is  a vector bundle of rank $r+1$ and $\mathcal{L}_{\bullet}$ is the corresponding complete flag of the vector bundles, 
  $$\mathcal{L}_{r+1}\subset ...\subset \mathcal{L}_{i+1}\subset\mathcal{L}_i\subset\mathcal{L}_{i-1}\subset...\subset \mathcal{L}_1=E,$$ 
  where $\mathcal{L}_{r+1}$ is a line bundle, so that 
  $A\in \Hom_{\cO_{U}}(E,E^\hbar)$ 
  satisfies the following conditions:\\ 
i) $A\cdot \mathcal{L}_i\subset \mathcal{L}_{i-1} $,\\
ii)  There exists a Zariski open dense subset $U \subset \P^1$, such that the restriction of
  the connection $A\in Hom(\mathcal{L}_{\bullet}, \mathcal{L}^\hbar_{\bullet})$ to $U \cap M_\hbar^{-1}(U)$, which belongs to $GL(r+1)$ and satisfies the condition that the induced operator 
  $\bar{A}:\mathcal{L}_{i}/\mathcal{L}_{i+1}\to \mathcal{L}^{\hbar}_{i-1}/\mathcal{L}^{\hbar}_{i}$ is an isomorphism on $U \cap M_\hbar^{-1}(U)$.\\
  An $(SL(r+1),\hbar)$-$oper$ is a $(GL(r+1),\hbar)$-oper with the condition that $det(A)=1$ on $U \cap M_\hbar^{-1}(U)$.
\end{Def}

Let us choose a trivialization so that $\mathcal{L}_{r+1}$ is generated by the vector $(0,\dots, 0, 1)$. Then we obtain that locally the $\hbar$-oper connection can be represented in the form

\begin{equation}\label{qop1}
A(z)=
\begin{pmatrix}
\star & \phi_1(z) & 0 & 0 & \dots &0& 0\\
\star& \star & \phi_2(z)  & 0 & \dots & 0&0\\
\star & \star &  \star & \phi_3(z)  & \dots & 0& 0\\
\vdots& \vdots & \cdots & \ddots & \ddots & \vdots & \vdots \\
\vdots& \vdots & \cdots & \dots & \ddots & \phi_{r-1}(z) & 0 \\
\star & \star & \star & \dots &\dots &\star &\phi_r(z)\\
\star& \star & \star & \dots &\dots&\star &\star
\end{pmatrix}
\end{equation}

where elements on the superdiagonal $\phi_i(z) \in\C(z)$ and  the rest matrix elements are such that
their zeros and poles are outside the subset $U \cap M_\hbar^{-1}(U)$ of
$\P^1$. 

Now we extend further this definition, adding extra {\it Miura structure}.

\begin{Def}    \label{Miuraflag}
  A {\em Miura $(SL(r+1),\hbar)$-oper} on $\mathbb{P}^1$ is a quadruple
  $(E, A, \mathcal{L}_{\bullet}, \hat{\mathcal{L}}_{\bullet})$, where $(E, A, \mathcal{L}_{\bullet})$ is a
  meromorphic $(SL(r+1),\hbar)$-oper on $\P^1$ and $\hat{\mathcal{L}}_{\bullet}=\{\mathcal{L}_i\}$  is another full flag 
  of subbundles in $E$ that is preserved by the
  $\hbar$-connection $A$.
\end{Def}

These two flags of bundles have various relative positions, namely if 
$\mathcal{L}_{\bullet}=g\cdot \hat{\mathcal{L}}_{\bullet}$ and $g\in B_-wB_-$, we say that they are in relative position $w$, where $w$ is Weyl group element of $SL(r+1)$. One can show 
that on a certain Zariski dense subset they are in {\it generic} position, namely $w=w_0$, the longest element from the Weyl group. It was shown in \cite{Frenkel:2020} that  Miura opers satisfy the following structural theorem: 

\begin{Thm}    \label{gen elt}
 Every Miura $(SL(r+1), \hbar)$-oper associated with the oper connection (\ref{qop1})can be written
in the form:
\begin{equation}    \label{gicheck}
\prod_i [g_i(z)]^{\check{\alpha}_i}e^{\frac{\phi_i(z)}{g_i(z)}e_i}, \qquad
g_i(z) \in \mathbb{C}(z),
\end{equation}
\end{Thm}
In other words, these are the matrices with nontrivial elements on the diagonal and above the diagonal.
The diagonal elements generate the abelian $Cartan$ $connection$ associated to Miura oper, namely 
$$
A^H(z)=\prod_i [g_i(z)]^{\check{\alpha}_i}.
$$  
\subsection{$Z$-twisted Miura $\hbar$-opers}

In this paper, we consider a class of (Miura) $\hbar$-opers that are gauge
equivalent to a constant element of $SL(r+1)$ (as $(SL(r+1),\hbar)$-connections). Moreover, we assume that such an element 
$Z$ be the regular element of the maximal torus $H$. One can express it as follows 
\begin{equation}    \label{Z}
Z = \prod_{i=1}^r \zeta_i^{\check \alpha_i}, \qquad \zeta_i \in
\C^\times.
\end{equation}

\begin{Def}    \label{Ztwoper}
  A {\em $Z$-twisted $(SL(r+1),\hbar)$-oper} on $\mathbb{P}^1$ is a $(SL(r+1),\hbar)$-oper
  that is equivalent to the constant element $Z \in H \subset H(z)$
  under the $\hbar$-gauge action of $SL(r+1)(z)$, i.e. if $A(z)$ is the
  meromorphic oper $\hbar$-connection (with respect to a particular
  trivialization of the underlying bundle), there exists $g(z) \in
  SL(r+1)(z)$ such that
\begin{eqnarray}    \label{Ag}
A(z)=g(\hbar z)Z g(z)^{-1}.
\end{eqnarray}
A {\em $Z$-twisted Miura $(SL(r+1),\hbar)$-oper} is a Miura $(SL(r+1),\hbar)$-oper on
$\mathbb{P}^1$ that is equivalent to the constant element $Z \in H
\subset H(z)$ under the $\hbar$-gauge action of $B_+(z)$, i.e.
\begin{eqnarray}    \label{gaugeA}
A(z)=v(\hbar z)Z v(z)^{-1}, \qquad v(z) \in B_+(z).
\end{eqnarray}
\end{Def}

It follows from Definition \ref{Ztwoper} that any $Z$-twisted
$(SL(r+1),\hbar)$-oper is also $Z'$-twisted for any $Z'$ in the Weyl group orbit of
$Z$. But if we endow it with the structure of a $Z$-twisted Miura
$(SL(r+1),\hbar)$-oper (by adding a flag $\hat{\mathcal{L}}_\bullet$ preserved by the
oper $\hbar$-connection), then we fix a specific element in this
$S_{r+1}$-orbit.

Thus we have the following Proposition, which allows to characterize $Z$-twisted Miura $(SL(r+1),\hbar)$-opers 
associated to $Z$-twisted $(SL(r+1),\hbar)$-opers.

\begin{Prop}    \label{Z prime}
  Let $Z \in H$ be regular. For any $Z$-twisted $(SL(r+1),\hbar)$-oper $(E,A,\mathcal{L}_{\bullet})$
  and any choice of the flag $\hat{\mathcal{L}}_{\bullet}$ preserved
  by the oper $\hbar$-connection $A$, the resulting Miura $(SL(r+1),\hbar)$-oper is
  $Z'$-twisted for a particular $Z' \in w \cdot Z$, where $w$ is an element of the Weyl group of $SL(r+1)$. 
\end{Prop}

Now, if our Miura $\hbar$-oper is $Z$-twisted (see Definition
\ref{Ztwoper}), then we also have $A(z)=v(\hbar z)Z v(z)^{-1}$, where
$v(z)\in B_+(z)$.  Since $v(z)$ can be written as
\begin{equation}    \label{vz}
v(z)=
\prod_i y_i(z)^{\check{\alpha}_i} n(z), \qquad n(z)\in N_+(z), \quad
y_i(z) \in \C(z)^\times,
\end{equation}
the Cartan $\hbar$-connection $A^H(z)$ has the form
\begin{equation}    \label{AH1}
A^H(z)=\prod_i
y_i(\hbar z)^{\check{\alpha}_i} \; Z \; \prod_i y_i(z)^{-\check{\alpha}_i}
\end{equation}
and hence we will refer to $A^H(z)$ as $Z$-{\em twisted Cartan
  $\hbar$-connection}. This formula shows that $A^H(z)$ is completely
determined by $Z$ and the rational functions $y_i(z)$, namely:
\begin{equation}    \label{giyi}
g_i(z)=\zeta_i\frac{y_i(\hbar z)}{y_i(z)}\,.
\end{equation}

We note that $A^H(z)$ determines the $y_i(z)$'s uniquely
up to scalar.  

\subsection{Miura $\hbar$-opers with regular singularities and nondegeneracy} 
Let $\{ \Lambda_i(z) \}_{i=1,\ldots,r}$ be a collection of
non-constant polynomials.

\begin{Def}    \label{d:regsing}
  A $(SL(r+1),\hbar)$-{\em oper with regular singularities determined by $\{
    \Lambda_i(z) \}_{i=1,\ldots,r}$} is a $\hbar$-oper on $\P^1$ whose
  $\hbar$-connection \eqref{qop1} may be written in the form
\begin{equation}    \label{Lambda}
A(z)=
\begin{pmatrix}
\star & \Lambda_1(z) & 0 & 0 & \dots &0& 0\\
\star& \star & \Lambda_2(z)  & 0 & \dots & 0&0\\
\star & \star &  \star & \Lambda_3(z)  & \dots & 0& 0\\
\vdots& \vdots & \cdots & \ddots & \ddots & \vdots & \vdots \\
\vdots& \vdots & \cdots & \dots & \ddots & \Lambda_{r-1}(z) & 0 \\
\star & \star & \star & \dots &\dots &\star &\Lambda_r(z)\\
\star& \star & \star & \dots &\dots&\star &\star
\end{pmatrix}
\end{equation}

  {\em A Miura $(SL(r+1),\hbar)$-oper with regular singularities determined by
polynomials $\{ \Lambda_i(z) \}_{i=1,\ldots,r}$} is a Miura
  $(SL(r+1),\hbar)$-oper such that the underlying $\hbar$-oper has
regular singularities determined by $\{ \Lambda_i(z)
\}_{i=1,\ldots,r}$.
\end{Def}

The following theorem follows from Theorem \ref{gen elt} and gives an explicit parameterization of generic elements from the space of Miura opers.

\begin{Thm}    \label{Miura form}
For every Miura $(SL(r+1),\hbar)$-oper with regular singularities determined by
the polynomials $\{ \Lambda_i(z) \}_{i=1,\ldots,r}$, the underlying
$\hbar$-connection can be written in the form 
\begin{equation}    \label{form of A}
A(z)=\prod_i
[g_i(z)]^{\check{\alpha}_i} \; e^{\frac{\Lambda_i(z)}{g_i(z)}e_i}, \qquad
g_i(z) \in \C(z)^\times.
\end{equation}
\end{Thm}

Now we introduce the notion of nondegeneracy of Miura ($SL(r+1),\hbar)$-oper. We refer to \cite{Frenkel:2020} for a geometric interpretation of this condition and for its geometric origin.

\begin{Def}    \label{nondeg}
  A Miura $(SL(r+1),\hbar)$-oper $A(z)$ of the form \eqref{form of A} is called
  nondegenerate if the corresponding $(H,\hbar)$-connection
  $A^H(z)$ can be written in the form \eqref{AH1}, where 
\begin{itemize}
\item $y_i(z)$ are polynomial

\item 
  for all
  $i,j,k$ with $i\ne j$ and $a_{ik} \neq 0, a_{jk} \neq
    0$, the zeros and poles of $y_i(z)$ and $y_j(z)$ are
  $\hbar$-distinct from each other and from the zeros of
  $\Lambda_k(z)$.
  \end{itemize} 
\end{Def}

If we apply a $\hbar$-gauge transformation by an element $h(z)\in H[z]$ to
$A(z)$, we get a new $Z$-twisted Miura $(SL(r+1),\hbar)$-oper.
The following proposition shows that it is only nondegenerate
if $h(z)\in H$.  As a consequence, the $\Lambda_k$'s of a
nondegenerate $\hbar$-oper are determined up to scalar multiples.

\begin{Prop} \cite{Frenkel:2020} If $A(z)$ is a nondegenerate $Z$-twisted Miura  \\
$(SL(r+1),\hbar)$-oper and $h(z)\in H[z]$, then $h(\hbar z)A(z)h(z)^{-1}$ is
  nondegenerate if and only if $h(z)$ is a constant element of $H$.
\end{Prop}

\subsection{$QQ$-systems and Miura $\hbar$-opers}

In the previous section, we found the explicit structure of the $Z$-twisted non-degenerate  Miura $(SL(r+1),\hbar)$-oper   
with regular singularities defined by $\{\Lambda_i(z)\}_{i=1,\dots, r}$ and associated with regular element $Z=\prod_i\zeta_i^{\check{\alpha}_i}$.  The local expression, namely $A(z)$ can be expressed as follows:
\begin{equation}\label{form of A1}
A(z)=\prod_i
g_i(z)^{\check{\alpha}_i} \; e^{\frac{\Lambda_i(z)}{g_i(z)}e_i}, \qquad
g_i(z)=\zeta_i\frac{Q_i^+(\hbar z)}{Q_i^+(z)}\,.
\end{equation}
where $Q_i^{+}(z)$ are polynomials (here we changed the notation $y_i(z)\equiv Q_i^{+}(z)$). 
From now on, we will assume that $Z$ satisfies the following property:
\begin{equation}    \label{assume}
\prod_{i=1}^r \zeta_i^{a_{ij}}=\frac{\zeta_j^2}{\zeta_{j+1}\zeta_{j-1}} \notin \hbar^\Z, \qquad
\forall j=1,\ldots,r\,,
\end{equation}
where $a_{ij}$ are matrix elements of the Cartan matrix for $\mathfrak{sl}_{r+1}$.
Since $\prod_{i=1}^r \zeta_i^{a_{ij}}\ne 1$ is a special case of
\eqref{assume}, this implies that $Z$ is {\em regular semisimple}.

\subsection{The $SL(r+1)$ $QQ$-system}
In \cite{Frenkel:2020} (see also similar result in \cite{Mukhin_2005} for the XXX case) the following statement was proven (we specialize that result to the case of $SL(r+1)$):
\begin{Thm}    \label{inj}
  There is a one-to-one correspondence between the set of
  nondegenerate $Z$-twisted Miura $(SL(r+1),\hbar)$-opers and the set
  of nondegenerate polynomial solutions of the $QQ$-system 
\begin{equation}\label{eq:QQAtype}
\xi_{i} Q^+_i(\hbar z) Q^-_i(z) - \xi_{i+1} Q^+_i(z) Q^-_i(\hbar z) = \Lambda_i (z) Q^+_{i-1}(z)Q^+_{i+1}( \hbar z)\,, \qquad i = 1,\dots, r
\end{equation}
subject to the boundary conditions $Q^\pm_{0}(z)=Q^\pm_{r+1}(z)=1$ and $\xi_0=\xi_{r+2}=1$ so that  
$$
\xi_1=\zeta_1,\quad \xi_2= \frac{\zeta_2}{\zeta_1},\quad \dots \quad \xi_r=\frac{\zeta_{r}}{\zeta_{r-1}},\quad \xi_{r+1}= \frac{1}{\zeta_r}\,.
$$
\end{Thm}
Note, that $\xi_i$ is the $i$th element on the diagonal of $Z$ from \eqref{Z}.

We will say that a polynomial solution $\{ Q_i^+(z),Q_i^-(z)
\}_{i=1,\ldots,r}$ of \eqref{eq:QQAtype} is {\em nondegenerate} if the following conditions are satisfied: relation \eqref{assume} holds;
for $i\neq j$ the zeros of $Q^+_i(z)$ and $Q^-_{j}(z)$ are
$\hbar$-distinct from each other and from the zeros of $\Lambda_{k}(z)$ for $|i-k|=1,\,|j-k|=1$.

For the convenience we will rewrite \eqref{eq:QQAtype} as follows:
\begin{equation}\label{eq:QQgamma}
\xi_{i} \phi_i(z)- \xi_{i+1} \phi_i( \hbar z) = \rho_i(z)\,,
\end{equation}
where 
\begin{equation}
\phi_i(z)=  \frac{Q^-_i(z)}{Q^+_i(z)} \,,\qquad \rho_i(z)= \Lambda_i (z)\frac{Q^+_{i-1}( \hbar z)Q^+_{i+1}(z)}{Q^+_i(z)Q^+_i( \hbar z)} \,.
\end{equation}

\subsection{Extended $QQ$-system and $Z$-twisted $(SL(r+1),\hbar)$-opers}
As it was demonstrated in \cite{Frenkel:2020} for a simply-connected simple Lie group $G$ the set of nondegenerate $Z$-twisted Miura-Pl\"ucker $(G,\hbar)$-opers includes as a subset the set of $Z$-twisted Miura $(G,\hbar)$-opers.
The opposite inclusion was possible provided that $Z$-twisted Miura-Pl\"ucker $\hbar$-opers are in addition $w_0$-generic.

In this section we shall demonstrate that when $G$ is a special linear group then we do not need this extra condition and that the corresponding  $Z$-twisted Miura-Pl\"ucker $(SL(r+1),\hbar)$-oper will be $Z$-twisted Miura $(SL(r+1),\hbar)$-oper, namely there exists $v(z) \in B_+(z)$, such that the q-connection $A(z)$ reduces to an element of the form \eqref{Z}, or, equivalently
\begin{equation}\label{eq:qGaugeTrSpecial}
v( \hbar z)^{-1}A(z)= Zv(z)^{-1}\,.
\end{equation}
Moreover, we will construct an explicit expression for $v(z)$.

The following statement is a generalization of the result of \cite{Mukhin_2005} to $Z$-twisted $\hbar$-opers.

\begin{Thm}\label{th:SLNqMiura}
Let $A(z)$ be as in \eqref{form of A1} and $Z$ as in \eqref{Z}. 
Suppose $Q^-_{i,i+1,\dots, j}(z)$ ( $i,j\in\mathbb{Z}$, $i<j$) are polynomials, satisfying equations:
\begin{align}\label{eq:QQAll}
\xi_{i} \,\phi_i(z)-\xi_{i+1}\,  \phi_i(\hbar z)&= \rho_i(z)\,,\qquad \qquad &i=1,\dots,r\notag\\
\xi_{i} \,\phi_{i,i+1}(z)-\xi_{i+2}\,  \phi_{i, i+1}( \hbar z)&=\rho_{i+1}(z)\phi_i( \hbar z)\,,\qquad \qquad &i=1,\dots,r-1\notag\\
\dots&\dots\\
\xi_{i} \,\phi_{i,\dots,r-2+i}(z)-\xi_{r+i-1}\,  \phi_{i,\dots,r-2+i}( \hbar z)&=\rho_{r-1}(z)\phi_{i,\dots,r-3+i}( \hbar z)\,,\qquad \qquad &i=1,2\notag\\
\xi_{1} \phi_{1,\dots,r}(z)-\xi_{r+1}  \phi_{1,\dots,r}( \hbar z)&= \rho_r(z)\phi_{1,\dots,r-1}( \hbar z)\,,\qquad \qquad &\notag
\end{align}
where for all $j>i$
\begin{equation}\label{eq:phimdef}
\phi_{i,\dots,j}(z)=\frac{Q^-_{i,\dots,j}(z)}{Q^+_j(z)}.
\end{equation}
Then there exist $v(z)\in B_+(z)$ such that \eqref{eq:qGaugeTrSpecial} holds and is given by
\begin{equation}\label{eq:qGaugeGen}
v(z)= \prod\limits_{i=1}^r Q^+_i(z)^{\check{\alpha}_i} \cdot \prod\limits_{i=1}^r V_i(z)\,,
\end{equation}
where 
\begin{equation}
V_i(z)= \prod\limits_{j=i}^r \exp\left(-\phi_{i,\dots,j}(z)\, e_{i,\dots,j}\right)\,, \quad e_{i,\dots,j}=[\dots[[e_i,e_{i+1}],e_{i+2}]\dots e_j]\,.
\end{equation}
\end{Thm}

Notice that although the expression for $v(z)$ in \eqref{eq:qGaugeGen} is rather complicated, the inverse $v(z)^{-1}$ can be succinctly presented as
\begin{equation}\label{eq:vinverse}
v(z)^{-1}=\displaystyle\begin{pmatrix}
\frac{1}{Q_1^+(z)} & \frac{Q^-_1(z)}{Q_2^+(z)} & \frac{Q^-_{12}(z)}{Q_3^+(z)} & \dots  & \frac{Q^-_{1,\dots,r-1}(z)}{Q_r^+(z)}  & Q^-_{1,\dots,r}(z) \\
0 & \frac{Q^+_1(z)}{Q_2^+(z)} & \frac{Q^-_2(z)}{Q_3^+(z)} & \dots & \frac{Q^-_{2,\dots,r-1}(z)}{Q_r^+(z)}  & Q^-_{2,\dots,r}(z)\\
0 & 0 & \frac{Q^+_2(z)}{Q_3^+(z)}  & \dots & \frac{Q^-_{3,\dots,r-1}(z)}{Q_r^+(z)} & Q^-_{3,\dots,r}(z)\\
\vdots & \vdots & \vdots & \ddots & \vdots  & \vdots\\
0 & \dots & \dots& \dots &  \frac{Q^+_{r-1}(z)}{Q_r^+(z)} & Q^-_r(z) \\
0 & \dots & \dots& \dots &  0 & Q^+_r(z)
\end{pmatrix}\,.
\end{equation}

Using this Lemma we can rewrite the q-connection \eqref{form of A1} such that the roots of $SL(r+1)$ are placed in the decreasing order. 

\begin{Lem}\label{Th:BCHAConnection}
Let
\begin{equation}
\rho_i(z)=\Lambda_i(z)\frac{Q_{i-1}(\hbar z)Q_{i+1}(z)}{Q_{i}(\hbar z)Q_{i}(z)}\,.
\end{equation}
Then the $(SL(r+1),\hbar)$-oper reads 
\begin{equation}\label{new form of A1}
A(z)=\prod_{i=r}^1 Q^+_i(\hbar z)^{\check{\alpha}_i} \cdot \prod\limits_{i=r}^1 e^{\frac{\zeta_{i}}{\zeta_{i+1}}\rho_i(z)e_i} \cdot \prod_{i=r}^1 \zeta_i^{\check{\alpha}}Q^+_i(z)^{-\check{\alpha}_i}\,,
\end{equation}
or as a matrix
\begin{equation}\label{eq:MiuraqConnection}
A(z)=
\begin{pmatrix}
g_1(z) & \Lambda_1(z) & 0 & 0 & \dots &0& 0\\
0& \frac{g_2(z)}{g_1(z)} & \Lambda_2(z)  & 0 & \dots & 0&0\\
0& 0&  \frac{g_3(z)}{g_2(z)} & \Lambda_3(z)  & \dots & 0& 0\\
\vdots& \vdots & \cdots & \ddots & \ddots & \vdots & \vdots \\
\vdots& \vdots & \cdots & \dots & \ddots & \Lambda_{r-1}(z) & 0 \\
0& 0& 0& \dots &\dots &\frac{g_r(z)}{g_{r-1}(z)} &\Lambda_r(z)\\
0& 0& 0& \dots &\dots&0 &\frac{1}{g_r(z)}
\end{pmatrix}
\end{equation}
\end{Lem}

The first line of \eqref{eq:QQAll} is the $SL(r+1)$ $QQ$-system \eqref{eq:QQgamma}. 
Let us rewrite the above equations in terms of the Q-polynomials:
\begin{align}\label{eq:QQALLFull}
\xi_{i} Q^+_i(\hbar z) Q^-_i(z) - \xi_{i+1} Q^+_i(z) Q^-_i(\hbar z) &=\Lambda_i (z) Q^+_{i-1}(\hbar z)Q^+_{i+1}(z)\,,\notag\\
\xi_{i} Q^+_{i+1}(\hbar  z) Q^-_{i,i+1}(z) - \xi_{i+2} Q^+_{i+1}(z) Q^-_{i,i+1}(\hbar z) &= \Lambda_{i+1} (z) Q^-_{i}(\hbar z)Q^+_{i+2}(z)\,,\notag\\
\dots&\dots\\
\xi_i \,Q^+_{r-2+i}(\hbar z)Q^-_{i,\dots,r-2+i}(z)-\xi_{r-1+i}\,  Q^+_{r-2+i}(z)Q^-_{i,\dots,r-2+i}(\hbar z)&=\Lambda_{r-1+i}(z)Q^-_{i,\dots,r-1+i}(\hbar z)Q^+_{r+i}(z)\,,\notag\\
\xi_{1} Q^+_r(\hbar z)Q^-_{1,\dots,r}(z)-\xi_{r+1}  Q^+_r(z)Q^-_{1,\dots,r}(\hbar z)&= \Lambda_r(z) Q^-_{1,\dots,r-1}(\hbar z)\,.\notag
\end{align}
We shall refer to all equations of \eqref{eq:QQALLFull} as the \textit{extended $QQ$-system} for $SL(r+1)$. We call its solution {\it nondegenerate}, if the resulting solution of the original $QQ$-system is nondegenerate.

Let us now show that starting from the solution of the nondegenerate $QQ$-system, we obtain solutions to the extended $QQ$-system as well. To do that we need the result (which is true for other simply laced groups) of \cite{Frenkel:2020}: 

\begin{Thm}
The solutions of the nondegenerate $SL(r+1)$ $QQ$-system are in one-to-one correspondence to the solutions of the Bethe Ansatz equations for $\mathfrak{sl}(r+1)$ XXZ spin chain:
\begin{equation}    \label{eq:bethe}
\frac{Q^+_{i}(\hbar s_{i,k})}{Q^+_{i}(\hbar^{-1}w^k_i)} \frac{\xi_i}{\xi_{i+1}}=
- \frac{\Lambda_i(s_{i,k}) Q^{+}_{i+1}(\hbar s_{i,k})Q^{+}_{i-1}(s_{i,k})}{\Lambda_i(\hbar^{-1}s_{i,k})Q^{+}_{i+1}(s_{i,k})Q^{+}_{i-1}(\hbar^{-1}s_{i,k})},
\end{equation}
where $Q_i=\prod_{k=1}^{m_i}(z-s_{i,k}),\, \Lambda_i(z)=\prod_{k=1}^{l_i}(z-a_{i,k}), \,i=1,\ldots,r$.
\end{Thm} 

This Theorem can be extended as follows.

\begin{Thm}\label{Th:BetheQQEquiv}
There is a one-to-one correspondence between the set of nondegenerate solutions of the extended $QQ$-system \eqref{eq:QQALLFull}, the set of nondegenerate solutions of the $QQ$-system \eqref{eq:QQAtype}, and the set of  solutions of Bethe Ansatz equations \eqref{eq:bethe}.
\label{Th:LemmaNondeg}
\end{Thm}

\begin{Rem}
We note here that the equations (\ref{eq:bethe}) emerge as a critical condition on a Yang-Yang function:
\begin{equation}
\exp 2\pi \frac{\partial Y}{\partial \sigma_{i,k}}=1\,,
\end{equation}
where
\begin{align}
\label{eq:YAYangYang1}
Y(\{\sigma_{i,k}\},\{\mathrm{a}_i\}, \{x_i\}, \epsilon) &= \sum_{i=1}^r \left(\sum_{a=1}^{m_i} \sum_{b=1}^{l_i} \ell (\sigma_{i,a} - \mathrm{a}_{i,b}+\epsilon)+ \ell (-\mathrm{a}_{i,b}+\sigma_{i,a}+\epsilon)\right)\notag\\
&+\sum_{i=1}^r\left(\sum_{a=1}^{m_i}\sum_{b=1}^{m_{i+1}}\ell(\sigma_{i,a} - \sigma_{i+1,b}+\epsilon)+\ell(-\sigma_{i,a} + \sigma_{i+1,b}+\epsilon)\right)\notag\\
&+\sum_{i=1}^{r}\left( \sum_{a,b=1}^{m_i} \ell(\sigma_{i,a}-\sigma_{i,b} +\epsilon)+\ell(-\sigma_{i,a}+\sigma_{i,b} +\epsilon)+(x_{i+1}-x_{i}) \sum_{a=1}^k \sigma_{i,a}\right)\,,
\end{align}
where deg $\Lambda_i(z)=l_i$,
$$
s_{i,b} = e^{2\pi \sigma_{i,b}}\,,\quad a_{i,b} = e^{2\pi \mathrm{a}_{i,b}}\,,\quad \hbar= e^{2\pi \epsilon}\,,\quad \xi_i = e^{2\pi x_i}\,,
$$
and $\ell(u)$ be a multi-valued function, which can be written in terms of quantum dilogarithm Li$_2$ (see  \cite{Gaiotto:2013bwa}), such that 
$$
\exp 2\pi \frac{\partial \ell(u)}{\partial u} = 2\sinh \pi u\,.
$$

\end{Rem}

\vskip.1in

In \cite{Frenkel:2020} B\"acklund transformations were introduced for Miura $\hbar$-opers (see Proposition 7.1) and
were associated to the $i$-th simple reflection from the Weyl group:
\begin{Prop}    \label{fiter}
  Consider the $\hbar$-gauge transformation of the $q$-connection given by \eqref{form of A1}
  \begin{eqnarray}
A \mapsto A^{(i)}=e^{\mu_i(\hbar z)f_i}A(z)e^{-\mu_i(z)f_i},
\quad \operatorname{where} \quad \mu_i(z)=\frac{Q^+_{i-1}(z)Q^+_{i+1}(z)}{Q^i_{+}(z)Q^i_{-}(z)}\,.
\label{eq:PropDef}
\end{eqnarray}
Then $A^{(i)}(z)$ can be obtained from $A(z)$ by
substituting in formula \eqref{form of A1}
\begin{align}
Q^j_+(z) &\mapsto Q^j_+(z), \qquad j \neq i, \\
Q^i_+(z) &\mapsto Q^i_-(z), \quad \zeta_i\mapsto \displaystyle\frac{\zeta_{i-1}\zeta_{i+1}}{\zeta_i}\,\quad \,.
\label{eq:Aconnswapped}
\end{align}
\end{Prop}

It is possible that after the transformation the resulting operator gives rise to the nondegenerate $QQ$-system. 
Denoting the the  $QQ$-system after the B\"acklund transformation as $\{\widetilde{Q}^{\pm}_i\}_{i=1,\dots, r}$, we obtain:
\begin{align} \label{qqm2}
  \{ \wt{Q}^+_j \}_{j=1,\ldots,r} &= \{ Q_1^{+}, \dots, 
 Q_{i-1}^+,Q_i^-,Q_{i+1}^+ \dots , Q^r_{+} \}; \\ \notag
  \{ \wt{Q}^-_j \}_{j=1,\ldots,r} &=  \left\{ Q^-_1, \dots,
   {Q^*}_{i-1}^-,-Q_i^+,Q_{i,i+1}^- \dots , Q^-_{r} \right\}\,\\
  \{ \wt{\zeta}_j \}_{j=1,\ldots,r} &= \left\{
  \zeta_1,\dots,\zeta_{i-1},\frac{\zeta_{i-1}\zeta_{i+1}}{\zeta_i}
,\dots,\zeta_r\right\} \notag
\end{align}

The last line can be also rewritten in terms of $\xi$ variables as follows: 
$$
\{ \wt{\xi}_j \}_{j=1,\ldots,r} = \left\{ \xi_1,\dots,\xi_{i-1},\xi_{i+1}, \xi_i, \xi_{i+2}
,\dots,\xi_{r+1}\right\} 
$$

Here we placed $Q^-_{i,i+1}$ in the position $Q^-_{i+1}$, since the equation this new polynomial satisfies, is the second one from the extended $QQ$-system. At the same time, the new polynomial ${Q^*}^-_{i-1}(z)$ does not belong to what we called the extended $QQ$-system.

\subsection{Line bundles and Wronskians}

In this subsection we describe Z-twisted Miura $\hbar$-opers with regular singularities, following \cite{KSZ}.  
Namely, we have a complete flag of subbundles $\cL_\bullet$ such that $\hbar$-connection $A$ maps $\cL_{i}$ 
  into $\cL_{i-1}^\hbar$ and the induced maps
  $\bar{A}_i:\cL_{i}/\cL_{i+1}\to \cL^\hbar_{i-1}/\cL^\hbar_{i}$ are
  isomorphisms for $i=1,\dots,r$ on  $U \cap M_\hbar^{-1}(U)$, where $U$ is the Zariski open dense subset.
Explicitly, considering the determinants  
\begin{equation}\label{altqW} 
 \left(\Big(\prod_{j=0}^{i-2}(A(\hbar^{i-2-j}z)\Big)s(z)\wedge\dots\wedge A(\hbar^{i-2}z) s(\hbar^{i-2}z)\wedge s(\hbar^{i-1}z)
  \right)\bigg|_{\Lambda^i\cL_{r-i+2}^{\hbar^{i-1}}}
\end{equation}
for $i=1,\dots, r+1$,  
where $s$ is a local section of $\cL_{r+1}$. From the definition we know that  
$(E,A,\cL_\bullet)$ is an $(SL(r+1),\hbar)$-oper if and only if at every point of $U \cap M_\hbar^{-1}(U)$, there 
exists local section for which each  such determinant is nonzero (see \cite{KSZ}).   
When we encounter the case of regular singularities, each  
$\bar{A}_i$ is an isomorphism except at zeroes of $\Lambda_i$ and thus we require the determinants to vanish at zeroes of the following polynomial $W_k(s)$: 
\begin{eqnarray}\label{eq:WPDefs}
W_k(s)=P_1(z) \cdot P_2(\hbar^2z)\cdots P_{k}(\hbar^{k-1}z),  \qquad 
P_i(z)=\Lambda_{r}\Lambda_{r-1}\cdots\Lambda_{r-i+1}(z)\,.
\end{eqnarray}  

Recall that the Miura condition implies that there exists a flag $\hcL_\bullet$ which is preserved by the $\hbar$-connection $A$. The $Z$-twisted condition implies that in the gauge when $A$ is given by fixed semisimple diagonal element $Z\in H$ such flag is formed by the standard basis $e_1, \dots, e_{r+1}$. 

The relative position between two flags is generic on $U \cap M_\hbar^{-1}(U)$. The regular singularity condition implies that {\it quantum Wronskians}, namely  determinants
\begin{equation}\label{qD}
\mathcal{D}_k(s)=e_1\wedge\dots\wedge{e_{r+1-k}}\wedge
Z^{k-1}s(z)\wedge Z^{k-2} s(\hbar z)\wedge\dots\wedge Z s(\hbar^{k-2})\wedge s(\hbar^{k-1}z)\,
\end{equation}
have a subset of zeroes, which coincide with those of
$\cW_k(s)$.  To be more explicit, for
$k=1,\dots,r+1$, we have nonzero constants $\alpha_k$ and polynomials
\begin{equation} \mathcal{V}_k(z) = \prod_{a=1}^{r_k}(z-v_{k,a})\,,
\label{eq:BaxterRho}
\end{equation}
for which 
\begin{equation}\label{eq:MiuraQOperCond}
\det\begin{pmatrix} \,     1 & \dots & 0 & \xi_1^{k-1}s_{1}(z) & \cdots & \xi_{1} s_{1}(\hbar^{k-2}z)  &  s_{1}(\hbar^{k-1}z) \\ 
 \vdots & \ddots & \vdots& \vdots & \vdots & \ddots & \vdots \\  
0 & \dots & 1&\xi_{k}^{k-1}s_{r+1-k}(z) &\dots & \xi_{k} s_{r+1-k}(\hbar^{k-2}z) &   s_{k}(\hbar^{k-1}z)  \\  
0 & \dots & 0&\xi^{k-1}_{k+1}s_{r+1-k+1}(z) & \dots & \xi_{r+1-k+1} s_{k+1}(\hbar^{k-2}z)  &  s_{k+1}(\hbar^{k-1}z)  \\
\vdots & \ddots & \vdots&\vdots & \vdots & \ddots & \vdots \\
0 & \dots & 0&\xi_{r+1}^{k-1}s_{r+1}(z) & \dots &\xi_{r+1} s_{r+1}(\hbar^{k-2}z)  & s_{r+1}(\hbar^{k-1}z)  \, \end{pmatrix} =\alpha_{k} W_{k}
\cV_{k} \,; 
\end{equation}
Since $\cD_{r+1}(s)=W_{r+1}(s)$, we have
$\cV_{r+1}=1$.  We also set $\cV_0=1$; this is consistent with the fact
that \eqref{qD} also makes sense for $k=0$, giving
$\cD_0=e_1\wedge\dots\wedge e_{r+1}$.

We can also rewrite \eqref{eq:MiuraQOperCond} as
\begin{equation}
  \underset{i,j}{\det} \left[\xi_{r+1-k+i}^{k-j} s_{r+1-k+i}(\hbar^{j-1}z)\right] = \alpha_{k} W_{k} \mathcal{V}_{k}\,,
\label{eq:MiuraDetForm}
\end{equation}
where $i,j = 1,\dots,k$. 

\begin{Thm}[\cite{KSZ}]\label{qWthKSZ}
Polynomials $\{\mathcal{V}_k(z)\}_{k=1,\dots, r}$ give the solution to the $QQ$-system \ref{eq:QQAtype} 
so that $Q^+_j(z)=\mathcal{V}_j(z)$ under the nondegeneracy condition that for all $i,j,k$ with $i \neq j$ and $a_{ik} \neq  0, a_{jk} \neq 0$, the zeros of $\mathcal{V}_i(z)$ and $\mathcal{V}_j(z)$ are
$q$-distinct from each other and from the zeros of $\Lambda_k(z)$.
\end{Thm}

The following theorem allows to relate the section $s(z)$, generating the line bundle $\mathcal{L}_{r+1}$ with the elements of the extended $QQ$-system  using the transformation \eqref{eq:qGaugeTrSpecial}.

\begin{Prop}\label{Th:SectionsQQ}
Let $v(z)$ be the  gauge transformation from \eqref{eq:qGaugeTrSpecial} and $s(z)$ be the section generating $\mathcal{L}_{r+1}$ in the definition of the $(SL(r+1),\hbar)$-oper. Then the components of $s(z)$ in the 
gauge when $(SL(r+1,\hbar)$-oper connection is equal to $Z$ is given by:
\begin{equation}\label{eq:comptssection}
s_{r+1}(z)=Q^+_{r}(z)\,,\qquad s_r(z)=Q^-_r(z)\,,\qquad s_k(z)=Q^-_{k,\dots,r}(z)\,,
\end{equation}
for $k=1,\dots, r-1$.
\end{Prop}

It follows from the direct application of \eqref{eq:vinverse}
Starting from \eqref{eq:vinverse} the Proposition follows after acting with $v(z)^{-1}$ on the basis vector $e_{r+1}=(0,0,\dots, 0, 1)$. 

One can show the extended $QQ$-system can be obtained from various minors in q-Wronskian matrices. 
First, we will rewrite the extended $QQ$-system in a more convenient way to relate it to the minors in the $\hbar$-Wronskian matrix. Namely, we multiply $Q$-terms by certain polynomials to get rid of the $\Lambda$-polynomials on the right-hand side. This is done in the following Lemma.
\begin{Lem}\label{Th:PropQQtilde}
The system of equations \eqref{eq:QQAll} is equivalent to the following set of equations
\begin{align}\label{eq:QQAll1}
\xi_{i} \,\mathscr{D}^+_{i}(\hbar z)\mathscr{D}^-_{i}(z)-\xi_{i+1}\, \mathscr{D}^+_{i} (z)\mathscr{D}^-_{i}( \hbar z)&= (\xi_{i}-\xi_{i+1})\,\mathscr{D}^+_{i-1} (\hbar z)\mathscr{D}^+_{i+1}(z)\,,\notag\\
\xi_{i} \,\mathscr{D}^+_{i+1}(\hbar z)\mathscr{D}^-_{i,i+1}(z)-\xi_{i+2}\, \mathscr{D}^+_{i+1} (z)\mathscr{D}^-_{i,i+1}(  \hbar z)&= (\xi_{i}-\xi_{i+2})\,\mathscr{D}^-_{i} ( \hbar z)\mathscr{D}^+_{i+2}(z)\,,\notag\\
\dots&\dots\\
\xi_i \,\mathscr{D}^+_{r+i-2}(\hbar z)\mathscr{D}^-_{i,\dots,r-1+i}(z)-\xi_{r+i-1},  \mathscr{D}^+_{r+i-2}(z)\mathscr{D}^-_{i,\dots,r-1+i}(\hbar  z)&=(\xi_i-\xi_{r+i-1})\mathscr{D}^-_{i,\dots,r-2+i}( \hbar z)\mathscr{D}^+_{r+i-1}(z)\,,\notag\\
\xi_{1} \mathscr{D}^+_{r}(\hbar z)\mathscr{D}^-_{1,\dots,r}(z)-\xi_{r+1} \mathscr{D}^+_{r}(z)\mathscr{D}^-_{1,\dots,r}(\hbar z)&= (\xi_{1}-\xi_{r+1})\mathscr{D}^-_{1,\dots,r-1}( \hbar z)\,.\notag
\end{align}
where index $i$ ranges between the same values as in the corresponding equations in \eqref{eq:QQAll}, for the polynomials 
\begin{equation}\label{eq:Dkdef}
\mathscr{D}^+_k={Q^+_k}{F_k}\,,\quad \mathscr{D}^-_k= Q^-_k {F_k}\hbar^{\half}_k,\,\qquad \mathscr{D}^-_{l,\dots, k}= Q^-_{l,\dots, k} {F_k}\hbar^{\half}_{l,\dots, k}\,.
\end{equation}
where 
$$
F_i(z)=  W_{r-i}(\hbar^{r-i}z)\,,\qquad \hbar^{\half}_{l,\dots, i} = \prod_{a=0}^{i-l}(\xi_l-\xi_{l+a+1})\,.
$$
\end{Lem}

For the future, we shall refer to \eqref{eq:QQAll1} as the \textit{extended} $\mathscr{D}\mathscr{D}$-system for $SL(r+1)$ and to its first line specifically as merely the $\mathscr{D}\mathscr{D}$-system.

As we shall see below, one can express the solutions of the $QQ$- and $\mathscr{D}\mathscr{D}$-systems in terms of section $s(z)$ of subbundle  $\mathcal{L}_{r+1}$.
Following the discussion of \cite{KSZ} (Section 4) we consider the following matrices: 
\begin{equation}
M_{i_1,\dots, i_j}=
\begin{pmatrix} \,  
\xi_{i_1}^{j-1}s_{i_1}(z)  & \cdots  & \xi_{i_1}s_{i_1}(\hbar^{j-2}z)&  s_{i_1}(\hbar^{j-1}z) \\ \vdots & \ddots & \vdots & \vdots \\   \xi_{i_j}^{j-1}s_{i_j}(z)  & \cdots& \xi_{i_j} s_{i_j}(\hbar^{j-2}z)  & s_{i_j}(\hbar^{j-1}z) \, 
\end{pmatrix}\\, \quad V_{i_1,\dots, i_j}=
\begin{pmatrix} 
\,   \xi_{i_1}^{j-1}&\cdots & \xi_{i_1} &  1  \\ \vdots & \ddots & \vdots & \vdots \\  \xi_{i_j}^{j-1} & \cdots & \xi_{i_j}   & 1  \, 
\end{pmatrix}\\,
\label{eq:MiuraQOperCond1q}
\end{equation}
where $s_i$ are polynomials and $V_{i_1,\dots,i_j}$ is the Vandermonde-like matrix whose determinant is
\begin{equation}
\text{det} V_{i_1,\dots, i_j} = \prod_{i<j}(\xi_i-\xi_j)\,.
\end{equation}

The extended $\mathscr{D}\mathscr{D}$-system can be recovered in the following way.

\begin{Prop}\label{eq:qWronskiansShift2}
There exist unique polynomials
$s_1,\dots,s_{r+1}$ such that 
the polynomials $\mathscr{D}^-_{i,\dots i+k}$ from \eqref{eq:QQAll1} read
\begin{equation}\label{eq:Dmfomrulae}
\mathscr{D}^-_{i,\dots,i+k} = \frac{\det\, M_{i,i+k+2\dots,r+1}}{\det\, V_{i,i+k+2\dots,r+1}}\,,\qquad i=1,\dots, r\,,\quad k=0,\dots, r-i\,.
\end{equation}
In particular,
\begin{equation}
\mathscr{D}^+_i (z)= \frac{\det\, M_{r+2-i,\dots,r+1}(z)}{\det\,
  V_{r+2-i,\dots,r+1}}\,\qquad\text{ and } \qquad
\mathscr{D}^-_i(z) = \frac{\det\, M_{r+1-i,r+3-i,\dots,r+1}(z)}{\det\, V_{r+1-i,r+3-i,\dots,r+1}}\,,
\label{eq:MM0Inda}
\end{equation}
where matrix $M$, $V$ are given in \eqref{eq:MiuraQOperCond1q}.
\end{Prop}

The proof of Proposition \ref{eq:qWronskiansShift2} partly based on Lemma 4.4 from \cite{KSZ} which ensures that polynomials $\mathscr{D}^\pm_i$ and $Q^\pm_i$ can be written as quantum Wronskians. Later in the paper we shall use this Lemma in the proof of 3d mirror symmetry for $X_{k,l}$.

\begin{Lem}\label{Th:existencePoly}
Suppose that $\gamma_1,\dots,\gamma_{k-1}$ are nonzero complex numbers
such that $\gamma_j\notin
  q^{\mathbb{N}_0}\gamma_{k}$ for $j<k$.  Let $f_1,\dots,f_{k-1}$ be
  polynomials that do not vanish at $0$,
  and let $g$ be an arbitrary polynomial.
   Then there exist unique polynomials  $f_1,\dots,f_k$ satisfying 
\begin{equation}\label{Fmatrix}
g=\det\begin{pmatrix} f_1 & \gamma_{1} f_1^{(1)} & \cdots &
  \gamma_{1}^{k-1} f_1^{(k-1)} \\ \vdots & \vdots & \ddots & \vdots \\
  f_k  & \gamma_{k} f_k^{(1)} & \cdots & \gamma_k^{k-1} f_k^{(k-1)} \,
\end{pmatrix}.  
\end{equation}
\end{Lem}

\section{$A_r$-quivers, $QQ$-systems and magnetic frame}\label{Sec:QQA}

\subsection{Quiver data and Miura $\hbar$-opers.}
For us, the framed $A_r$ quiver means the graph of $A_r$ type with vertices labeled by natural numbers 
To each vertex, we may attach a box labeled by a natural number as well. We will refer to the boxes as the framing of $A_n$ graph.
\begin{figure}[!h]
\includegraphics[scale=0.45]{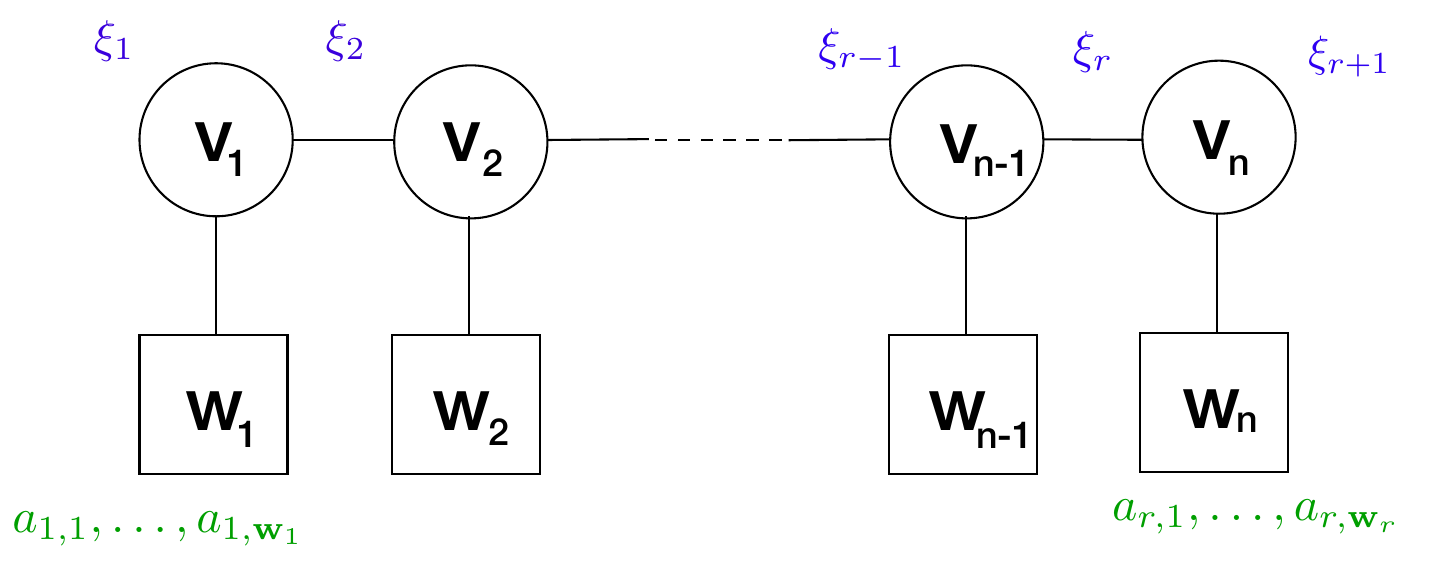}
\caption{Generic $A_r$ quiver variety}
\label{fig:GenAQuiv}
\end{figure}

Thus we see that such quiver is entirely defined by two vectors with the components ${\bf v}=({\bf v}_1, \dots, {\bf v}_r)$ and ${\bf w}=({\bf w}_1, \dots, {\bf w}_r)$, so that $\{\bf v_i\}, \{\bf w_i\}\in \mathbb{N}$, $i=1,\dots, r$. 
In the following, we will refer to quiver with this data as $Y_{{\bf v}, {\bf w}}$ and to ${\bf w_i}$ as the rank of framing of the $i$-th vertex. 
We associate a $QQ$-system to such quiver in the following way. We assign to each vertex $i$ with the label ${\bf v}_i$ (we count vertices from left to right)  the $Q^+_i(z)$ -polynomial of degree ${\bf v}_i$. At the same time, we associate the polynomial $\Lambda_i(z)$ of degree ${\bf w}_i$ to each vertex with the framing of rank ${\bf w}_i$.  

We will refer to the resulting space of $Z$-twisted nondegenerate Miura $(SL(r+1),\hbar)$-opers, associated with such $QQ$-systems and thus entirely defined by quiver as 
$\hbar{\rm{Op}}(Y_{{\bf v}, {\bf w}})$. Such opers are defined by the position of regular singularities, i.e. roots $\{a_{k,j}\}_{j=1,\dots,r, ~k=1, \dots, {\rm deg}(\Lambda_j)}$ of 
${\Lambda_j(z)}_{j=1,\dots, r}$, monic polynomials $\{Q^+_j(z)\}_{j=1,\dots, r}$ defined by their Bethe roots $\{s_{k,j}\}_{j=1,\dots,r}^{k=1, \dots, {\rm deg}(Q^+_j)}$, and $\{\xi_k\}_{k=1, \dots, r+1}$ parametrizing the $Z$-twist. 

We will refer to the following algebra as the algebra of functions on the space $\hbar{\rm{Op}}(Y_{{\bf v}, {\bf w}})$:
\begin{equation}
{\rm Fun}(\hbar{\rm{Op}})(Y_{{\bf v}, {\bf w}}):= \frac{S(\{a_i\},\{\xi_k\}, \hbar)(\{s_{i,k}\})}{\rm{(Bethe~ equations)}}\,,
\end{equation}
i.e. rational functions (with coefficients being rational functions of $\{a_i\},\{\xi_k\}, \hbar$) of the elementary symmetric functions of Bethe root variables (with symmetrization is over index $k$ for all $i$) with the relations on variables are given by Bethe equations from (\ref{eq:bethe}).

\begin{Rem}\label{Rem:nonlocal} 
From a perspective of the integrable XXZ spin chains, this is known as Bethe algebra represented in terms of nonlocal Hamiltonians (coefficients of expansion of the Baxter Q-operators, with coefficients being elementary symmetric functions of Bethe roots), considered in the weight sector determined by the data of the quiver $Y_{\mathbf{v}, \mathbf{w}}$. Later we will see other expressions for Bethe algebra using another set of Hamiltonians.
\end{Rem}
The following Proposition is a crucial property of quantum B\"acklund transformations, which directly follows from Proposition \ref{fiter}:
\begin{Prop}
Conisider the quiver $Y_{{\bf v}, {\bf w}}$ and the related $Z$-twisted Miura $(SL(r+1), \hbar)$-oper. Consider the quiver $Y_{{\bf v'}, {\bf w}}$ which corresponds to the $Z$-twisted Miura $(SL(r+1), \hbar)$-oper after the series of B\"acklund transformations.
Then ${\rm Fun}(\hbar{\rm{Op}})(Y_{{\bf v}, {\bf w}})\cong{\rm Fun}(\hbar{\rm{Op}})(Y_{{\bf v'}, {\bf w}})$, where the isomorphism is given by the transformations of the $QQ$-system are given in the Proposition \ref{fiter}.
\end{Prop}

At the same time, there is another presentation of the Miura $(SL(r+1),\hbar)$-opers, namely using the section 
$s(z)=(s_1(z), s_2(z), \dots, s_{r+1}(z))$, where  $s_i(z)=\tilde{c}_i\prod_k(z-p^k_i)$, where we labeled the roots of $s_i(z)$ 
as $p_i^k$. As we know from the definition of $s_i$, their degrees are determined from the data on the quiver.  

Let us introduce the algebra 
$$
{\rm Wr}(Y_{{\bf v}, {\bf w}}):=\frac{S(\{a_i\},\{\xi_k\}, \hbar)(\{p_i^k\})}{\rm{(Wronskian~ relations)}}\,,
$$ 
where $S(\{a_i\},\{\xi_k\}, \hbar)(\{p_i^k\})$ stand for the rational functions of the elementary symmetric functions of $\{p_i^k\}$ (with symmetrization over index $k$) so that the coefficients are rational functions of $\{a_i\},\{\xi_k\}, \hbar$. 
The Wronskian relations are given in \eqref{eq:MiuraQOperCond}. 
We will refer to the variables $\{p_i^k\}$ and $\{\xi^k_i\equiv \xi_i\}$ as {\it momenta} and {\it coordinates} correspondingly. 

Thus we have a theorem, which is a consequence of the results in Section 2.

\begin{Thm}\label{Wronskian}
There is an isomorphism of algebras:
$${\rm Fun}(\hbar{\rm{Op}})(Y_{{\bf v}, {\bf w}})={\rm Wr}(Y_{{\bf v}, {\bf w}}).$$
\end{Thm}

Altogether, we call such represenation of ${\rm Fun}(\hbar {\rm Op})(Y_{{\bf v}, {\bf w}})$ as {\it oper magnetic frame}.

\vskip.1in

Let us count the degrees of the components of the section of the $\hbar$-oper line bundle section $s(z)$ from \eqref{eq:comptssection}, which describes quiver in \figref{fig:GenAQuiv} using the extended $QQ$-system and B\"acklund transformations.

\begin{Lem}\label{Th:DegreesQ}
Let the degree of $s_i(z)$ in \eqref{eq:comptssection} be equal to $\rho_i$ for $i=1,\dots, r$. Then 
\begin{equation}\label{eq:weights0}
\rho_{i}-\rho_{i+1} = {\bf w}_i + {\bf v}_{i-1} + {\bf v}_{i+1} - 2 {\bf v}_i\,,\quad i=1,\dots,r-1
\end{equation}
with $\rho_0=\rho_{r+2}=0$ and $\rho_{r+1}={\bf v}_r$.
\end{Lem}

\begin{proof}
The proof can be performed using explicit formulae for the $\mathscr{D}^{\pm}$ and $Q^\pm$ polynomials \eqref{eq:Dmfomrulae} to show that 
\begin{equation}
\rho_i = {\bf v}_{i-1}-{\bf v}_{i} + \sum_{a=i}^{r} {\bf w}_{a}\,, \qquad i =1,\dots, r+1\,.
\end{equation}
However, let us proceed combinatorially instead. Since the extended $QQ$-system \eqref{eq:QQALLFull} is nondegenerate, in order to calculate the degrees of the components of the section \eqref{eq:comptssection}, we simply calculate them from the $QQ$-system in an iterative manner.

Let deg $Q^+_i={\bf v}_i$ and deg $\Lambda_i={\bf w}_i$. Recall that $Q^+_r=s_{r+1}$ and $\rho_{r+1}={\bf v}_r$.
Starting from the extended $QQ$-system \eqref{eq:QQALLFull} we perform B\"acklund transformation on the $r$th node
$$
\xi_{r} Q^+_r(\hbar z) Q^-_r(z) - \xi_{r+1} Q^+_r(z) Q^-_r(\hbar z) =\Lambda_r (z) Q^+_{r-1}(\hbar z)\,.
$$
In order to calculate the degree of $Q^-_r=s_{r}$ which is equal to $\rho_r = {\bf w}_r + {\bf v}_{r-1}  -  {\bf v}_r$. This is consistent with \eqref{eq:weights0}. 

Then we need to calculate the degrees of $Q^-_{r-1,r}$ and others and verify that they satisfy \eqref{eq:weights0}. This is, however, a daunting combinatorial process and there is plethora of cancellations of degrees. In order to navigate this process consider the diagram from \figref{fig:CombProof}.

\begin{figure}[!h]\label{fig:CombProof}
\includegraphics[scale=.4]{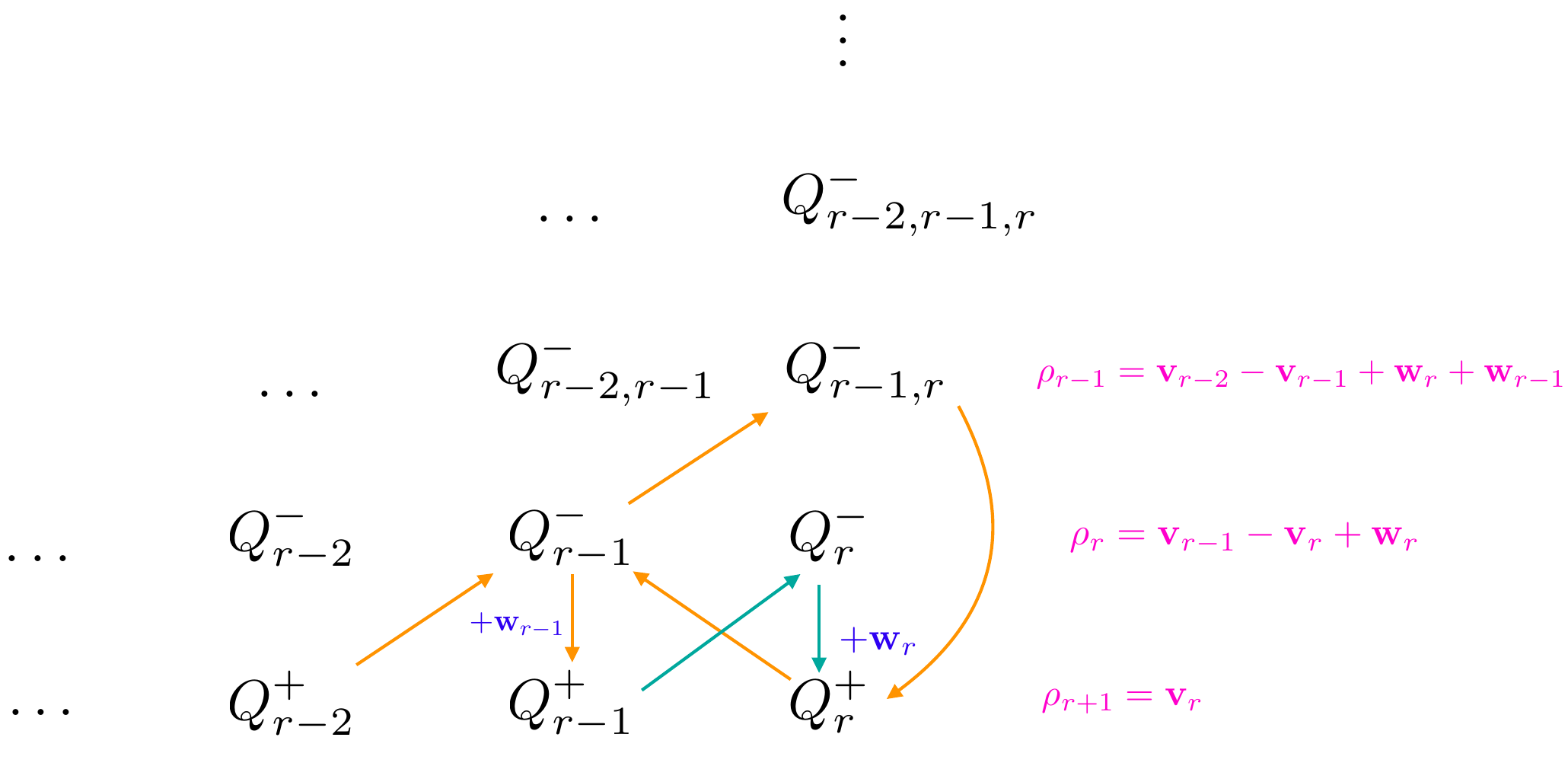}
\caption{Combinatorial proof of Lemma \ref{Th:DegreesQ}.}
\end{figure}

In this figure for each Q-polynomials there are incoming and outgoing arrows which respectively designate adding or subtracting degrees of the polynomials from/to which these arrows connect. For every down arrow we add the degree of the framing ${\bf w}_i$.
Note that there is a closed loop with orange arrows which suggests that the corresponding degrees will cancel. Indeed, from the corresponding $QQ$-relation 
$$
\xi_{r-1} Q^+_{r}(\hbar z) Q^-_{r-1,r}(z) - \xi_{r+1} Q^+_{r}(z) Q^-_{r-1,r}(\hbar z) =\Lambda_r (z) Q^-_{r-1}(\hbar z)\,.
$$
we get
$$
\rho_{r-1}=\text{deg}\, Q^-_{r-1,r} = \text{deg} \,Q^-_{r-1} +  {\bf w}_r - {\bf v}_r={\bf v}_{r-2}+ {\bf v}_{r} - {\bf v}_{r-1} + {\bf w}_{r-1} +  {\bf w}_r - {\bf v}_r
$$
$$
= {\bf v}_{r-2} - {\bf v}_{r-1}+{\bf w}_{r-1} +  {\bf w}_r\,.
$$

We can proceed by induction to derive the remaining degrees $\rho_i$ which satisfy \eqref{eq:weights0}.

\end{proof}

\subsection{Specific families of quivers and their $QQ$-systems}

In this section, we will describe families of quivers which will be the main examples in this paper. 

First, we give an example of what we call a complete flag quiver family. It has only one framing on the last vertex and it is labeled as shown in the picture:
\begin{figure}[!h]
\includegraphics[scale=0.36]{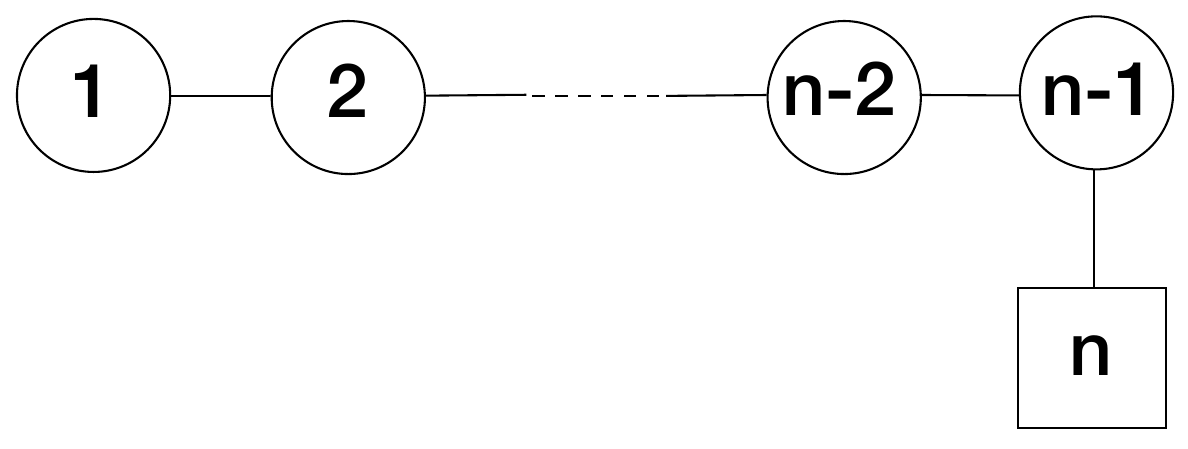}
\end{figure}

We will denote the regular singularities for the corresponding quiver, i.e. the roots of the $\Lambda_L$ polynomial as $a_1, \dots, a_L$.

In \cite{Koroteev:2017aa} (see also \cite{KSZ}) we proved that in this case the polynomials $s_i(z)$ are of degree 1.

Below we will give a more involved family of $A_r-quivers$, which we refer to as $X_{k,l}$, so that $r=l+k-1$ with the same property. 

\begin{figure}[!h]
\includegraphics[scale=0.47]{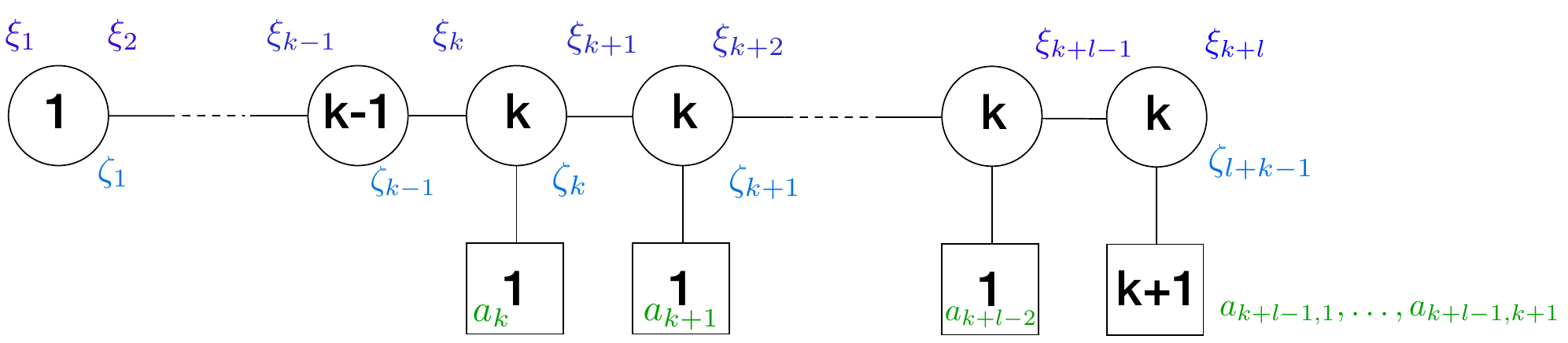}
\caption{Self-mirror $A_{k+l-1}$ quiver variety.}
\label{fig:QuiverABNew}
\end{figure}

The regular singularity parameters $a_1, \dots, a_{k+l}$ are labeled as it is shown in the picture. 

Notice, that in the $l\rightarrow \infty$ limit, we obtain the $A_{\infty}$ quiver (see Fig. \ref{fig:QuiverAinf}). From this standpoint, one can treat the quiver $X_{k,l}$ as the "regularized" version of it.


\begin{figure}[!h]
\includegraphics[scale=0.3]{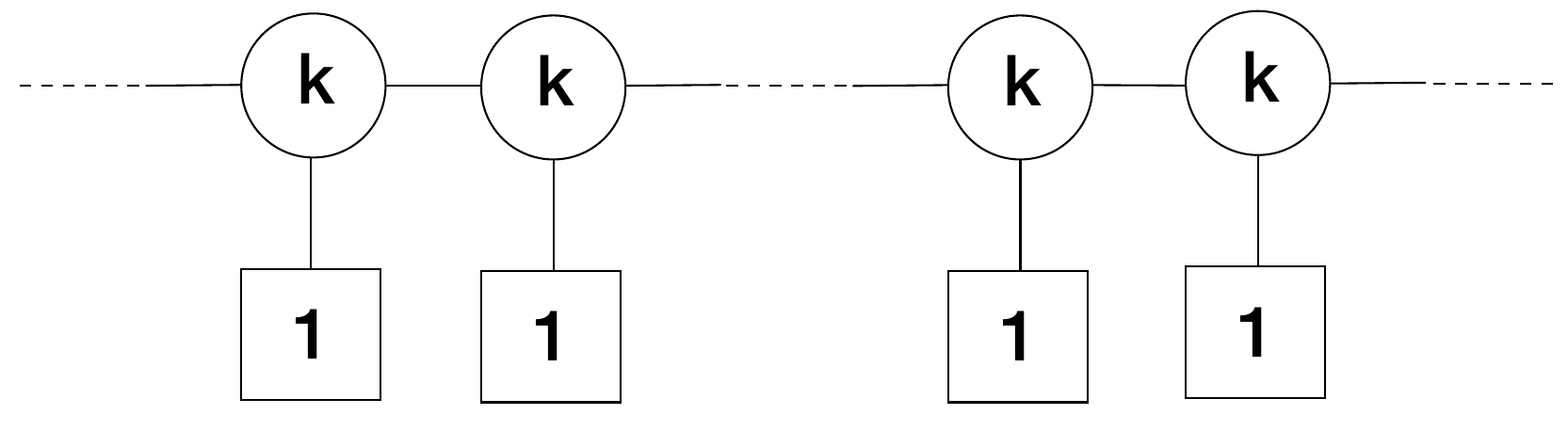}
\qquad\qquad  
\label{fig:QuiverAinf}
\end{figure}




\subsection{ $X_{k,l}$, $F\mathbb{F}l_L$  and the Ruijsennars-Schneider Hamiltonians}

Let us consider quiver $X_{k,l}$ (see bottom of \figref{fig:mirroranew}) of rank and study the corresponding extended $QQ$-system in detail. It is described by the following polynomials
\begin{align}
Q^+_1(z)&=\prod_{a=1}^k(z-s_{1,a}),\,\dots, \,Q^+_r(z)=\prod_{a=1}^k(z-s_{l,a}),\, \notag\\
Q^+_{l+1}(z)&=\prod_{a=1}^{k-1}(z-s_{l+1,a}),\,\dots,\, Q^+_{l+k-1}(z)=z-s_{l+k-1}\,,
\end{align}
and framing polynomials
\begin{equation}
\Lambda_1(z)=\prod_{b=1}^{k+1}(z-a_{1,b}),\, \qquad \Lambda_c(z)=z-a_c,\,\qquad c =1,\dots,l\,,
\end{equation}
All other polynomials are trivial.

The corresponding extended $QQ$-system  reads:
\begin{align}\label{eq:QQALLFull2}
\xi_{i} Q^+_i(\hbar z) Q^-_i(z) - \xi_{i+1} Q^+_i(z) Q^-_i(\hbar z) &=\Lambda_i (z) Q^+_{i-1}(\hbar z)Q^+_{i+1}(z)\,,\notag\\
\xi_{i} Q^+_{i+1}(\hbar z) Q^-_{i,i+1}(z) - \xi_{i+2} Q^+_{i+1}(z) Q^-_{i,i+1}(\hbar z) &= \Lambda_{i+1} (z) Q^-_{i}(\hbar z)Q^+_{i+2}(z)\,,\notag\\
\dots&\dots\\
\xi_i \,Q^+_{k+l-3+i}(\hbar z)Q^-_{i,\dots,k+l-3+i}(z)-&\xi_{k+l-2+i}\,  Q^+_{k+l-3+i}(z)Q^-_{i,\dots,k+r-3+i}(\hbar z)\,,\notag\\
&=\Lambda_{k+l-3+i}(z)Q^-_{i,\dots,k+l-2+i}(\hbar z)Q^+_{k+l-1+i}(z)\,,\notag\\
\xi_{1} Q^+_{k+l-1}(\hbar z)Q^-_{1,\dots,k+l-1}(z)-\xi_{k+l}  Q^+_{k+l-1}(z)Q^-_{1,\dots,k+l-1}(\hbar z)&= Q^-_{1,\dots,k+l-2}(\hbar z)\,.\notag
\end{align}

Let $s(z)$ be a section of line bundle $\mathcal{L}_{r+k}$ which in a vector with $r+k$ components as in \eqref{eq:MiuraDetForm}. As we have seen, the extended $QQ$-system contains the components of $s(z)$:
\begin{eqnarray}\label{eq:ComptsQQSection}
&&s_1(z)=Q^-_{1,2,\dots,k+l-1}(z)\,,\quad s_2(z)=  Q^-_{2,3,\dots, k+l-1}(z),\dots,\nonumber \\
&& s_{k+l-1}(z) = Q^-_{k+l-1}(z),\quad s_{k+l+1}(z)=Q^+_{k+l}(z).
\end{eqnarray}

The following Lemma characterizes the structure of the components $s_i(z)$ and their roots for $F\mathbb{F}l_L$.

\begin{Lem}
Given \eqref{eq:ComptsQQSection} we have
\begin{equation}
s_1(z)= z - p_1,\,\quad,\dots,\quad s_{k+r}(z)=z-p_{k+l}\,,
\end{equation}
where 
\begin{equation}\label{eq:piQQ0}
p_{k+l+1-p} = -\frac{Q^+_{p}(0)}{Q^+_{p-1}(0)}\,.
\end{equation}
\end{Lem}

\begin{proof}
From Lemma \ref{Th:DegreesQ} we conclude that all components of $s_p(z)$ are of the first order. Now we need to prove \eqref{eq:piQQ0}. Let us denote the matrix inside the brackets on the left-hand side of the expression \eqref{eq:MiuraDetForm} by $M_p(z)$ and substitute $z=0$. The following immediately follows
\begin{equation}
M_p(0)=-\text{diag}(p_{k+l+1-p},\dots,p_{k+l+1})\cdot  V_p\,,
\end{equation}
since $M_p(0)_{i,j}=-\xi^{p-j}_{k+l+1-p+i}p_{k+l+1-p+i}$ and $(V_p)_{i,j}=\xi^{p-j}_{k+l+1-p+i}$. Therefore 
\begin{equation}
-\text{det}\left[M_p(0)\cdot V_p^{-1}\right]=(-1)^pp_{k+l+1-p}\cdots p_{k+l+1}  \,.
\end{equation}
Thus \eqref{eq:piQQ0} follows if we define $\alpha_p=\frac{\text{det}V_p}{W_{p}(0)}$.
\end{proof}

For the $X_{k,l}$ the components of the section have the following degrees. This Lemma will be used in later sections in proving 3d mirror symmetry for $X_{k,l}$.

\begin{Cor}\label{Th:LemmaDegrees}
Given \eqref{eq:ComptsQQSection} quiver $X_{k,l}$ has the following degrees
\begin{equation}\label{eq:DegreesSectionXkl}
\text{deg}\, s_{l+1}= \dots = deg\, s_{l+k}=1,\qquad deg\, s_{l-i+1}(z)=i, \qquad i = 1,\dots,l\,.
\end{equation}

\end{Cor}

\begin{proof}
The statement directly follows from Lemma \ref{Th:DegreesQ}.

One can see from \figref{fig:QuiverABNew} that quiver $X_{k,l}$ contains a `tail' of $F\mathbb{F}l_k$. Using the previous Lemma one can verify that the last $k$ components of the section $s(z)$ of the oper bundle \eqref{eq:DegreesSectionXkl} are degree one polynomials. Moreoever, since the the rest of the quiver has different data the similarity with the complete flag ends and degrees of other components of the section grow linearly according to \eqref{eq:QQALLFull2}.
\end{proof}

Let us remind the reader that we gave a realization of the algebra ${\rm Fun}(\hbar{\rm{Op}}(Y_{{\bf v}, {\bf w}})$ as ${\rm Wr}(Y_{{\bf v}, ~{\bf w}})$. It turns out that in the case of $Y_{{\bf v}, {\bf w}}={F\mathbb{F}l}_L, X_{k,l}$ there is a more explicit version of that algebra in terms of the variables $\{p_i^k\}$.\\

\begin{Thm}\label{Th:qOpMag}
i) There is an isomorphism of algebras:
\begin{eqnarray}
{\rm Fun}(\hbar{\rm Op})({F\mathbb{F}l}_L))\cong \frac{\mathbb{C}(\{\xi_i\}, \{a_i\}, \{p_i\}, \hbar)}{\{H_i(\{p_j\}, \{\xi_j\}, \hbar)=e_i(a_1,\dots, a_{L}\})_{i=1,\dots, L}},
\end{eqnarray}
where 
\begin{equation}\label{eq:tRSRelationsEl}
H_k=\sum_{\substack{\mathcal{I}\subset\{1,\dots,L\} \\ |\mathcal{I}|=k}}\prod_{\substack{i\in\mathcal{I} \\ j\notin\mathcal{I}}}\frac{\xi_i - \hbar\,\xi_j }{\xi_i-\xi_j}\prod\limits_{m\in\mathcal{I}}p_m \,,
\end{equation}
and $\{e_i\}$ are the elementary symmetric functions of variables $\{a_i\}_{i=1,\dots, L}$.\\

\vskip.1in
ii) There is an isomorphism of algebras:
\begin{eqnarray}
&&{\rm Fun}(\hbar{\rm Op})(X_{k,l})\cong\\\nonumber
&&\frac{S(\{\xi_i\}, \{a_i\}, \hbar)(\{p_i\})}{
\left((B_i(\{p_j\}, \{\xi_j\}, \hbar)=\ell_i(a_1,\dots, a_{k+l}))_{i=1,\dots, k+l}, ~ {\rm low.~ order~ Wronsk. ~relations}\right)}\,, 
\end{eqnarray}
where $B_i$ are coefficients of the characteristic polynomial of the left-hand side of \eqref{eq:MiuraDetForm}
and $\ell_i$ are elementary symmetric polynomials of the following $k + l(l-1)/2$ variables:
\begin{align}\label{eq:rootsWXkl}
&a_{1,1},\dots, a_{1,k} \cr
&a_2, \cr
&a_3, \hbar a_3, \cr
&a_4, \hbar a_4, \hbar^2 a_4 \cr
&\dots\dots\\
&a_{l-1},\hbar a_{l-1}, \hbar^2 a_{l-1},\dots, \hbar^{l-2}a_{l-1}\,.\notag
\end{align}
For insatnce
\begin{eqnarray}
\label{eq:tRSXkrMagnetic}
&&\ell_{k+l}(a_1,\dots, a_{k+l})= \\\nonumber
&&\sum_{a=1}^{k+1}a_{1,a}+ a_2+(1+\hbar)a_3 + (1+\hbar+\hbar^2)a_4+\dots+\sum_{l=0}^{l-1}\hbar^{l-1} a_l\,.
\end{eqnarray}
\end{Thm}

\begin{proof}
Relations in i) are the standard tRS equations that follow from the oper condition \eqref{eq:MiuraDetForm}.

The proof of ii) follows from calculating the coefficients in front of powers of $z$ in the right hand side of \eqref{eq:MiuraDetForm}
applied to ${\rm Fun}(\hbar{\rm Op})(X_{k,l})$. The polynomial $W_{k+l}$ has the following form
\begin{equation}
W_{k+l}=\prod_{i=1}^k (z-a_{1,i}) \cdot (z-a_2)\cdot (z-a_3)(z-\hbar a_3)\cdot (z-a_4)(z-\hbar a_4)(z-\hbar^2 a_4)\cdots \prod_{j=0}^{l-1}(z-\hbar^j a_l)
\end{equation} 

\end{proof}

In the next section, we will explain in full detail the full meaning of the rational functions $H_k$. 
Here let us make an important remark concerning the geometric interpretation of the Theorem above.\\

Consider a symplectic space with coordinates 
$\{\chi_i\}, \{p_i\}$ with the symplectic form 
$\omega=\sum_i\frac{d\chi_i}{{ \chi}_i}\wedge \frac{dp_i}{p_i}$.
One can consider two Lagrangian subspaces. One is the natural one, which sends the coordinate variables to a fixed value $\{\chi_i=\xi_i\}$. The other one is defined by setting the functions $H_k$ to a fixed value, which in the theorem above is given by $\ell_i(\{a_i\}, \hbar)$. The Lagrangian property is justified by the fact that with respect to the Poisson bracket based on $\omega$, thus giving rise to an integrable system known as {\it trigonometric Ruijsenaars-Schneider (tRS) integrable system}.

\section{tRS system and  $QQ$-system in electric frame}\label{Sec:tRSElectric}
Let us describe the algebraic construction (see \cite{Oblomkov2004}) which is very important for this paper, which gives an algebraic formulation of the tRS system and the Lagrangian subvariety we introduced in the previous section.
\begin{Def}
Let V be an $N$-dimensional vector space over $\mathbb{C}$. Let $\mathcal{M}'$ be the subset of $GL(V)\times GL(V)\times V\times V^\ast$ consisting of elements $(M,T, u,v)$ such that
\begin{equation}\label{eq:FlatConNew}
\hbar M T -  T M = u\otimes v^T\,.
\end{equation}
The group $GL(N;\mathbb{C})=GL(V)$ acts on $\mathcal{M}'$ by conjugation
\begin{equation}\label{eq:SimilarityTRansf}
(X,Y, u,v)\mapsto (gMg^{-1},gTg^{-1}, gu,vg^{-1})\,,\qquad g\in GL(V)\,.
\end{equation}
The quotient of $\mathcal{M}'$ by the action of $GL(V)$ is called Calogero-Moser space $\mathcal{M}$.
\end{Def}

Note that when $\hbar$ is not a root of unity the $GL(V)$ action above is free so $\mathcal{M}'$ and $\mathcal{M}$ are nonsingular.

The following Lemma is true.
\begin{Lem}
Let $M$ and $T$ satisfy \eqref{eq:FlatConNew}. In the basis where $M$ is a diagonal matrix with eigenvalues $\chi_1,\dots, \chi_N$ the components of matrix $T$ are given by the following expression:
\begin{equation} \label{eq:lax1}
T_{ij} = \frac{u_i v_j}{\hbar \chi_i - \chi_j}\,.
\end{equation}
\end{Lem}

One can define the {\it tRS momenta} $p_i,\,i=1,\dots,N$  as follows:
\begin{equation}\label{eq:lax2}
p_i = - u_i v_i \frac{\prod\limits_{k \neq i} (\chi_i - \chi_k)}{\prod\limits_{k}\left(\chi_i -\chi_k \hbar\right)}\,.
\end{equation}

Using the above formula we can represent the components of matrix $T$ \eqref{eq:lax1} by properly scaling vectors $u$ and $v$:
\begin{equation}\label{eq:LaxFullFormula}
T_{ji}(\{p_i\}, \{\chi_i\})= \frac{\chi_j(1 - \hbar)}{\chi_j - \chi_i \hbar}\prod\limits_{k \neq j} \frac{\chi_j-\chi_k \hbar}{\chi_j - \chi_k}\, p_j 
=\frac{\prod\limits_{k \neq i}(\chi_j-\chi_k \hbar)}{\prod\limits_{k \neq j}(\chi_j - \chi_k)}\, p_j \,.
\end{equation}
Matrix $T$ \eqref{eq:LaxFullFormula} is known as the Lax matrix of the tRS model \cite{MR1329481} (see also \cite{Gorsky:1993dq,Fock:1999ae} for more details, where a slightly different gauge is used).

The coefficients of the characteristic polynomial of the Lax matrix are the tRS Hamiltonians
\begin{equation}
 \text{det}\left(u\cdot 1 -  T(\chi_i, p_i,\hbar) \right) = \sum_{k=0}^L (-1)^lH_k(\chi_i, p_i,\hbar) u^{n-k}\,,
\label{eq:tRSLaxDecomp}
\end{equation}

The corresponding {\it tRS energy relations} (or integrals of motion) can be obtained by equating the above characteristic polynomials to $\prod_i (u-\xi_i)$ 
\begin{equation}
\sum_{\substack{\mathcal{I}\subset\{1,\dots,L\} \\ |\mathcal{I}|=k}}\prod_{\substack{i\in\mathcal{I} \\ j\notin\mathcal{I}}}\frac{\hbar\,\chi_i - \chi_j }{\chi_i-\chi_j}\prod\limits_{m\in\mathcal{I}}p_m = e_k (\xi_i)\,,
\end{equation}
where $e_l$ is the $l$-th elementary symmetric function.

\begin{Rem}
The tRS model can be regarded as relativistic deformation of the Calogero-Moser model with a finite speed of light. 
\end{Rem}

\subsection{Partial flag quiver $X^{\boldsymbol{\lambda}}$ in the Electric Frame.}

In this section, we consider a generalization of complete flag quiver, namely partial flag quiver  $X^{\boldsymbol{\lambda}}$ defined by vector $\boldsymbol{\lambda}=\{\textbf{v}_{1}, \textbf{v}_{2}-\textbf{v}_{1},\dots,\textbf{v}_{r}-\textbf{v}_{r-1}, L-\textbf{v}_{r}\}$:

\begin{figure}[!h]
\includegraphics[scale=0.4]{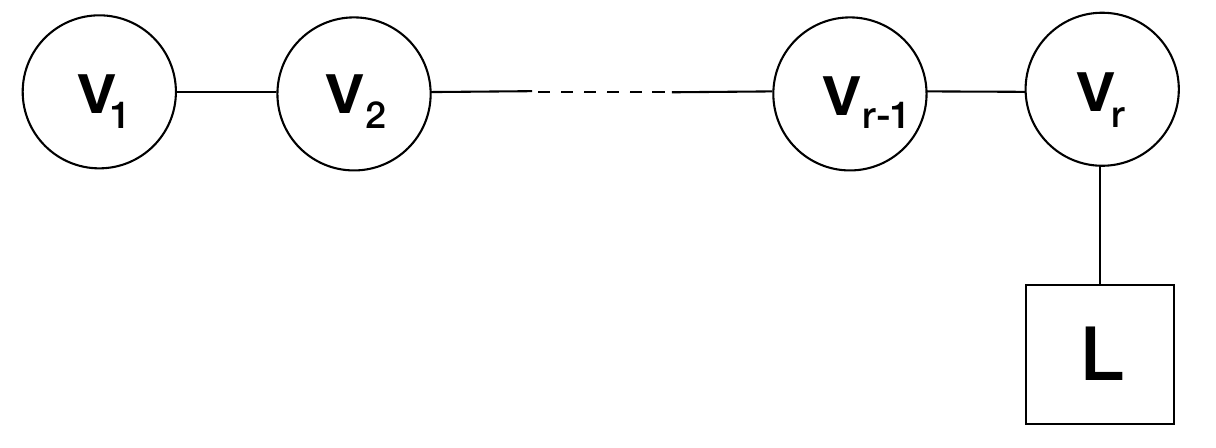}
\caption{Partial flag quiver $X^{\boldsymbol{\lambda}}$.}
\label{fig:PartialFlag}
\end{figure}

Below we will show that the algebra ${\rm Fun}({\hbar\rm Op})(X^{\boldsymbol{\lambda}})$ can be again described as an algebra of functions on the intersection of Lagrangian subvarieties tRS Hamiltonians. However, the choice of these subvarieties is different. Namely, we set the first as $\chi_i=a_i$, where $\{a_i\}$ define regular singularities for the Miura $(SL(r+1),\hbar)$-oper. The second one is again defined via a commuting family of tRS Hamiltonians. We will show that below directly by looking at the spectrum of the tRS Lax matrix \eqref{eq:LaxFullFormula}. We will refer to such description of ${\rm Fun}({\hbar\rm Op})(X^{\boldsymbol{\lambda}})$ as {\it electric frame}.

First, let us give an explicit answer, to how the electric frame tRS coordinates and momenta variables are related to the parameters of regular singularities and Bethe roots:  
\begin{eqnarray}\label{eq:deftRSel}
&&\chi_i=a_i\,,\qquad i=1,\dots, L\,, \nonumber \\
&&p_i = \xi_{r+1}\hbar^{-\textbf{v}_r} \prod_{j=1}^{\textbf{v}_r} \frac{\hbar a_i - s_{r,j}}{s_{r,j} - a_i}=(-\hbar)^{-\textbf{v}_r}\xi_{r+1}\frac{Q_r^+(\hbar a_i)}{Q_r^+(a_i)}\,,\qquad i=1,\dots, L\,,
\end{eqnarray}

\begin{Rem}
One may notice that they emerge from the  XXZ Yang-Yang function as follows: $p_i= \exp\partial_{\log a_i} Y$.
\end{Rem}

Let us have a look at the simplest nontrivial example to see how the construction works.

\vskip.1in

\noindent {\bf Example}: {$X^{\boldsymbol{\lambda}}$, $L=2$ and $\textbf{v}_r=1$}.\\
In this case the Lax matrix \eqref{eq:LaxFullFormula} reads
\begin{equation}
T = \left(
\begin{array}{cc}
 \displaystyle\frac{a_1-a_2 \hbar }{a_1-a_2} p_1&
    \displaystyle\frac{a_1(\hbar -1)}{a_1-a_2}p_1 \\
  \displaystyle\frac{a_2  (\hbar -1)}{a_2-a_1} p_2&
    \displaystyle\frac{a_2-a_1 \hbar}{a_2-a_1}p_2
   \\
\end{array}
\right)
\end{equation}
We can immediately find its eigenvectors and eigenvalues, however, we are looking for a convenient parameterization from which the relationship with Bethe equations will become transparent. Let
\begin{equation}
p_i = -\frac{\xi_2}{\hbar}\frac{\hbar a_i-s }{a_i-s}\,,\qquad i =1,2
\end{equation}
according to \eqref{eq:deftRSel}, then $T$ has the following eigenvalues
$$
\xi_2\,,\qquad \frac{\xi _2}{\hbar} \frac{\left(s -a_1 \hbar\right) \left(s -a_2 \hbar \right)}{\left(s-a_1\right) \left(s -a_2\right)}\,.
$$
If the denote the second eigenvalue $\xi_1/\hbar$ we get the following Bethe equation for $T^*\mathbb{P}^1$
\begin{equation}
\frac{\xi _2}{\xi_1} \frac{\left(s -a_1 \hbar\right) \left(s -a_2 \hbar \right)}{\left(s-a_1\right) \left(s -a_2\right)}=1\,.
\end{equation}
Thus the above equation is equivalent to the fact that $\xi_1$ and $\xi_2$ are the eigenvalues of $T$, or to the two underlying tRS equations
\begin{equation}
 \frac{a_1-a_2 \hbar}{a_1-a_2} p_1 +  \frac{a_2-a_1 \hbar}{a_2-a_1}p_2 = \xi_1+\xi_2\,,\qquad p_1 p_2 = \xi_1 \xi_2\,.
\end{equation}

\vskip.1in

Now let us consider a more general case of $X^{\boldsymbol{\lambda}}$ and formulate the main theorem.  It is equivalent to the main statement of \cite{Zabrodin:}. We shall give a different proof momentarily, which is relevant for the following. 

\begin{Thm}\label{Th:tRSXXZEl}
There is an isomorphism of algebras:
\begin{eqnarray}\label{eq:FunAlgFl}
{\rm Fun}({\hbar\rm Op})(X^{\boldsymbol{\lambda}})\cong \frac{\mathbb{C}(\{\xi_i\}, \{a_i\}, \hbar)(\{p_i\})}{\left(\det\left(u-T(\{p_i\}, \{a_i\}, \hbar)\right)-f(u,\{\xi_i\},\hbar)\right)},
\end{eqnarray}
where 
\begin{equation}\label{eq:Frhs}
f(u,\{\xi_i\},\hbar)=\prod_{i=1}^r\prod_{l=0}^{\textbf{v}_{i+1}-\textbf{v}_{i}}\left(u-\xi_i \hbar^{\textbf{v}_{i+1}-\textbf{v}_{i}-1-l}\right)
\end{equation}
so that $u$ is a formal variable and $\textbf{v}_{r+1}=L$.
\end{Thm}

In other words, Theorem \ref{Th:tRSXXZEl} states that for the tRS Lax matrix $M$ \eqref{eq:LaxFullFormula} to have eigenvalues
\begin{equation}\label{eq:degenerateeigenvalues}
\xi_{r+1},\dots,\xi_{r+1}\hbar^{L-\textbf{v}_r-1}, \xi_{r},\dots,\xi_{r+1}\hbar^{\textbf{v}_{r-1}-\textbf{v}_r-1},\dots, \xi_1\,\dots,\xi_1\hbar^{\textbf{v}_1-1}
\end{equation}
the following Bethe Ansatz equations for $X^{\boldsymbol{\lambda}}$ must be satisfied
\begin{equation}\label{eq:BetheXlambda}
\hbar^{2\textbf{v}_I-\textbf{v}_{I-1}-\textbf{v}_{I+1}}\frac{\xi_{I+1}}{\xi_I}\prod_{j=1}^{\textbf{v}_{I-1}} \frac{s_{I-1,j} - s_{I,i}}{s_{I,i} - \hbar s_{I-1,j}}\cdot\prod_{k\neq i}^{\textbf{v}_I}\frac{\hbar s_{I,i} -  s_{I,k}}{s_{I,i} - \hbar s_{I,k}}\cdot\prod_{b=1}^{\textbf{v}_{I+1}}\frac{s_{I,i}- \hbar s_{I+1,b}}{ s_{I,i}-s_{I+1,b}}=1\,,
\end{equation}
where $i=1,\dots,\textbf{v}_I$ and $I=1,\dots,r$, provided that tRS momenta in $T$ are given in \eqref{eq:deftRSel}.

\begin{proof}

Let $g\in GL(L)$ be a matrix such that the first $L-\textbf{v}_r$ eigenvalues of $T$ take values $\xi_{r+1} \hbar^{i}$ for $i=0,\dots,L-\textbf{v}_r-1$. Such matrix $g$ can put $T$ in a block-diagonal form 
\begin{equation}\label{eq:GaugeTransf1}
g^{-1} T g=  
\begin{pmatrix}
\xi_{r+1} &0  & \dots &0 &0\cr 
0 & \xi_{r+1}\hbar &  \dots &0 &0\cr
\vdots & \vdots & \ddots & \vdots &\vdots\cr
0 & 0& \dots & \xi_{r+1} \hbar^{L-\textbf{v}_r-1} &0 \cr
0 &0& \dots &0 &   T'
\end{pmatrix}\,,
\end{equation}
with $T'$ coinciding with the Lax matrix for $\textbf{v}_r$ particles which are parameterized by Bethe roots $s_{r,1},\dots, s_{r,\textbf{v}_r}$. 

After that, we can use another gauge transformation in order to reduce $T'$ to a block-diagonal form $(\textbf{v}_r-\textbf{v}_{r-1},\textbf{v}_{r-1})$ similar to the above and proceed by induction until we reach the left end of the quiver $X^{\boldsymbol{\lambda}}$.

Consider gauge transformation from \eqref{eq:GaugeTransf1}. Matrix $g$ can be presented as $(g_1,\dots, g_{L-\textbf{v}_r},g')$ where $g_1,\dots, g_{L-\textbf{v}_r}$ are columns and $g'$ is a $L\times(L-\textbf{v}_r)$ matrix.

We have the linear equations for $g$ coming from \eqref{eq:GaugeTransf1}
\begin{equation}
(T g_1,\dots, Tg_{L-\textbf{v}_r}, Tg') = (\xi_{r+1}g_1,\dots, \xi_{r+1} \hbar^{L-\textbf{v}_r-1}g_{L-\textbf{v}_r},g'T')
\label{eq:MTdecomp}
\end{equation}
From here we discover that $g_1,\dots, g_{L-\textbf{v}_r}$ are eigenvectors of $T$, with eigenvalues $\xi_{r+1} ,\dots,\xi_{r+1} \hbar^{L-\textbf{v}_r-1}$ and 
\begin{equation}\label{eq:RedT}
\sum\limits_{b=1}^L T_{ab} \cdot g'_{bi} = \sum\limits_{j=1}^{\textbf{v}_r}g'_{aj}\cdot T'_{ji}\,,
\end{equation}
where $T'$ is given by \eqref{eq:LaxFullFormula} with $a_i$'s replaced by $s_{r,i}$'s and $p_a$s replaced by $p^s_i$.

The following Lemma provides an explicit formula for the above similarity transformation in terms of Bethe roots. The relation \eqref{eq:MTdecomp} is satisfied provided that tRS momenta can also be expressed in terms of Bethe roots.

\begin{Lem}\label{Th:LemmaDiagonalization}
Let 
\begin{equation}\label{eq:gtranfsub}
g'_{ai} = \frac{\delta_{a}b_i }{a_a - s_{r,i}}\,, \qquad \delta_{a} =\sum\limits_{k=1}^{\textbf{v}_r}\gamma_k \frac{a_a(a_a-s_{r,k})}{a_a-s_{r,k}}\,,
\end{equation}
where $a=1,\dots,L\,, i=1,\dots, \textbf{v}_r$ and $b_i$ and $\gamma_k$ are some nonzero constants. Let
\begin{equation}\label{eq:pmupsigmasub}
p_a = \frac{\xi_{r+1}}{\hbar^{\textbf{v}_r}}\prod_{k=1}^{\textbf{v}_r}\frac{\hbar a_a - s_{r,k}}{a_a - s_{r,k}}\,, \qquad p^s_i = \frac{\xi_{r+1}}{\hbar^{L-2\textbf{v}_r}}\prod_{b=1}^L\frac{s_{r,i}- \hbar a_b}{ s_{r,i}- a_b}\prod_{k\neq i}^{\textbf{v}_r}\frac{\hbar s_{r,i} - s_{r,k}}{s_{r,i} - \hbar s_{r,k}}\,,
\end{equation}
for some nonzero $s_{r,1},\dots,s_{r,{\bf v}_r}$. Assume $s_{r,i}\neq \hbar^{\mathbb{Z}}s_{r,j}$ and $a_i\neq \hbar^{\mathbb{Z}} a_j$ for $i\neq j$. Then \eqref{eq:MTdecomp} is satisfied.
\end{Lem}

\begin{proof}
Using \eqref{eq:pmupsigmasub} and \eqref{eq:gtranfsub} we can rewrite \eqref{eq:RedT} as follows
\begin{align}\label{eq:rankreduction}
&\sum\limits_{b=1}^L p_a a_a \frac{1 - \hbar}{a_a - a_b \hbar}\prod_{c \neq a}^L \frac{a_a-a_c \hbar}{a_a - a_c} \cdot \frac{\delta_{b}b_i}{a_b - s_{r,i}} \notag \\
&= \sum\limits_{j=1}^{\textbf{v}_r} \frac{\delta_{a} b_j}{a_a - s_{r,j}} \cdot p^s_j s_{r,j} \frac{1 - \hbar}{s_{r,i} - s_{r,j} \hbar}\prod\limits_{k \neq j}^{\textbf{v}_r} \frac{s_{r,j} -s_{r,k} \hbar}{s_{r,j} - s_{r,k}}\,.
\end{align}
One can explicitly check that the above relations holds provided that $\delta_{a}$ is given by the second relation in \eqref{eq:gtranfsub} and tRS momenta are as in \eqref{eq:pmupsigmasub}. 
\end{proof}

At the next step of the recursion, we introduce the new set of Bethe roots $s_{r-1,1},\dots,s_{r-1,{\bf v}_{r-1}}$ and define the collection of momenta $p^s_i$ analogously to the first relation in \eqref{eq:pmupsigmasub} (with $\xi_{r+1}$ replaced by $\xi_r$, $a_i$'s replaced by $s_{r,i}$ and $s_{r,i}$'s replaced by $s_{r-1,i}$s)
\begin{equation}\label{eq:pschange}
p^s_i = \frac{\xi_r}{\hbar^{\textbf{v}_{r-1}}} \prod_{j=1}^{\textbf{v}_{r-1}} \frac{s_{r,i} - \hbar s_{r-1,j}}{s_{r-1,j} - s_{r,i}}\,.
\end{equation}
This way the Lax matrix $T'$ has the canonical form \eqref{eq:LaxFullFormula}. We thus can get the first set of Bethe equations from \eqref{eq:BetheXlambda} for the $s_{r,i}$ roots by comparing the above expression with the second relation in 
\eqref{eq:pmupsigmasub}
\begin{equation}\label{eq:BetheEqProof}
\hbar^{2\textbf{v}_r-L-\textbf{v}_{r-1}}\frac{\xi_{r+1}}{\xi_r}\cdot\prod_{j=1}^{\textbf{v}_{r-1}} \frac{s_{r-1,j} - s_{r,i}}{ s_{r,i} - \hbar s_{r-1,j}}\cdot\prod_{k\neq i}^{\textbf{v}_r}\frac{\hbar s_{r,i} - s_{r,k}}{ s_{r,i} - \hbar s_{r,k}}\cdot\prod_{b=1}^L\frac{s_{r,i}- \hbar a_b}{ s_{r,i}- a_b}=1\,.
\end{equation}

The recursive application of this reasoning completes the proof in one direction. 

\vskip.1in

It remains to be shown that for every solution of tRS equations with given degeneracy of the right hand side there exists a solutions of Bethe equations \eqref{eq:BetheXlambda}. For simplicity consider $X^{\boldsymbol{\lambda}}$ in the case $r=1$, ${\bf v_1}=k$, $L=n$. A generalization to other partitions $\boldsymbol{\lambda}$ is straightforward. The eigenvalues read
\begin{equation}
\xi_{2},\dots,\xi_{2}\hbar^{n-k-1}, \xi_{1},\dots,\xi_1\hbar^{k-1}
\end{equation}
Consider relations in \eqref{eq:FunAlgFl} expanded in $z$
\begin{equation}
\label{eq:ElRelationsGrass}
\sum_{i=0}^{n}(-1)^i z^{n-i} T_{i} = \left(\sum_{l=0}^k (-1)^l z^{l-l} e_{l}(\xi_{1},\dots,\xi_1\hbar^{k-1})\right)\left(\sum_{m=0}^{n-k} (-1)^m z^{n-k-m} e_{m}(\xi_{2},\dots,\xi_2\hbar^{n-k-1})\right)\,,
\end{equation}
where $T_0=1$, $T_1,\dots T_n$ are the tRS Hamiltonians and $e_l$ are elementary symmetric functions of their arguments. We can impose $\xi_2=\zeta\xi_1$. Rescaling all momenta by $\xi_1$ give the RHS of each of the equaltions for $T_i$in (\ref{eq:ElRelationsGrass}) expressed in terms of polynomials of $\zeta$ degree $k$. That leaves at most $k$ independent degrees of freedom for $T_i$ and thus for $p_i$, since transformation from $p_i$ to $T_i$ is generally non-degenerate. 
Thus formula \eqref{eq:pschange} provides a faithful change of variables.

\end{proof}

To summarize, the relations for ${\rm Fun}({\hbar\rm Op})(X^{\boldsymbol{\lambda}})$ are as follows:
\begin{equation}\label{eq:tRSRelXlambda}
\sum_{\substack{\mathcal{I}\subset\{1,\dots,L\} \\ |\mathcal{I}|=k}}\prod_{\substack{i\in\mathcal{I} \\ j\notin\mathcal{I}}}\frac{{a}_i - \hbar\,{a}_j }{{a}_i-{a}_j}\prod\limits_{m\in\mathcal{I}}{p}_m = \ell_k (\xi_i)\,,
\end{equation}
where symmetric functions $\ell_k$ were introduced in Theorem \ref{Th:qOpMag}. We will refer to the the set of variables $\{p_i\}$, $\{a_i\}$ as {\it electric frame} momenta and coordinates correspondingly.

\begin{Rem} 
It may seem that electric momenta $p_i$ are disconnected from the $\hbar$-oper formalism we described in the previous section. However, this is not the case. For the $(SL(r+1),\hbar)$-oper the natural coordinate system is provided by the $\hbar$-Miura transformation, which is given by the $\hbar$-gauge transformation from $B_-(z)\subset SL(r+1)(z)$ which moves diagonal elements 
to the bottom row in the $\hbar$-connection matrix. The elements of the bottom row are known to coincide with the eigenvalues of the transfer matrices of XXZ spin chain which correspond to fundamental representations. The  $\hbar$-Miura transformation gives an expression of the eigenvalues of the transfer matrices via eigenvalues of Baxter Q-operators. The transfer matrix corresponding to the defining representation of $SL(r+1)$ allows a polar decomposition with respect to $\{a_i\}$. The coefficients in that expansion are known as (fundamental) Hamiltonians (cf. nonlocal Hamiltonians, se remark \ref{Rem:nonlocal}) of XXZ model. They coincide with the electric frame momenta (that can be established from the formula \eqref{eq:pschange}) and thus serve as natural coordinates for $(SL(r+1),\hbar)$-opers.
\end{Rem}


\subsection{Resonance Conditions and degeneration of $\hbar$-opers.}\label{Sec:FlagElDegen}

So far we considered only nondegenerate $\hbar$-opers. We claim that special degenerations which involve relations between Bethe roots and regular singularities, which we will refer to as we refer to as {\it resonance conditions}, will lead to interesting correspondences between opers associated to various quivers.  

\subsubsection{Recursive Procedure for Bethe Ansatz}
In this section we will show that the ring ${\rm Fun }(\hbar{\rm Op})(X_{k,l})$ (top of \figref{fig:mirroranew}) can be obtained from 
${\rm Fun }(\hbar{\rm Op})(X^{\boldsymbol{\lambda}})$ for special $\boldsymbol{\lambda}$ by imposing additional relations. We will define them in the recursive way, by chipping the framing in the quiver $X_{\lambda}$ and carrying it to the neighboring vertices as it is indicated on figure \figref{fig:PartialFlag2}. We begin with the Bethe equations \eqref{eq:BetheXlambda} for the last node of quiver $X^{\boldsymbol{\lambda}}$ in \figref{fig:PartialFlag} 
\begin{equation}\label{eq:BetheXlastQQ}
\hbar^{2\textbf{v}_r-\textbf{v}_{r-1}-L}\frac{\xi_{r+1}}{\xi_r}\prod_{j=1}^{\textbf{v}_{r-1}} \frac{s_{r-1,j} - s_{r,i}}{s_{r,i} - \hbar s_{r-1,j}}\cdot\prod_{k\neq i}^{\textbf{v}_r}\frac{\hbar s_{r,i} -  s_{r,k}}{s_{r,i} - \hbar s_{r,k}}\cdot\prod_{b=1}^{L}\frac{s_{r,i}- \hbar a_{b}}{ s_{r,i}-a_{b}}=1\,,
\end{equation}
which can be written as
\begin{align}\label{eq:BetheXlastnew}
\hbar^{2\textbf{v}_r-\textbf{v}_{r-1}-(L-1)}\frac{\xi_{r+1}}{\xi_r}\cdot\prod_{j=1}^{\textbf{v}_{r-1}} \frac{s_{r-1,j} - s_{r,i}}{s_{r,i} - \hbar s_{r-1,j}}&\cdot\prod_{\substack{k=2\\ k\neq i}}^{\textbf{v}_r}\frac{\hbar s_{r,i} -  s_{r,k}}{s_{r,i} - \hbar s_{r,k}}
\cr
&
\cdot \frac{1}{\hbar}\frac{\hbar s_{r,i} -  s_{r,1}}{s_{r,i} - \hbar s_{r,1}}
\cdot \frac{ s_{r,i}-\hbar a_2}{ s_{r,i}-a_2} \cdot \frac{s_{r,i}- \hbar a_1}{ s_{r,i}- a_1}
\cr
&
\cdot\prod_{b=3}^{L}\frac{s_{r,i}- \hbar a_{b}}{ s_{r,i}-a_{b}}=1\,.
\end{align}
If we impose the resonance conditions 
\begin{equation}\label{eq:res1}
a_1 \hbar=s_{r,1}= a_2
\end{equation}
we can observe that the middle line of the above equation cancels off
$$
\frac{1}{\hbar}\frac{\hbar s_{r,i} -  s_{r,1}}{s_{r,i} - \hbar a_2}
\cdot \frac{ s_{r,i}-\hbar a_2}{ s_{r,i}-a_2} \cdot \frac{s_{r,i}- a_2}{ s_{r,i}- a_1}=
\frac{1}{\hbar}\frac{\hbar s_{r,i} -  s_{r,1}}{ s_{r,i}- a_1}=1\,.
$$
Meanwhile the Bethe equation for the $r-1$st node 
\begin{equation}\label{eq:BetheXlastQQ2nd}
\hbar^{2\textbf{v}_{r-1}-\textbf{v}_{r-2}-\textbf{v}_{r}}\frac{\xi_{r}}{\xi_{r-1}}\prod_{j=1}^{\textbf{v}_{r-2}} \frac{s_{r-2,j} - s_{r-1,i}}{s_{r-1,i} - \hbar s_{r-2,j}}\cdot\prod_{k\neq i}^{\textbf{v}_{r-1}}\frac{\hbar s_{r-1,i} -  s_{r-1,k}}{s_{r-1,i} - \hbar s_{r-1,k}}\cdot\prod_{b=1}^{\textbf{v}_{r}}\frac{s_{r-1,i}- \hbar s_{r,b}}{ s_{r-1,i}-s_{r,b}}=1\,,
\end{equation}
becomes
\begin{align}\label{eq:BetheXlast2}
\hbar^{2\textbf{v}_{r-1}-\textbf{v}_{r-2}-\textbf{v}_{r}}&\frac{\xi_{r}}{\xi_{r-1}}\prod_{j=1}^{\textbf{v}_{r-2}} \frac{s_{r-2,j} - s_{r-1,i}}{s_{r-1,i} - \hbar s_{r-2,j}}\cdot\prod_{k\neq i}^{\textbf{v}_{r-1}}\frac{\hbar s_{r-1,i} -  s_{r-1,k}}{s_{r-1,i} - \hbar s_{r-1,k}}\cr
&\cdot\prod_{b=2}^{\textbf{v}_{r}}\frac{s_{r-1,i}- \hbar s_{r,b}}{ s_{r-1,i}- s_{r,b}}\cdot \frac{ s_{r-1,i}-\hbar^2 a_1}{s_{r-1,i}-\hbar a_1}=1\,,
\end{align}
where the last term in the left hand side corresponds to the framing on the $r-1$st node with equivariant parameter $\hbar a_1$. 
Thus the rank of the framing bundle on the $r$th node decreases by two while the rank of the framing on the $r-1$st node becomes one. Let us call the resulting quiver variety $X^{\boldsymbol{\lambda}}_{\boldsymbol{\lambda'}}$.

\begin{figure}[!h]
\includegraphics[scale=0.3]{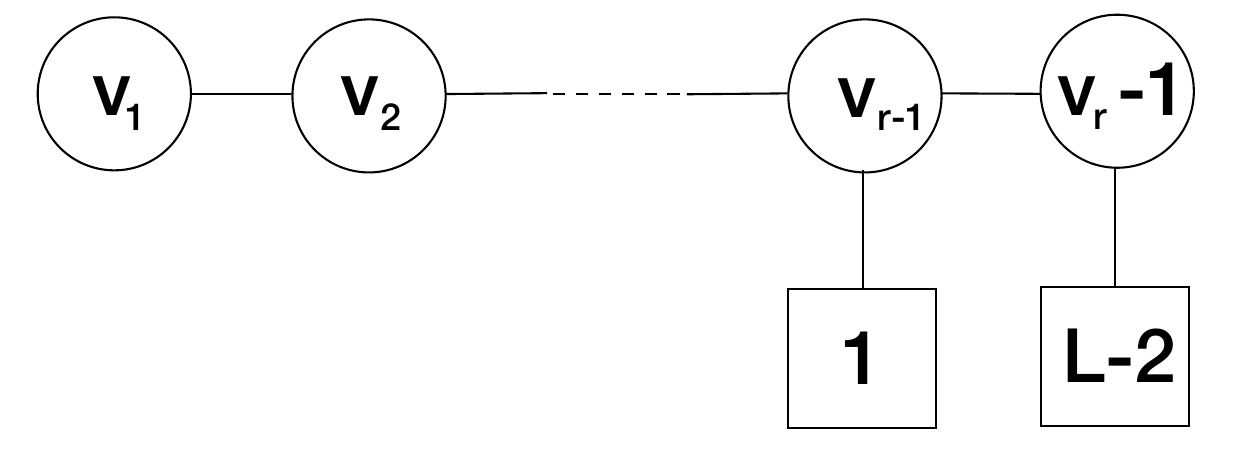}\quad \includegraphics[scale=0.3]{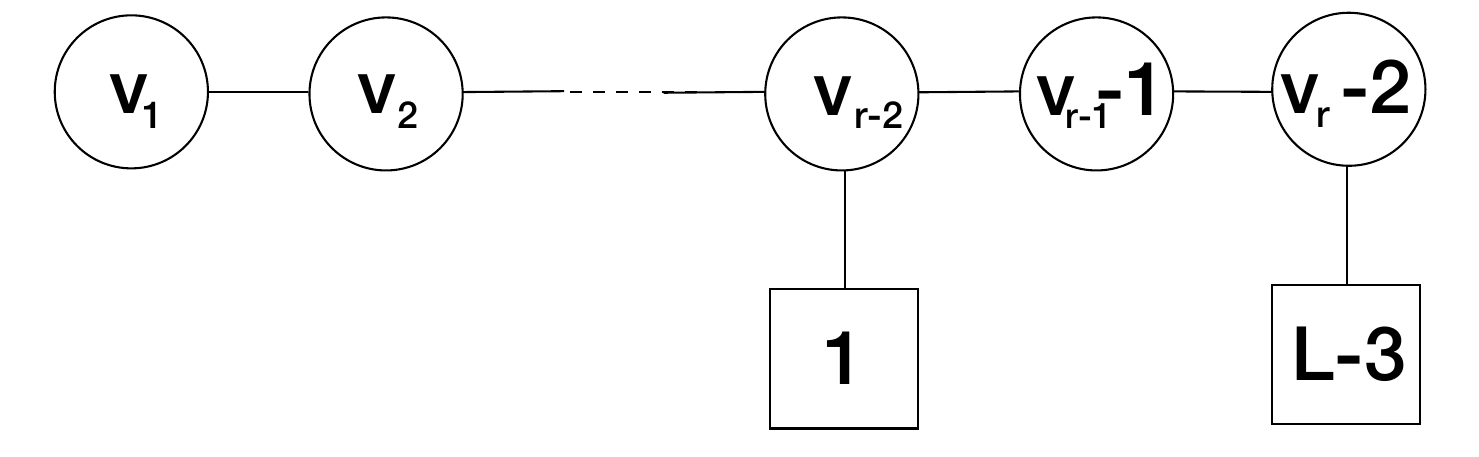}
\caption{Quivers $X^{\boldsymbol{\lambda}}_{\boldsymbol{\lambda'}}$ and $X^{\boldsymbol{\lambda}}_{\boldsymbol{\lambda''}}$.}
\label{fig:PartialFlag2}
\end{figure}

\subsubsection{Recursive Procedure for the tRS Model}

Lemma \ref{Th:LemmaDiagonalization} was proven under the nondegeneracy assumptions, namely, $s_{r,i}\neq \hbar^{\mathbb{Z}}s_{r,j}$ and $a_i\neq \hbar^{\mathbb{Z}} a_j$ for $i\neq j$. Once we impose \eqref{eq:res1} and similar relations these assumptions are no longer valid so we need to modify the proof which will lead to certain changes is the formulae for the tRS momenta in terms of Bethe roots \eqref{eq:pmupsigmasub}.

Let us see how implementing resonance conditions affects the formulae for $p_a$. Consider \eqref{eq:rankreduction} for $a=i=1$ 
\begin{align}\label{eq:rankreduction2}
\sum\limits_{b=1}^L p_1 a_1 \frac{1 - \hbar}{a_1 - a_b \hbar} \prod_{c \neq 1}^L \frac{a_1-a_c \hbar}{a_1 - a_c} \cdot \frac{\delta_{b}b_i}{a_b - s_{r,1}}
&= \sum\limits_{j=1}^{\textbf{v}_r} \frac{\delta_{1} b_j}{a_1 - s_{r,j}} \cdot p^s_j s_{r,j} \frac{1 - \hbar}{s_{r,1} - s_{r,j} \hbar}\prod\limits_{k \neq j}^{\textbf{v}_r} \frac{s_{r,j} -s_{r,k} \hbar}{s_{r,j} - s_{r,k}}\,.
\end{align}
and for $a=2,\,i=1$ 
\begin{align}\label{eq:rankreduction3}
\sum\limits_{b=1}^L p_2 a_2 \frac{1 - \hbar}{a_2 - a_b \hbar}\frac{a_2-a_1 \hbar}{a_2 - a_1} \prod_{c \neq 1,2}^L \frac{a_2-a_c \hbar}{a_2 - a_c} \cdot \frac{\delta_{b}b_i}{a_b - s_{r,1}}
&= \sum\limits_{j=1}^{\textbf{v}_r} \frac{\delta_{2} b_j}{a_2 - s_{r,j}} \cdot p^s_j s_{r,j} \frac{1 - \hbar}{s_{r,1} - s_{r,j} \hbar}\prod\limits_{k \neq j}^{\textbf{v}_r} \frac{s_{r,j} -s_{r,k} \hbar}{s_{r,j} - s_{r,k}}\,.
\end{align}
In order to consolidate poles in the left and right hand sides we define
\begin{equation}\label{eq:pmupsigmasub2}
p_1 = \frac{\xi_{r+1}}{\hbar^{\textbf{v}_r}}\prod_{k=1}^{\textbf{v}_r}\frac{\hbar a_1 - s_{r,k}}{a_1 - s_{r,k}}\cdot \frac{a_1- a_2}{\hbar a_1-a_2}\,, \qquad  p_2 = \frac{\xi_{r+1}}{\hbar^{\textbf{v}_r}}\prod_{k=1}^{\textbf{v}_r}\frac{\hbar a_2 - s_{r,k}}{a_2 - s_{r,k}}\cdot \frac{\hbar a_1-a_2}{a_1- a_2}\,,
\end{equation}
and 
\begin{equation}
p^s_1 = \frac{\xi_{r+1}}{\hbar^{L-2\textbf{v}_r}}\prod_{b=1}^L\frac{s_{r,1}- \hbar a_b}{ s_{r,1}- a_b}\prod_{k\neq i}^{\textbf{v}_r}\frac{\hbar s_{r,1} - s_{r,k}}{s_{r,1} - \hbar s_{r,k}}\cdot\frac{s_{r,1}-a_2}{ s_{r,1}- \hbar a_2}\,.
\end{equation}
Now if we take simultaneous limits $s_{r,1}\to a_2$ and $a_2\to a_1 \hbar$ all the above expressions remain finite and become
\begin{equation}\label{eq:pmupsigmasub3}
p_1 = \frac{\xi_{r+1}}{\hbar^{\textbf{v}_r}}\prod_{k=2}^{\textbf{v}_r}\frac{\hbar a_1 - s_{r,k}}{a_1 - s_{r,k}}\,, \qquad  p_2 = -\frac{\xi_{r+1}}{\hbar^{\textbf{v}_r-1}}\prod_{k=2}^{\textbf{v}_r}\frac{\hbar a_2 - s_{r,k}}{a_2 - s_{r,k}}\,,
\end{equation}
as well as
\begin{equation}
p^s_1 = \frac{\xi_{r+1}}{\hbar^{L-2\textbf{v}_r}}\prod_{b=1}^{L-1}\frac{a_2- \hbar a_b}{a_2- a_b}\prod_{k\neq i}^{\textbf{v}_r}\frac{\hbar a_2 - s_{r,k}}{a_2 - \hbar s_{r,k}}\,.
\end{equation}
Effectively this amounts to rescaling the tRS momenta for the unconstrained parameters $a_i$ from the proof of Lemma \ref{Th:LemmaDiagonalization} by
\begin{equation}\label{eq:rescalingp1}
p_1 \mapsto p_1 \frac{\hbar a_1-a_2}{a_1- a_2}\,,\qquad p_2\mapsto p_2 \frac{a_1- a_2}{\hbar a_1-a_2}
\end{equation}

Proceeding along the lines of the proof of Lemma \ref{Th:LemmaDiagonalization} and using \eqref{eq:pschange} with $s_{r,1}=a_2$ we shall arrive at XXZ Bethe equations \eqref{eq:BetheXlastnew} for the last node and to \eqref{eq:BetheXlast2} for the $r-1$st node with other equations unchanged. In other words, we have realized in the language of the tRS models in the electric frame how to perform the reduction from $X^{\boldsymbol{\lambda}}$ to quiver $X^{\boldsymbol{\lambda}}_{\boldsymbol{\lambda'}}$ shown in \figref{fig:PartialFlag2} on the left.

Thus the following Lemma follows from the above calculations and from Theorem \ref{Th:tRSXXZEl}.

\begin{Lem}\label{Th:HiggsingLemma}
The algebra ${\rm Fun }(\hbar{\rm Op})(X^{\boldsymbol{\lambda}}_{\boldsymbol{\lambda'}})$ is given by the following quotient
\begin{equation}
{\rm Fun }(\hbar{\rm Op})(X^{\boldsymbol{\lambda}}_{\boldsymbol{\lambda'}})\cong 
\frac{\mathbb{C}(\{\xi_i\}, \{a_i\},\{p_i\}, \hbar)}{\left(\det\left(u-T(\{p_i\}, \{a_i\}_{i\neq 2}, \hbar)\right)-f(u,\{\xi_i\},\hbar)\right)\vert_{a_2=\hbar a_1}}
\end{equation}
\end{Lem}

\vskip.1in

If one imposes further resonance conditions in the form
\begin{equation}\label{eq:res2}
\hbar^2 a_1 = \hbar s_{r-1,1} = s_{r,2} =  a_3
\end{equation}
a similar transition in Bethe equations happens and these equations now correspond to a quiver variety with rank-one framing  on the $(r-2)$-nd node and the framing of rank $L-3$ on the last node (right picture in \figref{fig:PartialFlag2}). 

\begin{Lem}\label{Th:HiggsingLemma2}
The algebra ${\rm Fun }(\hbar{\rm Op})(X^{\boldsymbol{\lambda}}_{\boldsymbol{\lambda''}})$ is given by the following quotient
\begin{equation}
{\rm Fun }(\hbar{\rm Op})(X^{\boldsymbol{\lambda}}_{\boldsymbol{\lambda''}})\cong
\frac{\mathbb{C}(\{\xi_i\}, \{a_i\}_{i\neq 2,3},\{p_i\}, \hbar)}{\left(\det\left(u-T(\{p_i\}, \{a_i\}, \hbar)\right)-f(u,\{\xi_i\},\hbar)\right)\vert_{a_2=\hbar a_1, ~a_3=\hbar^2 a_1}}
\end{equation}
\end{Lem}

The corresponding momenta variables $p_1$ and $p_3$ will acquire rescaling factors similar to \eqref{eq:rescalingp1} such that the resulting rational functions of Bethe roots will be nonsingular given by \eqref{eq:res2}.

\vskip.1in
\begin{Rem}
This process of imposing resonance conditions on the equivariant parameters will be later illustrated in \figref{fig:hwmoves} using moves of D-branes.
\end{Rem}

\subsubsection{Quiver $X_{k,l}$ in Electric Frame and Degenerate Miura $(SL(r+1),\hbar)$-Opers.} 

We are now ready to describe the set of relations needed for the $X_{k,l}$ family \figref{fig:QuiverABNew}. Let us consider the specific class of $X^{\boldsymbol{\lambda}}$ in \figref{fig:PartialFlag} which depends on $k, l\in \mathbb{N}$ and $r=k+l-1$
\begin{eqnarray}\label{xlambdadata}
&&\textbf{v}_i = i,\quad i = 1,\dots, k,\qquad \textbf{v}_{k+m} = k+\frac{m(m+1)}{2}\,,~ m=1, \dots, l-1\\
&&L = k+\frac{l(l-1)}{2}\,.\nonumber
\end{eqnarray}
We denote the twist parameters $\upxi_1,\dots,\upxi_{k+l}$ and $L$ coordinates parametrizing regular singularities are relabeled 
as follows:
$$
a_{k,1},\dots,a_{k,k+1}, 
$$
$$
a_{k+1,1},\dots, a_{k+1,l-1},
$$
$$
a_{k+2,1},\dots, a_{k+2,l-2},
$$
$$
\dots\dots
$$
$$
a_{k+l-2,1},a_{k+l-2,2},
$$
$$
a_{k+l-1}\,.
$$

In order to describe the transition from ${\rm Fun }(\hbar{\rm Op})(X^{\boldsymbol{\lambda}})$ to ${\rm Fun }(\hbar{\rm Op})(X_{k,l})$ we need impose a family of resonance conditions. Let us consider the following constraints imposed on the variables $\{a_{i,j}\}$, thereby introducing new parameters $\{a_{j}\}$ for $j=k+1,\dots,k+l-2$
\begin{align}\label{eq:HiggsingConditions}
a_{k+1}:=a_{k+1,1}&=\hbar a_{k+1,2}=\hbar^2 a_{k+1,3}= \dots=\hbar^{r-2}a_{k+1,l-1}, \notag\\
a_{k+2}:=a_{k+2,1}&=\hbar a_{k+2,2}=\hbar^2 a_{k+2,3}= \dots=\hbar^{l-3}a_{k+2,l-3}, \notag\\
&\dots\notag\\
a_{k+l-2}:=a_{k+l-2,1}&=\hbar a_{k+l-2,2}.
\end{align}

Additionally, according to the results of the previous section, one has electric momenta $\{ p_{i,j}\}$ to each of the above $\{a_{i,j}\}$ coordinates.

\vskip.1in
\textbf{Remark.}
The reader may have noticed the similarity between the structures of \eqref{xlambdadata} and \eqref{eq:rootsWXkl} from the previous section. This observation will play a role in proving self-mirror duality for $X_{k,l}$ later.

\vskip.1in

We denote the algebra ${\rm Fun }(\hbar{\rm Op})(X^{\boldsymbol{\lambda}})$ with the resonance conditions imposed on regular singularities as ${\rm{Fun}}_{k,l}^{res}({\hbar\rm Op})(X^{\boldsymbol{\lambda}})$. We obtain the following statement.

\begin{Thm}Given $X^{\boldsymbol{\lambda}}$ with $\boldsymbol{\lambda}$ described by ${\bf v}_i$ in (\ref{xlambdadata}). 
Then algebra ${\rm Fun }(\hbar{\rm Op})(X_{k,l})$ is the following quotient
\begin{eqnarray}\label{eq:FunQuot}
{\rm Fun }(\hbar{\rm Op})(X_{k,l})\cong {\rm{Fun}}_{k,l}^{res}(\hbar{\rm Op})(X^{\boldsymbol{\lambda}}),
\end{eqnarray}
where the correspondence between the equivariant parameters of $X_{k,l}$ and $X^{\boldsymbol{\lambda}}$ reads
$$
a_{k+j} = a_{k+j,l-j}\,,\qquad j=1,\dots, l-2\,.
$$
\end{Thm}

\begin{proof}
The statement can be verified directly by iteratively applying Lemma \ref{Th:HiggsingLemma} to two selected equivariant parameters for each equality sign in \eqref{eq:HiggsingConditions} and the corresponding Bethe roots. 

In the tRS Lax matrix \eqref{eq:tRSRelXlambda} one needs to rescale the momenta before imposing each resonance condition from \eqref{eq:HiggsingConditions}. 
For instance, for the first line of \eqref{eq:HiggsingConditions} we have the following $l-2$ conditions 
\begin{align}\label{eq:resonance1}
\hbar a_{k+1,1} &= s_{k+l-1,1} =  a_{k+1,2}\,,\cr
\hbar^2 a_{k+1,1} &=  \hbar s_{k+l-2,1} = s_{k+l-1,2}= a_{k+1,3}\,,\cr
\dots & \dots\\
\hbar^{l-1} a_{k+1,1} &= \hbar^{l-2} s_{k+l,1} = \dots =  a_{k+1,l-1}\,.\notag
\end{align}
We then apply the following rescaling for the first two momenta
\begin{align}\label{eq:rescale3}
p'_{k+1,1} = p_{k+1,1} \frac{a_{k+1,1}-a_{k+1,2}}{\hbar a_{k+1,1}- a_{k+1,2}}\,,\qquad  p'_{k+1,2} = p_{k+1,2} \frac{\hbar a_{k+1,1}- a_{k+1,2}}{a_{k+1,1}-a_{k+1,2}}\,
\end{align}
and take limits $s_{k+l-1,1}\to a_{k+1,2}$ and $a_{k+1,2}\to\hbar a_{k+1,1}$. Then repeat this procedure for the other variables in \eqref{eq:resonance1}. After imposing these relations the space ${\rm Fun }(\hbar{\rm Op})(X_{k,l})$ becomes isomorphic to ${\rm Fun }(\hbar{\rm Op})(X^{\boldsymbol{\lambda}}_{\boldsymbol{\lambda}^{(1)}})$ where  $X^{\boldsymbol{\lambda}}_{\boldsymbol{\lambda}^{(1)}}$ is shown in \figref{fig:AmodelNewX1}. The Bethe equations transform accordingly.
\begin{figure}[!h]
\includegraphics[scale=0.33]{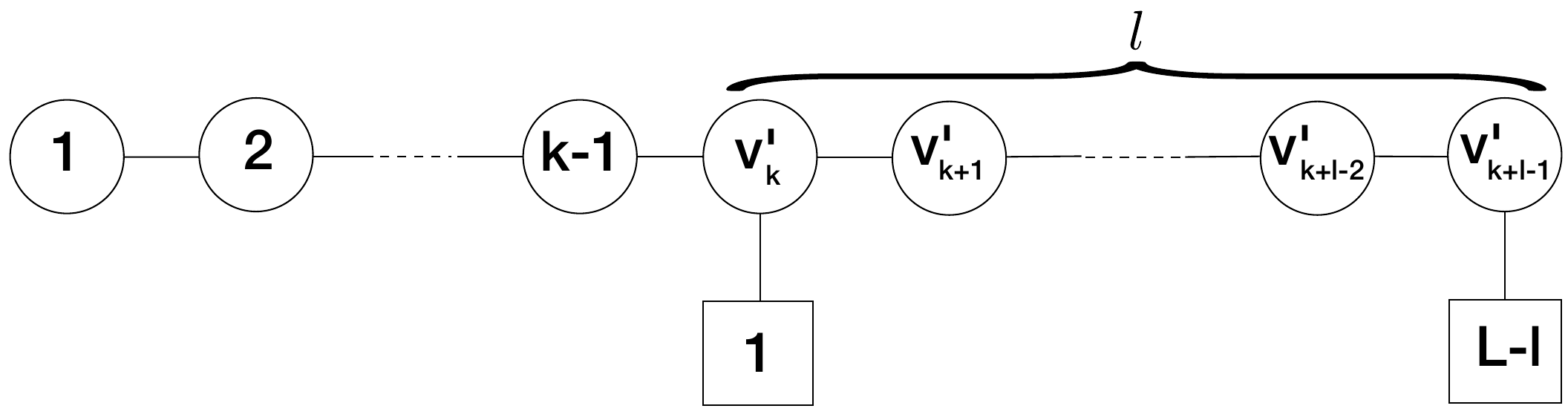}
\caption{Quiver $X^{\boldsymbol{\lambda}}_{\boldsymbol{\lambda}^{(1)}}$.}
\label{fig:AmodelNewX1}
\end{figure}
Here ${\bf v}'_{k+j} = {\bf v}_{k+j}-(j-1)$ for $j=1,\dots, l-1$. 

On the next step we impose conditions analogous to \eqref{eq:resonance1} which enable the isomorphism between ${\rm Fun }(\hbar{\rm Op})(X^{\boldsymbol{\lambda}}_{\boldsymbol{\lambda}^{(1)}})$ and ${\rm Fun }(\hbar{\rm Op})(X^{\boldsymbol{\lambda}}_{\boldsymbol{\lambda}^{(2)}})$ where the latter quiver variety is shown in \figref{fig:AmodelNewX1}.
\begin{figure}[!h]
\includegraphics[scale=0.33]{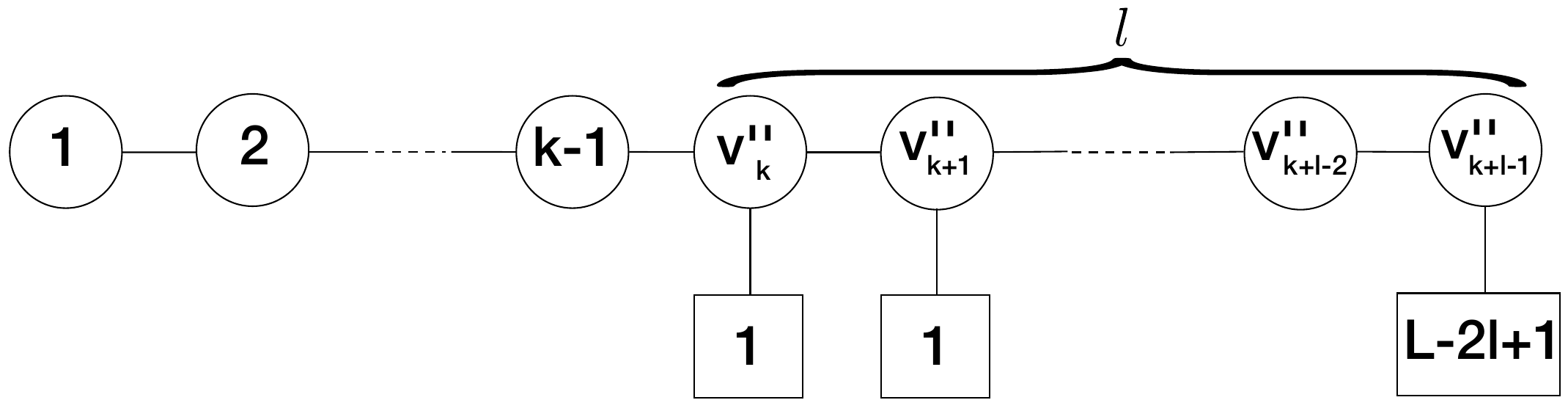}
\caption{Quiver $X^{\boldsymbol{\lambda}}_{\boldsymbol{\lambda}^{(2)}}$.}
\label{fig:AmodelNewX2}
\end{figure}

Proceeding by induction we arrive at ${\rm Fun }(\hbar{\rm Op})(X_{k,l})$ and the theorem follows.
\end{proof}

\vskip.1in

\subsection{Example}
It is instructive to show explicitly how the tRS Lax matrix in the electric frame looks like in the presence of the resonance conditions. Consider $r=4$, so that we deal with the five-particle model, and impose 
\begin{equation}\label{eq:degeneration5}
a_3=\hbar a_2= a_1 \hbar^2\,,\qquad a_5=\hbar a_4. 
\end{equation}
Then matrix $T$ has the following $3\times 2$ block form:
{\small
\begin{equation}\label{eq:TRSSlodowy25}
\hspace*{-3.1cm}  \left(
\begin{array}{ccccc}
 \frac{p_1 \left(\hbar ^2+\hbar +1\right) \left(a_1-a_4 \hbar
   ^2\right)}{\left(a_1-a_4\right) \hbar ^2} & -\frac{p_1 \left(\hbar
   ^2+\hbar +1\right) \left(a_4 \hbar -a_1\right) \left(a_4 \hbar
   ^2-a_1\right)}{\left(a_1-a_4\right) \hbar ^4 \left(a_1 \hbar -a_4\right)} &
   \frac{p_1 \left(a_1-a_4 \hbar \right) \left(a_1-a_4 \hbar
   ^2\right)}{\hbar ^4 \left(a_1 \hbar -a_4\right) \left(a_1 \hbar
   ^2-a_4\right)} & \frac{a_1^3 p_1 (\hbar -1)^2 (\hbar +1) \left(\hbar
   ^2+\hbar +1\right) \left(a_4 \hbar ^2-a_1\right)}{\left(a_1-a_4\right) a_4
   \hbar ^2 \left(a_1 \hbar -a_4\right) \left(a_1 \hbar ^2-a_4\right)} &
   \frac{a_1^3 p_1 (\hbar -1)^2 (\hbar +1) \left(\hbar ^2+\hbar
   +1\right)}{\left(a_1-a_4\right) a_4 \hbar ^4 \left(a_1 \hbar -a_4\right)} \\
 p_2 \hbar ^2 & 0 & 0 & 0 & 0 \\
 0 & p_3 \hbar ^2 & 0 & 0 & 0 \\
 \frac{a_4^2 p_4 \left(a_1 \hbar ^2-a_4\right) \left(a_1 \hbar
   ^3-a_4\right)}{a_1^2 \left(a_1-a_4\right) \hbar ^2 \left(a_1-a_4 \hbar
   \right)} & -\frac{a_4^2 p_4 (\hbar +1) \left(a_1 \hbar
   ^3-a_4\right)}{a_1^2 \left(a_1-a_4\right) \hbar ^4} & \frac{a_4^2
   p_4}{a_1^2 \hbar ^4} & \frac{p_4 (\hbar +1) \left(a_1 \hbar
   ^3-a_4\right)}{\left(a_1-a_4\right) \hbar ^2} & \frac{p_4 \left(a_1
   \hbar ^2-a_4\right) \left(a_1 \hbar ^3-a_4\right)}{\left(a_1-a_4\right) \hbar
   ^4 \left(a_4 \hbar -a_1\right)} \\
 0 & 0 & 0 & p_5 \hbar ^2 & 0 \\
\end{array}
\right)
\end{equation}
}
\vskip.1in
Notice that this matrix depends only on coordinates $a_1$ and $a_4$ and all five momenta. Also the above matrix has the so-called Slodowy structure meaning that each diagonal block has the form $F+a$, where $F$ is the image of the lowering operator $f$ from the $\sl_2$ triple under embedding $\rho: \mathfrak{sl}_2 \hookrightarrow \mathfrak{sl}_5$, and $a$ is the highest weight vector (top row) (see \cite{Gaiotto:2008sa} where the opposite convention was used: $e$ instead of $f$).

The corresponding tRS Hamiltonians are thus the coefficients of the characteristic polynomial of this matrix and the tRS relations read
$$
\det(u-T) = \prod_{i=1}^5(u-\upxi_i)
$$
These tRS relations describe space Fun$(\hbar {\rm Op})(X_{\boldsymbol{\lambda}})$ where partition $\lambda = (3,2)$.
\begin{figure}
\includegraphics[scale=0.4]{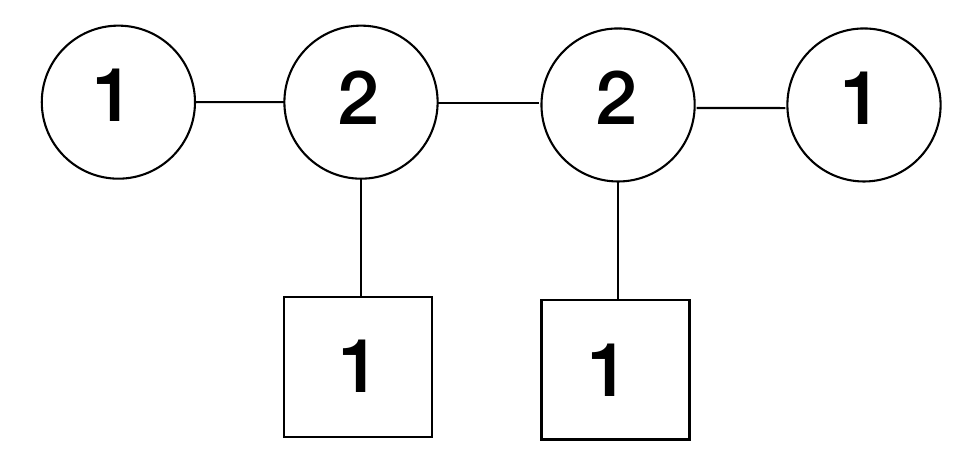}
\caption{Quiver $X_{(3,2)}$}
\end{figure}
In other words: 
$$
{\rm Fun}(\hbar {\rm Op})(X_{(3,2)})\simeq \frac{{\rm Fun}(\hbar {\rm Op})(F\mathbb{F}l_5)}{(\ref{eq:degeneration5})}\,.
$$

\section{$A_r$-Quiver Varieties, 3d Mirror Symmetry and Branes}\label{Sec:KThe}
The goal of this section is to formulate the 3d mirror symmetry concept for $A_n$ quiver varieties. 

First we explain what quiver varieties are and define their quantum equivariant K-theory rings. Next we formulate the concept of the 3d mirror symmetry as a pair of `mirror dual' quiver varieties with given properties and isomorphic quantum K-theory rings. 

The next subsection is devoted to the physical approach to 3d mirror symmetry and combinatorial calculations involving D-branes, which allow to compute a ``mirror'' for a given quiver variety. 

Our main goal is to show that quiver varieties associated to $X_{k,l}$ are self-dual. We will formulate what it means on the level of electric and magnetic frame tRS coordinates, which we introduced. We will prove self-duality in the next section.

\subsection{$A_r$ Quiver Varieties, Equivariant Quantum K-theory and 3d Mirror Symmetry} 

\subsubsection{Definition of $A_r$ quiver varieties}
Consider the quiver $Y_{{\bf v}, {\bf w}}$. A representation of $Y_{{\bf v}, {\bf w}}$ is a set of vector spaces $V_i,W_i$, where $V_i$ correspond to original vertices, and $W_i$ correspond to their copies, together with a set of morphisms between these vertices, corresponding to edges of the quiver. Let $R=\text{Rep}(\mathbf{v},\mathbf{w})$ denote the linear space of quiver representation with dimension vectors $\mathbf{v}$ and $\mathbf{w}$, where $\mathbf{v}_i=\text{dim}\ V_i$, $\mathbf{w}_i=\text{dim}\ W_i$. Then the group $G=\prod_i GL(V_i)$ acts on this space in an obvious way. As any cotangent bundle, $T^*R$ has a symplectic structure. This action of $G$ on this space is Hamiltonian with moment map $\mu :T^*R\to \text{Lie}{(G)}^{*}$. Let $L(\mathbf{v},\mathbf{w})=\mu^{-1}(0)$ be the zero locus of the moment map.

The Nakajima variety $X$ corresponding to the quiver is an algebraic symplectic reduction
$$
Y=L(\mathbf{v},\mathbf{w})/\!\! /_{\theta}G=L(\mathbf{v},\mathbf{w})_{ss}/G,
$$
depending on a choice of stability parameter $\theta\in {\mathbb{Z}}^I$ (see e.g. \cite{Ginzburg:} for a detailed definition). The group
$$\prod_i GL({W}_i)\times \mathbb{C}^{\times}_\hbar$$
acts as automorphisms of $X$, coming form its action on $\text{Rep}(\mathbf{v},\mathbf{w})$. Here $\mathbb{C}_{\hbar}$ scales cotangent directions with weight $\hbar$ and therefore symplectic form with weight $\hbar^{-1}$. Let us denote by $\mathsf{T}$ a maximal torus of this group. Choosing the set of characters, we describe the maximal torus of $GL({W}_i)$  as $\prod^{{\bf w_i}}_{k=1}\mathbb{C}^{\times}_{a_{i,k}}$.

A nontrivial example here is the partial flag quiver. The related quiver variety is the cotangent bundle to the partial flag variety (for details, see e.g. \cite{Koroteev:2017aa}).  

\subsubsection{Description of the quantum equivariant K-theory}

We now will use this data to produce the equivariant quantum K-theory ring. 
First, let us remind the construction of the  classical equivariant K-theory ring $K_{\mathsf{T}}(Y)$. For a Nakajima quiver variety $Y$, withe the set of vertices $I$ one can define a set of tautological bundles on it $\mathcal{V}_i, \mathcal{W}_i, i\in I$, as bundles constructed by applying the associated bundle construction to the $G$-representations $V_i$ and $W_i$. It follows from this construction, that  all bundles $W_i$ are topologically trivial. Tensorial polynomials of these bundles and their duals generate $K_{\mathsf{T}}(Y)$ according to Kirwan's surjectivity theorem, which is recently proven in \cite{McGerty:2016kir}. 
One can represent the corresponding K-theory classes as diagonal operators acting in the basis of fixed points in the localized equivariant K-theory $K^{loc}_{\mathsf{T}}(Y)$. 

The quantum equivariant K-theory is a deformation of this algebra by means of $r$ new variables, known as K\"ahler parameters. There are several ways of constructing this deformation using the tools of enumerative geometry. The first one, obtained following the traditional approach using the counts of holomorphic maps is due to Givental and Lee \cite{2001math8105G}, see also \cite{Givental:2015e,Givental:2015d,Givental:2015f,Givental:2015b,Givental:2015g,Givental:2015c,Givental:2015,Givental:2015a,Givental:2017b,Givental:2017a,Givental:2017} for recent developments in permutation invariant quantum K-theory. The second one is specific to the setup of quiver varieties and uses the counts of quasimaps \cite{Koroteev:2017aa}, \cite{Okounkov:2015aa}: this is the approach we address in this paper. 
 
In \cite{Koroteev:2017aa}  we defined quantum equivariant K-theory of $Y$ as the so-called Bethe algebra, using the symmetric functions of Bethe roots. Here we give an equivalent definition. 

\begin{Def}
The quantum equivariant K-theory $K^{q}_{\mathsf{T}}(Y_{\bf{v},\bf{w}})$ of quiver variety associated with quiver $Y_{\bf{v},\bf{w}}$ is the algebra ${\rm Fun}({\hbar\rm Op})(Y_{\bf{v},\bf{w}})$  with regular 
singularities described by $\{\Lambda_i(z)=\prod^{{\bf w}_i}_{k=1}(z-a_{i,k})\}_{i=1,\dots, r}$ so that $\{a_{i,k}\}$ are the characters of the maximal torus $T=\prod^r_{i=1}\prod^{\bf w_i}_{k=1}\mathbb{C}^{\times}_{a_{i,k}}\times \mathbb{C}_{\hbar}$ and $Z$-twist is related to the K\"ahler parameters $\xi_1, \dots , \xi_r$ as follows: 
$Z=\prod^r_{i=1}\xi_i^{\check{\alpha}_i}$.
\end{Def}

Here the Bethe root variables parametrizing the Miura $(SL(r+1), \hbar)$-oper correspond to the eigenvalues of the quantum tautological classes in the deformation of the basis of fixed points. 

\begin{Rem} We defined $K^{q}_{\mathsf{T}}(Y_{\bf{v},\bf{w}})$ as localized with respect to K\"ahler parameters $\zeta_i$. 
Without localization, in the limit $\{\zeta_i\rightarrow 0\}$ one obtains the classical equivariant K-theory $K_T(Y_{\bf{v},\bf{w}})$. 
\end{Rem}

Let us be more concrete about the relation between quantum K-theory generators and the $QQ$-systems. 
We have shown in \cite{Koroteev:2017aa,Pushkar:2016qvw} that the $Q^i_+(z)$-functions are actually the eigenvalues of the generating functions for the quantum K-theoretic classes corresponding to exterior powers of $\mathcal{V}^i$: $${\bf Q}^i(z)=\sum^{\bf v_i}_{k=1}z^k\Lambda^{n-k}{\mathcal{V}^i}^{(q)}$$
acting as operators in the space of $K_T^{loc}(Y)$. This way, elementary symmetric functions of Bethe roots of degree $k$ are the eigenvalues of the quantum K-theory classes $\Lambda^{k}\mathcal{V^i}^{(q)}$, thus identifying  ${\bf Q}^i(z)$ with the Baxter operators acting in the appropriate weight section of $SL(r+1)$ XXZ spin chain isomorphic to the vector space $K_T^{loc}(Y)$. 
At the same time, $\Lambda_i$-functions are the products of Drinfeld polynomials defining the Hilbert space of the $XXZ$ spin chain and serve as generating functions of the exterior powers of the trivial bundles $\mathcal{W}_i$, capturing the data of equivariant parameters $\{a_i\}$ associated to torus $T$.

\subsubsection{Quantum B\"acklund Transformations and 3d Mirror Symmetry}
We already noticed that there is $n!$ isomorphic algebras ${\rm Fun}(\hbar {\rm Op})(Y_{\bf{v},\bf{w}})$ corresponding to 
various Miura $\hbar$-opers for a given $\hbar$-oper. The isomorphism is given by B\"acklund transformations. 
One may wonder what that means on the level of quiver varieties. It turns out the following Proposition holds.

\begin{Prop}
Quivers $Y_{\bf{v},\bf{w}}$ $Y_{\bf{v}',\bf{w}}$, with parameters related by B\"acklund transformations produce isomorphic 
algebras $K^{q}_{\mathsf{T}}(Y_{\bf{v},\bf{w}})\cong K^{q}_{\mathsf{T}}(Y_{\bf{v}',\bf{w}})$. The corresponding quiver varieties 
can be produced from a single quiver by changing the stability parameter $\theta$. 
\end{Prop}
\begin{Rem}
 A classic example of this statement is the case of $T^*Gr(k,n)$ corresponding to the quiver with 
one vertex carrying label $k$ and framing with label $n$. The  quantum B\"acklund transformation sends it to 
$T^*Gr(n-k,n)$. On the level of $QQ$-system that corresponds to the transformation $Q_-(z)\rightarrow Q_+(z)$ and $Z\rightarrow Z^{-1}$. Thus in the nonlocalized quantum K-theory, $$\lim_{Z\to 0}K^{q}_{\mathsf{T}}(T^*Gr(k,n))=K_{\mathsf{T}}(T^*Gr(k,n)), \quad \lim_{Z\to \infty}K^{q}_{\mathsf{T}}(T^*Gr(k,n))=K_{\mathsf{T}}(T^*Gr(n-k,n)),$$
where by limits we understand keeping the terms in the algebra which remain finite with respect to that limit. 
In general, one picture the ring  $K^{q}_{\mathsf{T}}(Y_{\bf{v},\bf{w}})$ as a polyhedron, symbolizing the moduli space of parameters $\zeta_i$. At the faces we put classical equivariant $K$-theories of quiver varieties corresponding to quantum B\"acklund transformations applied to the quiver $Y_{\bf{v},\bf{w}}$. These classical equivariant $K$-theories emerge by sending part of the K\"ahler parameters to $0$ and $\infty$ correspondingly. 
\end{Rem}

Now we are ready to formulate the statement of 3d mirror symmetry.  To do that, it is convenient to use $\xi_i$ variables for parametrization of $Z$-twist/K\"ahler parameters. We will call them modified K\"ahler parameters.

\begin{Def}\label{mirdef}
Two quiver varieties $Y$, associated to quiver $Y_{\bf{v},\bf{w}}$ and $Y^!$ associated to $Y^!_{\bf{v}^!,\bf{w}^!}$ form a 3d mirror pair if 
\begin{eqnarray}
K^{q}_{\mathsf{T}}(Y_{\bf{v},\bf{w}})\cong K^{q}_{\mathsf{T}}(Y^!_{\bf{v}^!,\bf{w}^!})
\end{eqnarray}
so that there is 1-to-1 correspondence between the sets of equivariant and modified K\"ahler parameters: 
\begin{eqnarray}
\{\xi_i\}_{i=1,\dots,r}\leftrightarrow \{a^!_{i,k}\}_{i=1,\dots, r^!; k=1,\dots, {\bf w}^!_i}; \nonumber \\
\{\xi^!_i\}_{i=1,\dots,r^!}\leftrightarrow \{a_{i,k}\}_{i=1,\dots, r; k=1,\dots, {\bf w}_i}, \quad \hbar\to \hbar^!.
\end{eqnarray}
\end{Def}

We remark here that as before, we did not specify the order of parameters $\xi_i$ in the correspondence, since we are talking about the classes of quiver varieties (each related to Miura $\hbar$-opers) corresponding to a given $\hbar$-oper. 

It is not a priori clear that a dual for such quiver variety exists. 3d-mirror symmetry (and related to it \textit{symplectic duality}) is a conjecture which claims that for every symplectic variety, there exists a dual variety. Note, however, that the dual variety does not have to be a Nakajima quiver variety (albeit in the current paper it always will be). See \cite{Aganagic:2016aa,Rimanyi:2019aa,Dinkins:2019pwj,Kamnitzer:2022aa} for more details and references.

Our main goal is to show that $X_{k,l}$ varieties are self-dual as full flag varieties are \cite{Koroteev:2017aa}. Notice that most of our methods, which have to do with the transition to tRS variables can be expanded further for more general $A_r$ quivers. That will be explained in the next section. In this section, we will outline the D-brane approach, which provides a useful tool for finding such mirror pairs.

\subsection{D-branes, Quivers, Linking Numbers and 3d mirrors}

\subsubsection{D-branes and quivers}

The notion of 3d mirror symmetry in $\mathcal{N}=4$ supersymmetric gauge theories appeared in physics literature \cite{Intriligator:1996ex} and can be universally understood as a manifestation of S-duality transformation in Type IIB string theory \cite{Hanany:1996ie}. Below we shall only review the necessary part of the material which is required for this paper.

We shall only use the formal prescription itself without reviewing its physical meaning: an interested reader is welcome to consult the original physics literature, in particular \cite{Gaiotto:2008ak,Gaiotto:2013bwa}. 

Let us introduce the main characters in the corresponding D-brane arithmetic, which describe the Nakajima varieties in type $A$ and their 3d duals.

Superstring theory contains strings and branes in ten spacetime dimensions. The latter are extended objects on which open strings and other branes may end. In order to engineer a quiver variety we will need:
\begin{itemize}
\item NS5 branes, drawn as vertical lines
\item D3 branes, drawn as horizontal line segments connecting NS5 branes
\item D5 branes expressed as red circles between NS5 branes
\end{itemize}
The resulting picture can be viewed as a projection on the two-dimensional space of this system of three and five-dimensional objects. 
Let us first describe how this system defines a framed quiver together with its enumerative data, namely equivariant parameters, K\"ahler parameters. 
\begin{itemize}
\item rank of quiver= number of NS5 branes +1. One can associate with the vertex of the quiver the space between NS5 branes.
\item integral label $\bf v_i$ = number of D3-branes between the $i$-th and $(i+1)$-st NS5 branes 
\item  integral label $\bf w_i$ for the framing for a given vertex = number of D5 branes between $i$-th and $(i+1)$-st NS5 branes
\item equivariant parameters are associated with the horizontal coordinate of D5 branes
\item K\"ahler parameters correspond to the `distance' parameters between pairs of NS5 branes  
\end{itemize}
Such diagrammatic prescription and the way how it defines quiver variety is due to Hanany-Witten \cite{Hanany:1996ie}. 
One can look at the first and the second picture in \figref{fig:QuiverBranes}) for the quiver $X_{2,4}$ with its brane interpretation. 

\begin{Rem}
As a matter of fact, 
all branes coincide with each other along two spatial directions and the time direction. Thus in the field theory limit, this brane construction engineers a supersymmetric 3-dimensional gauge theory with eight supercharges. The matter content of this theory, and therefore the data of the quiver variety (which serves as a description of its Higgs vacua) is described by the mutual orientation of branes in the other seven directions. For reference, D3 branes occupy directions 0123, D5 lie in 012456, while NS5 branes span 012789. In the two-dimensional figure, the third direction is horizontal, while the vertical direction is one of 789. D5 branes are thus points (red circles) on this two-dimensional plane.
\end{Rem}
The linear algebra data of a quiver variety is encoded via open strings which end on various branes. Open strings may not end on an NS5 brane, however, on other branes, like D3 branes, can. The dictionary between the quiver variety data and the branes works as follows:
\begin{center}
Hom$(V_i,W_i)$\qquad \includegraphics[trim=2cm 3cm 0 0, scale=0.38]{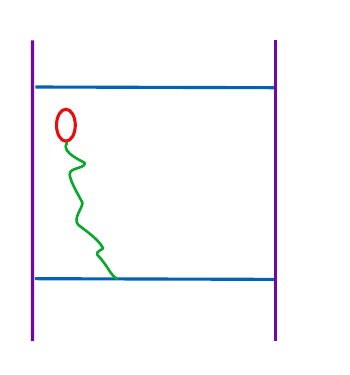} Hom$(V_i,V_{i+1})$\qquad   \includegraphics[trim=2cm 2.5cm 0 0, scale=0.4]{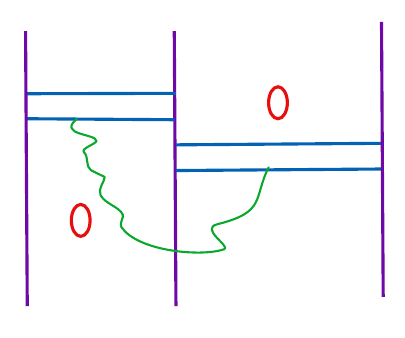} \qquad Hom$(V_i,V_{i})$\qquad  \includegraphics[trim=2cm 2.3cm 0 0, scale=0.45]{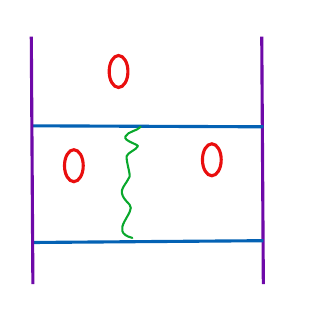}
\end{center} 

\vskip.4in

In other words, in order to describe $Hom(A,B)$ the left end of the string needs to be attached to a brane which corresponds to space $A$, while its right end to be attached to the brane which stands for space $B$. Since strings are oriented $Hom(A,B)$ is distinguished from $Hom(B,A)$.

\subsubsection{Linking numbers and 3d Mirror map}

In order to define the 3d mirror map we have to to introduce the following concept of linking numbers.

\begin{Def}
The linking number $r_i^!$ for an NS5 brane can be defined as 
the number of D3 branes ending from the right, plus the number of D5 branes on the left of it, minus the number of D3 branes emerging from the left
\begin{equation}
r_i^!=\# D3(R)-\# D3(L)+\# D5(L)\,.
\label{eq:linkNS5}
\end{equation}
The linking number $r_i$ for a D5 brane is defined as
\begin{equation}
r_i=\# D3(L)-\# D3(R)+\# NS5(R)\,.
\end{equation}
\end{Def}

The precise positions of the branes along the horizontal direction in \figref{fig:QuiverBranes} do not affect the quiver data, but the relative order of the NS5 branes is important. A D5 brane and a NS5 brane can be transported across each other 
only if the number of D3 branes in the system is changed appropriately, in order to keep constant the {\it linking numbers} of each NS5 brane.  The linking numbers can be realized through the movement of D5 branes to the to the right. When a D5 brane passes through an NS5 brane a D3 brane (horizontal segment) is created. One can then interpret $r_i^!$ as a linking number associated to an NS5 brane, counting the number of D3 branes that end on it on the right minus the number of D3 branes ending on it on the left.  
The number $r_i$ could be thought of as a linking number associated with a given D5 brane, counting the number of D3 branes attached to it upon moving this D5 brane to the very right. 

The two sets of linking numbers $\lambda^!$ of NS5 branes and $\lambda$ of D5 branes can be thus arranged as a set of natural numbers defining the Young tableaux.

\vskip.1in

\noindent {\bf Example.}  As an example consider again the family $X_{k,l}$ for $k=2$ and $r=4$.  When a D5 brane passes through an NS5 brane a D3 brane (horizontal segment) is created as in the bottom picture in \figref{fig:QuiverBranes}. This quiver gives rise to two sets of linking numbers for D5 branes $\lambda=(4,3,2,1,1,1)$ 
{\color{orange} {\tiny$\yng(1,2,3,6)$}}
 and for NS5 branes $\lambda^!=(4,3,2,1,1,1)$ {\color{green} {\tiny$\yng(1,2,3,6)$}}. For the self-dual family $X_{k,l}$ these sets are identical. For a general type A quiver variety there will be two distinct Young diagrams $\lambda$ and $\lambda^!$, which are interchanged under the 3d mirror symmetry, as discussed below.
\begin{figure}[!h]
\includegraphics[scale=0.5]{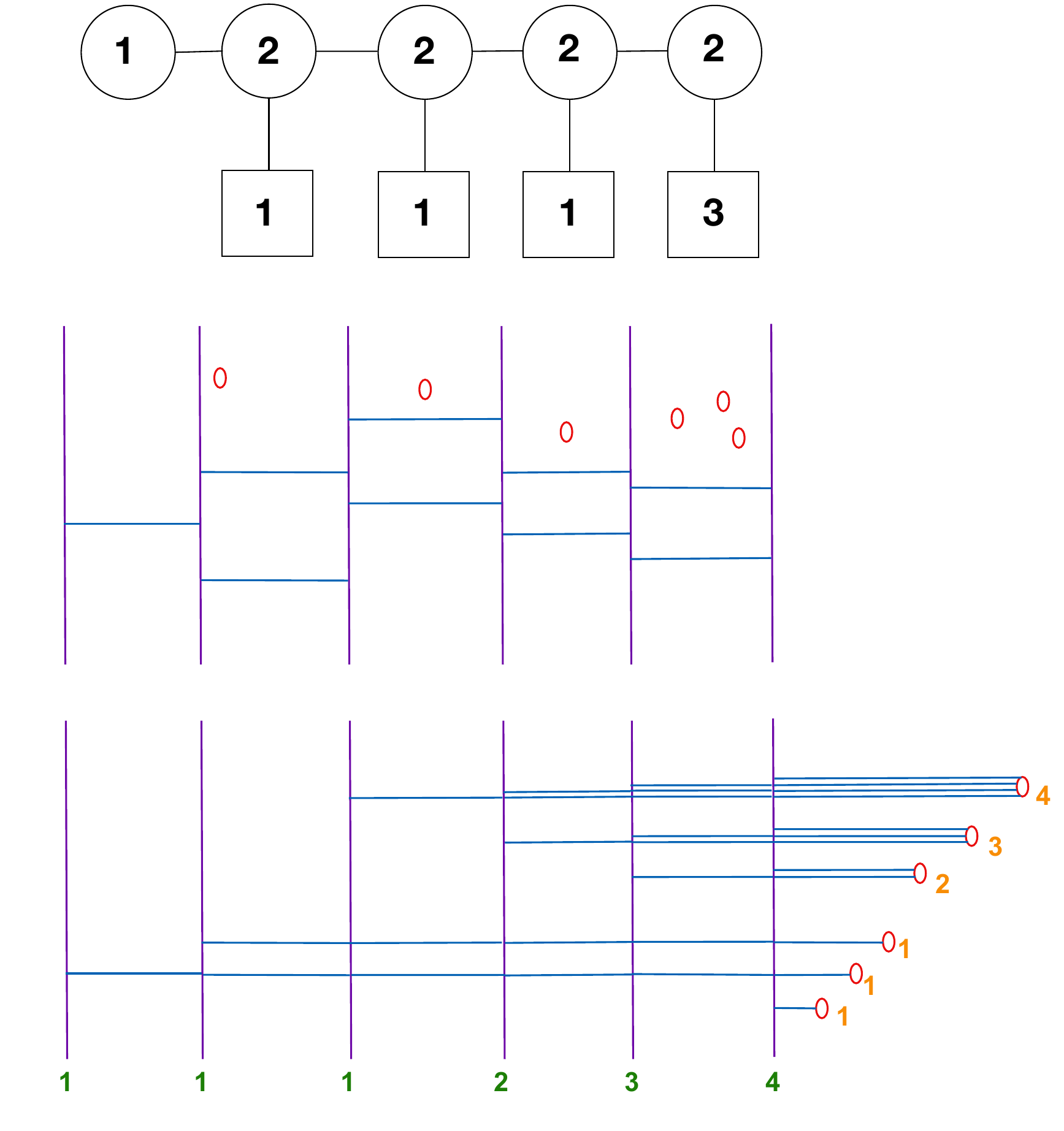}
\caption{Quiver $X_{2,4}$ (top) and its brane presentations, its Hanany-Witten brane description (middle), and its description in terms of 4d BPS boundary condition (bottom).}
\label{fig:QuiverBranes}
\end{figure}

\vskip.1in

Given the definition of linking numbers, we obtain that the sum of all NS5 linking numbers equals the sum of all D5 linking numbers. 
The linking numbers of the neighboring NS5 branes are related in the following way:
\begin{equation}\label{eq:weights}
r_{i}^!-r_{i+1}^! = {\bf w}_i + {\bf v}_{i-1} + {\bf v}_{i+1} - 2 {\bf v}_i\,,
\end{equation}
and are thus positive and nondecreasing to the left from $r_{L^!+1} ={\bf v}_{L^!}$ if the right-hand side of the above relation is nonnegative: this is automatically satisfied if the corresponding quiver variety exists. We also assume ${\bf v}_0={\bf w}_0={\bf v}_{L^!+1}={\bf w}_{L^!+1}=0$. In other words, by starting from the right-most NS5 brane we can calculate all NS5 linking numbers by iteratively applying \eqref{eq:weights}.


For a generic Type A quiver varieties the mirror map works as follows.
Let $X$ be a quiver variety of rank $L^!-1$ and $X^!$ have rank $L-1$ with labels
\begin{equation}\label{eq:XX1labels}
{\bf v}_1,\dots,{\bf v}_{L^!-1},\, {\bf w}_1,\dots,{\bf w}_{L^!-1}\,,\qquad
{\bf v}^!_1,\dots,{\bf v}^!_{L-1},\, {\bf w}^!_1,\dots,{\bf w}^!_{L-1}
\end{equation}
respectively. In order to express the dual labels via the original labels we first convert them into the linking numbers parameterized by two Young diagrams $\lambda$ and $\lambda^!$, which are then simply exchanged by the 3d mirror symmetry. 
%
The linking numbers for the D5 branes are given by $r_a =L^!-i$, where $a=1,\dots L^!$ and $i=1,\dots, L-1$ if the $a$-th D5 brane is positioned between $i$th and $(i+1)$st NS5 branes in the brane picture. We can order the labels $r_a$ so that they are non-decreasing so they form columns of a Young diagram.
The second set of linking numbers $r^!_i$ reads
\begin{equation}\label{eq:Xlabelsr}
r^!_i = {\bf v}_1\,,\quad r^!_i={\bf w}_i + {\bf v}_{i-1} + {\bf v}_{i+1} -2{\bf v}_{i} \,.
\end{equation}
In the mirror frame, we interchange the linking numbers
\begin{equation}\label{eq:X1labelsr1}
r_i = {\bf v}^!_a\,,\quad r_a={\bf w}^!_i + {\bf v}^!_{i-1} + {\bf v}^!_{i+1} -2{\bf v}^!_{i} \,,
\end{equation}
and $r^!_i = L -a$ if the $i$th D5 brane belongs to interval between $a$th and $(a+1)$st NS5 branes
which can be used to compute the labels of the dual quiver variety $\{{\bf v}^!,{\bf w}^!\}$.

\vskip.1in

At this moment the reader may find an obvious similarity of the NS5 branes linking numbers $r_i^!$ and degrees of the components $\rho_i$ of the section $s(z)$ in the oper magnetic frame \eqref{eq:weights0}.

%

\begin{Prop}
The linking numbers of the NS5 branes for quiver $X$ are equal to the degrees of the components of the section on the $\hbar$-oper bundle describing ${\rm Fun }(\hbar{\rm Op})(X)$.
\end{Prop}

\vskip.1in
\subsubsection{3d Mirror Symmetry for Quiver Varieties}

We are now ready to define the 3d mirror for quiver $X$ and formulate the main theorem about 3d mirror duals for a generic type A quiver variety $Y_{\bf v, \bf w}$.

\begin{Def}\label{eq:DefinitionMirror}
Let A-type quiver varieties $X$ and $X^!$ have labels as in \eqref{eq:XX1labels}. Then we say that  $X$ and $X^!$ are 3d mirror dual to each other if the labels satisfy \eqref{eq:X1labelsr1} and \eqref{eq:Xlabelsr}.
\end{Def}

The following Theorem will be proven in the next section.

\begin{Thm}
For all quiver varieties of type $A$ associated with quiver $Y_{\bf v, \bf w}$, which satisfy the condition ${\bf v}_{i-1}+{\bf v}_{i+1}+{\bf w}_{i}\geq 2{\bf v}_{i}$, quiver variety associated to quiver $Y^!_{\bf v, \bf w}$ is its 3d mirror dual in the sense  of the definition \ref{mirdef}.
\end{Thm}

\vskip.1in

\begin{Rem}{\it Alternative Definition via Stable Envelopes. }
In recent literature, an alternative definition of the 3d mirror symmetry was given. 

We say that quiver varieties $X$ and $X^!$ are 3d mirror dual to each other if:
\begin{enumerate}
\item There is an isomorphism between K\"ahler and equivariant tori
\begin{equation}
A\xrightarrow{\sim} K^!\,,\qquad K\xrightarrow{\sim} A^!\,,\qquad \mathbb{C}^\times_{\hbar} \xrightarrow{\sim} \mathbb{C}^\times_{\hbar^!};
\end{equation}
\item There is a bijection between the fixed points
\begin{equation}
X^A\xrightarrow{\sim} \left(X^!\right)^{A^!}\,.
\end{equation}
\end{enumerate}

Additionally one may require the existence of the line bundle (Mother function) $\mathfrak{M}$ over scheme Ell$_{T\times T^!}(X\times X^!)$ which reduces to elliptic stable envelopes on both $X$ and $X^!$. We will not study elliptic stable envelopes here and refer the interested reader to \cite{Rimanyi:2019aa}.
\end{Rem}


\subsubsection{Resonance Conditions and Relations in $QK_T(X^{\boldsymbol{\lambda}})$ from Branes}
In this section, we discuss how the resonance conditions from Section 4, i.e. \eqref{eq:res1}, which we used to degenerate the quantum K-theory ring $X_{\lambda}$ to $X_{k,l}$, can be illustrated using branes using examples in \figref{fig:hwmoves, fig:hwmoves2}. 

Consider $X^{\boldsymbol{\lambda}^{!}}$ on the left of \figref{fig:hwmoves} in which case $\boldsymbol{\lambda}^{!}=(2,2,1,1)$. Let us impose $a_1 \hbar=s_{3,1}= a_2 $. According to \figref{fig:PartialFlag2} and the equations in that subsection quiver $X^{\boldsymbol{\lambda}^{!}}$ transforms into $X^{\boldsymbol{\lambda}^{!}}_{\boldsymbol{\lambda}}$ depicted on the right of \figref{fig:hwmoves} where $\boldsymbol{\lambda}=(2,1,1,1,1)$.

In terms of branes the following phenomenon is occurring (see \cite{Gaiotto:2013bwa} for more details). The two D5 branes and one D3 brane align in the same position in the vertical direction. Then the segment of the D3 branes which is stretched between these two D5 branes can be `pulled' out towards the reader in the direction perpendicular to the picture off to infinity. According to Type IIB string theory such an operation does not change the spectrum of open strings which end on these branes. After the intermediate segment of the D3 brane is removed one can move one D5 brane past the NS5 brane on the left and another D5 brane past the NS5 brane on the right. Due to the Hanany-Witten phenomena, the pieces of D3 branes connecting the above NS5 and D5 branes disappear and we obtain the brane picture on the right in \figref{fig:hwmoves}.

As we discussed earlier, this step can be repeated further if we impose the condition $\hbar^2 a_1 = \hbar s_{r-1,1} = s_{r,2} = a_3$, cf. \eqref{eq:res2} (see \figref{fig:hwmoves2}). In this case, we need to align two D3 branes and two D5 branes along the vertical direction and then `pull out' two segments of D3 branes stretched between one D5 and the NS5 brane and another D5 and the NS5 brane between them. The result of this operation is quiver $X^{\boldsymbol{\lambda}^{!}}_{\boldsymbol{\lambda''}}$ with $\boldsymbol{\lambda''}=(3,1,1,1)$.

\begin{figure}[!h]
\includegraphics[scale=0.48]{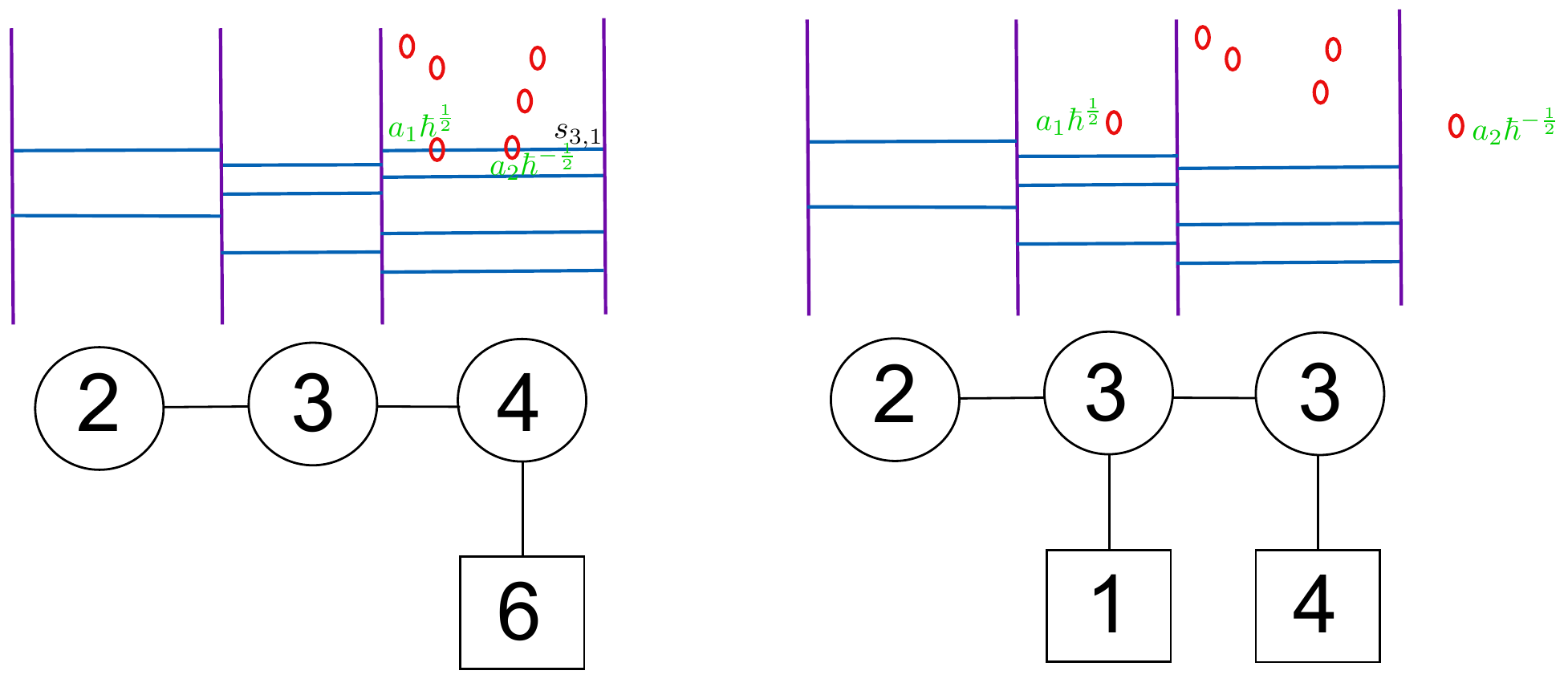}
\caption{Imposing resonance condition $a_1 \hbar=s_{3,1}= a_2 $ on quiver variety with labels ${\bf v}=(2,3,4)\,, {\bf w}=(0,0,6)$.}
\label{fig:hwmoves}
\end{figure}

\begin{figure}[!h]
\includegraphics[scale=0.6]{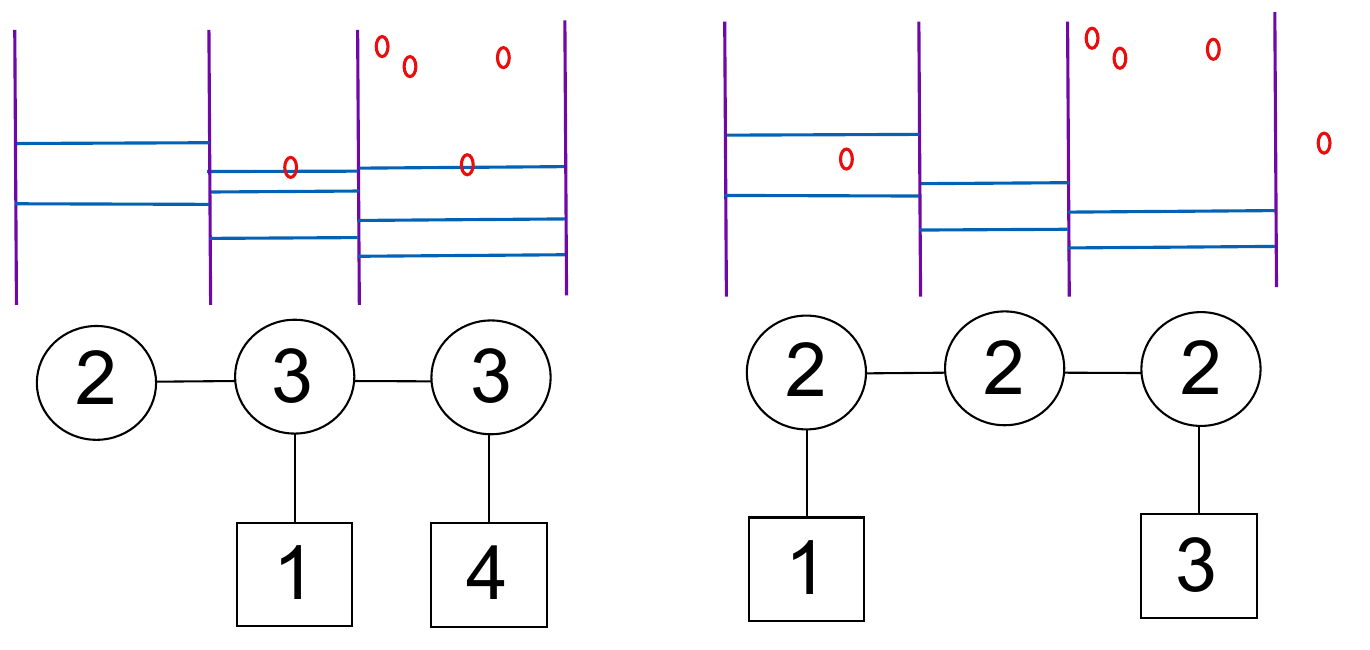}
\caption{Imposing resonance condition $\hbar^2 a_1 = \hbar s_{r-1,1} =s_{r,2} = \hbar a_3$ on quiver variety with labels ${\bf v}=(2,3,3)\,, {\bf w}=(0,1,4)$.}
\label{fig:hwmoves2}
\end{figure}


\section{Mirror Maps via Electric and Magnetic Frames}\label{Sec:MirrorElMag}

\subsection{tRS variables and $A_r$ quivers}

To summarize the results of the previous sections, now we produced two sets of variables for the quantum equivariant $K$-theory ring for Nakajima varieties. 

We called the first set the ``oper magnetic frame coordinates', they are related to the parametrization on Z-twisted Miura $(SL(r+1),\hbar)$-opers corresponding to a given quiver via the section of the line bundle. The roots of the polynomials $s_i(z)$, determining the line bundle, denoted by $\{p_i^k\}^{k=1,\dots, \deg(s_i)}_{i=1,\dots, r+1}$ determine a subvariety in the space of variables $p_i^k$, parametrized by $\{\xi_i\}$, $\{a_i\}$. The quantum equivariant K-theory algebra can be interpreted as the algebra of functions on this family of subvarieties. 

In a particular case of the full flag quiver (this is the case when $\deg(s_i)=1$), this family of subvarieties turns out to be a family of intersections of two Lagrangian subvarieties in the symplectic space with the coordinates $\{\chi_i,p^{\chi}_i\}$ with a canonical 2-form. One of them is given by setting $\chi_i=\xi_i$ and another is given by setting tRS Hamiltonians to be equal to the elementary symmetric functions of equivariant parameters. 

Clearly, one would like to construct a similar set of coordinates for a general quiver.  We will call such a set of coordinates  {\it a true magnetic frame}.  The key to constructing it is through the ``dual" system of coordinates which we called an electric frame, directly related to tRS model. 
 
For the electric frame, we started with a partial flag quiver. In this case, we established the isomorphism between the algebra of functions on the space of Miura $(SL(r+1),\hbar)$-opers and the algebra of functions on 
the intersection of two Lagrangian subvarieties in the symplectic space with coordinates $\{\rho_i, p^{\rho}_i\}_{i=1,\dots, L}$. where the first one is determined by 
$\{\rho_i=a_i\}_{1,\dots, L}$   and the 
second one is deterimined by the tRS Hamiltonians set to be equal to the elementary symmetric functions  of 
$\{\chi_i=\xi_{k_i}\hbar^{l_i}\}_{i=1,\dots,L}$, 
where $\{k_i, l_i\}$ are determined by partial flag labels. Of course, this is a degeneration of the algebra corresponding the full flag quiver, where we imposed the condition $\chi_i=\hbar^{k_{ij}}\chi_j$ for some integers $k_{ij}$. 

Remarkably, one can reproduce this tRS realization of the equivariant quantum K-theory by means of the recursive procedure of Section 4 by imposing degeneration conditions on the $\{a_i\}$ parameters. 

Given all this information, we will define a map from electric frame to a true magnetic frame using the explicit structure of tRS model (as was suggested in Sections 3 and 4 of \cite{Gaiotto:2013bwa}).

\subsection{tRS variables, the Mirror Map and the True Magnetic Frame}
As we know, the Lax matrix of the tRS model admits two different realizations -- it can be either matrix $M$  or matrix $T$ in \eqref{eq:FlatConNew}. This can be achieved by diagonalization of each of the matrices so we can pick a $g$ such that $M$ is diagonal with eigenvalues $\chi_1,\dots, \chi_N$ or $T$ diagonal with eigenvalues $\rho_1,\dots,\rho_N$. Assume for now that $\rho_i\neq \hbar^{\mathbb{Z}}\rho_j$ for $i\neq j$. As we discussed above, upon diagonalization of $M$, $T$ can be written in the form of Lax matrix: \eqref{eq:LaxFullFormula}. Notice that the same can be done for $M$ matrix, the only difference is that one has to exchange $\hbar\rightarrow \hbar^{-1}$. Let us reformulate this nontrivial relation in the following way.

Consider the symplectic spaces $\mathcal{M}^e$  with coordinates $(\rho_i, p^{\rho_i})$ and $\mathcal{M}^m$
with coordinates $ (\chi_i, p^{\chi_i})$. The tRS relation between $M$, $T$ matrices produces a Lagrangian subvariety $\mathcal{L}^e_{\chi}\subset\mathcal{M}$ described by setting the tRS Hamiltonians to be the symmetric functions of $\chi$ variables. On the other hand, it produces a Lagrangian subvariety $\mathcal{L}^m_{\rho}\subset\mathcal{M}$ with tRS Hamiltonians under the transformation $\hbar\to \hbar^{-1}$.

\begin{Lem}\label{Th:LemmatRS}
The transformation of the tRS system with 
$$
\det(u-T)=\prod^N_{i=1}(u-a_i)\,,\qquad \det(u-M)=\prod^N_{i=1}(u-\xi_i)\,,
$$ 
which maps
\begin{equation}\label{eq:tRSMirror}
i^{em}: M\to T\,,\qquad T\to M\,,\qquad \hbar\to \hbar^{-1}\,.
\end{equation}
corresponds to the following symplectic map 
$$
\mathcal{M}^e\to \mathcal{M}^m:\quad (\rho_i, p^{\rho}_i)\rightarrow (\chi_i, p^{\chi}_i)\,,
$$ 
which produces a one-to-one correspondence between the intersections of the pairs of Lagrangian subvarieties:
\begin{equation}
i^{em}: \{\rho_i=a_i\}\cap\mathcal{L}^e_{\xi}\to \{\chi_i=\xi_i\}\cap \mathcal{L}^m_{a}\,.
\end{equation}
\end{Lem}

We refer to $ (\chi_i, p^{\chi}_i)$ as electric frame for tRS system and $ (\rho_i, p^{\rho}_i)$ as its magnetic frame we call the map $i^{em}$ the {\it electric-magnetic map}. 
Notice, that we already established the isomorphism of Lemma \ref{Th:LemmatRS} in a different manner, when we discussed the ring $K_T(F\mathbb{F}l_L)$ in its electric an oper magnetic frames formulation, which coincide with the electric and magnetic frames of tRS systems.  
As a consequence we obtain the following statement, previously discussed in \cite{Gaiotto:2013bwa,Koroteev:2017aa}:

\begin{Thm}
The contangent bundle to the full flag variety is 3d Mirror self-dual.
\end{Thm}

\begin{Rem}
Consider the product of two $N$-body tRS model phase spaces $\mathcal{M}\times\mathcal{M}^!$. Recall that tRS momenta can be obtained from the XXZ Yang-Yang function for full flag variety: $Y=Y(\{s_{i,k}\},\{a_i\},\{\xi_i\},\hbar)$, which depends on Bethe variables, which provides the relation 
\begin{equation}
{p_i^\xi} = \exp \frac{\partial Y}{\partial{\xi_i}}\,,\qquad{p_i^a} = \exp \frac{\partial Y}{\partial{a_i}}\,.
\end{equation}
It turns out the Yang-Yang function serves as a generating function on Lagrangian subvariety $\mathcal{L}\subset\mathcal{M}\times\mathcal{M}^!$ which is specified by the choice of eigenvalues of tRS Lax matrices $T$ and $M$ above with the symplectic form 
\begin{equation}
\Omega = \sum_{i=1}^N \frac{d p^\xi_i}{p^\xi_i}\wedge\frac{d \xi_i}{\xi_i} - \frac{d p^a_i}{p^a_i}\wedge\frac{d a_i}{a_i} 
\end{equation}
vanishes.
From this geometric viewpoint transition from $\mathcal{M}$ to the dual phase space $\mathcal{M}^!$ is a canonical transformation of type I.
\end{Rem}

We will describe mirror maps for quiver varieties from this statement by applying degeneration constraints on K\"ahler and equivariant parameters. Previously we described the recursive procedure how to effectively degenerate the electric frame version of the 
$K^q_T(F\mathbb{F}l_L)$ to produce $K^q_T(Y_{\bf{v},\bf{w}})$. The first step in that procedure is to degenerate it to the partial flag $X^{\boldsymbol{ \lambda}}$ by imposing the relations on K\"ahler parameters $\xi_i$. One can use the map $i$ from Lemma \ref{Th:LemmatRS} to produce what we call a `true magnetic frame' for the partial flag. Namely, we have

\begin{Prop}
Consider the electric frame formulation of $K^q_T(X^{\boldsymbol{ \lambda}})$, i.e. using matrix $T$ as a Lax matrix. Let ${\rm Fun}^{\bf{\lambda}}(\hbar{\rm Op})(F\mathbb{F}l_L)$ be the space of $Z$-twisted Miura $\hbar$-opers corresponding to the quiver $F\mathbb{F}l_L$  with $Z$-twist components given by the eigenvalues of the tRS matrix $M$ and regular singularities given by the equivariant parameters of $X^{\boldsymbol{ \lambda}}$. Then we have the following isomorphism:
\begin{eqnarray}
K^q_T(X^{\boldsymbol{ \lambda}})\cong {\rm Fun}^{\boldsymbol{\lambda}}(\hbar{\rm Op})(F\mathbb{F}l_L).
\end{eqnarray}
\end{Prop}
\begin{proof}
Indeed, let us apply map $i$ in the case of $M$ with these degenerate eigenvalues. It still works, since tRS Hamiltonians do not produce any singularities upon the degenerations produced by the electric frame for partial flag. Then through Wronskian 
realization in the Theorem \ref{Wronskian} we obtain the corresponding oper space $\hbar{\rm Op}(F\mathbb{F}l_L)$.
\end{proof}

We call the roots of the monomials of section $s$ in the Wronskian formulation of $\hbar{\rm Op}^{\mathbf \lambda}(F\mathbb{F}l_L)$ as {\it true magnetic momenta}. Then one can interpret the space ${\rm Fun}^{\bf{\lambda}}(\hbar{\rm Op})(F\mathbb{F}l_L)$ as the space of functions on the intersection of two Lagrangian subvarieties in the space with coordinates $(p^{\chi}_i, \chi_i)_{i=1,\dots, L}$ by setting $\chi_i$ to be equal to the eigenvalues of matrix $M$ for the electric frame formulation of $K^q_T(X^{\boldsymbol{ \lambda}})$ and the second is given by the tRS Hamiltonians set to be equal to the symmetric functions of the parameters of regular singularities.

Certainly, there is a question about how the oper and true magnetic momenta for the partial flag are connected with each other.
Consider the Wronskian matrix for $\hbar{\rm Op}^{\mathbf \lambda}(F\mathbb{F}l_L)\cong {\rm Wr}^{\mathbf \lambda}(F\mathbb{F}l_L)$. The $Z$-matrix is diagonal with the following eigenvalues
\begin{eqnarray}
\xi_{r+1},\dots,\xi_{r+1}\hbar^{L-\textbf{v}_r-1}, \xi_{r},\dots,\xi_{r}\hbar^{\textbf{v}_{r-1}-\textbf{v}_r-1},\dots, \xi_1\,\dots,\xi_1\hbar^{\textbf{v}_1-1}\,.
\end{eqnarray}
We can immediately construct the minors $\mathscr{D}_i$ \eqref{eq:Dkdef}. For instance, $\mathscr{D}_{{\bf v}_r+{\bf v}_{r+1}}$ is the determinant of the following matrix whose components are degree-one polynomials 
{\tiny
\begin{equation}\label{eq:MiuraQOperCond2}
\hspace*{-1.4cm}\begin{pmatrix} \,      \xi_{r}^{{\bf v}_{r-1}+{\bf v}_{r}-1}\hbar^{{\bf v}_{r-1}({\bf v}_{r-1}+{\bf v}_{r}-1)}\widetilde{s}_{L-{\bf v}_{r-1}-{\bf v}_{r}}(z) & \cdots & \xi_{r} \hbar^{{\bf v}_r}\widetilde{s}_{L-{\bf v}_{r-1}-{\bf v}_{r}}(\hbar^{{\bf v}_{r-1}}z) &\vdots & \cdots & \vdots \\ 
 \vdots & \vdots & \vdots & \vdots & \ddots &\vdots \\  
\xi_{r}^{{\bf v}_{r-1}+{\bf v}_{r}-1}\hbar^{{\bf v}_{r-1}({\bf v}_{r-1}+{\bf v}_{r}-1)}\widetilde{s}_{L-{\bf v}_r-1}(z) &\dots & \xi_{r} \hbar^{{\bf v}_r} \widetilde{s}_{L-{\bf v}_r-1}(\hbar^{{\bf v}_{r}-1}z) & \vdots & \cdots & \vdots \\  
\dots & \dots & \dots & \xi_{r+1}^{{\bf v}_r-1}\hbar^{{\bf v}_{r-1}} \widetilde{s}_{L-{\bf v}_r}(\hbar^{{\bf v}_{r-1}} z)  & \cdots & \widetilde{s}_{L-{\bf v}_r}(\hbar^{k-1}z)  \\
\vdots & \ddots & \vdots& \vdots & \ddots &\vdots \\
\cdots & \cdots& \cdots&\xi_{r+1}^{{\bf v}_r-1}\hbar^{{\bf v}_{r-1}({\bf v}_r-1)} \widetilde{s}_{L}(\hbar^{{\bf v}_{r-1}}z)  &\cdots & \widetilde{s}_{L}(\hbar^{k-1}z)  \,
\end{pmatrix} 
\end{equation}
}

This automatically guarantees that these minors have the same degrees as the corresponding minors for $\hbar{\rm Op}(X^{\boldsymbol{\lambda}})$.  At the same time, using the Lemma $\ref{Th:existencePoly}$, we can reconstruct polynomials $s_i(z)=c_i\prod_k(z-p^k_i)$ for the partial flag, e.g. the minor $\mathscr{D}_2$ for partial flag corresponding to the minor of true magnetic momenta above is as follows:
 \begin{eqnarray}
\mathscr{D}_2=\det\begin{pmatrix}
\xi_{r} \hbar s_{r-1}(\hbar z) & s_{r-1}(\hbar z) \\
\xi_{r+1} \hbar s_{r}(\hbar z) &  s_{r}(\hbar z) .
\end{pmatrix}
\end{eqnarray}

Thus we have the following statement.

\begin{Prop}\label{Prop:opertrue}
The isomorphism $${\rm Fun}^{\boldsymbol{\lambda}}(\hbar{\rm Op})(F\mathbb{F}l_L)\cong {\rm Wr}^{\boldsymbol{ \lambda}}(F\mathbb{F}l_L)\to {\rm Fun}(\hbar{\rm Op})(X^{\boldsymbol{\lambda}})\cong {\rm Wr}(X^{\boldsymbol{\lambda}})$$ 
is realized by identifying minors in the Wronskian formulation of both algebras as described above.
\end{Prop}

The generalization of this identification between minors and will be used below in the explicit formulation of the 3d Mirror map.

\vskip.1in

\vskip.1in
Now we are ready to formulate the 3d mirror map from this point of view. Let us denote the electric frame coordinates of $K^q_T(X_{\mathbf{\lambda}})$ and its degenerations to $K_T^q(Y_{\mathbf{u}, \mathbf{w}})$ 
as $\{(\mathfrak{a}_i, \mathfrak{p}_i)\}$, while $\{\upxi_j\}$ stand for the corresponding K\"ahler parameters.
Let us denote the true magnetic frame coordinates of $K^q_T(Y^!_{\mathbf{u}, \mathbf{w}})$ as 
$\{(\xi_i, p_i)\}$ and the corresponding framing torus parameters as $\{a_i\}$. 

For instance, the tRS equations in the electric frame for $X^{\boldsymbol{\lambda}}$ (cf. \eqref{eq:tRSRelXlambda}) will read
\begin{equation}\label{eq:tRSRelXlambda2}
\sum_{\substack{\mathcal{I}\subset\{1,\dots,L\} \\ |\mathcal{I}|=k}}\prod_{\substack{i\in\mathcal{I} \\ j\notin\mathcal{I}}}\frac{\hbar\,{\mfa}_i - {\mfa}_j }{{\mfa}_i-{\mfa}_j}\prod\limits_{m\in\mathcal{I}}{\mfp}_m = \ell_k (\upxi_i)\,.
\end{equation}

Then, according to the consideration from the previous section, the following theorem about 3d mirror duals holds:

\begin{Thm} 
Upon identification of parameters 
\begin{eqnarray}
j:(\mathfrak{a}_k, \mathfrak{p}_k)\rightarrow (\xi_k, p_k)\,,\qquad \upxi_l\rightarrow a_l \,,\qquad \hbar\to \hbar^!=\hbar^{-1}
\end{eqnarray}
 we have the following isomorphism:
\begin{eqnarray}\label{mirisom}
K_T^q(Y_{\mathbf{u}, \mathbf{w}})\cong K_T^q(Y^!_{\mathbf{u}, \mathbf{w}}).
\end{eqnarray}
\end{Thm}

\begin{Rem}In other words, given a Nakajima quiver variety $X_\lambda^{\lambda^!}$ of type A such that for each vertex
\begin{equation}\label{eq:linkingnumberscond}
2{\bf v}_i\leq {\bf v}_{i-1}+{\bf v}_{i-1}+{\bf v}_{i+1}+{\bf w}_{i}
\end{equation}
Then its 3d mirror dual variety is $X_{\lambda!}^{\lambda}$.
\end{Rem}

\begin{proof}
Let us consider the electric frame formulation of  $K_T^q(Y_{\mathbf{u}, \mathbf{w}})$. It is obtained by degenerations of the full 
flag quiver as prescribed in Section \ref{Sec:FlagElDegen}. Namely, we impose $\mathfrak{a}_i=\mathfrak{a}_j\hbar^{c_{ij}}$. Let us apply mirror map $j$ and write down the relations \eqref{eq:tRSRelXlambda2} using the Wronskian matrix with polynomials $\{\tilde{s}_i\}$ of degree one, as in \eqref{eq:MiuraQOperCond2}, so that its determinant is equal to
\begin{eqnarray}
W(z)=\prod^{r+1}_{k=1}\prod^{\bf{v}_r-\bf{v}_{r-1}}_{l_k=0}(z-a_k\hbar^{-l_k}),
\end{eqnarray}
which corresponds directly to the degeneration of $\{\upxi_i\}$ in the partial flag. 
Now we construct the minors $\mathscr{D}_k$ according to the prescribed degeneration procedure of $\mathfrak{a}_i$ and $\xi_i$ as it is done in the Proposition \ref{Prop:opertrue}.  Thus we obtain the minors and polynomials $\{s_k\}$ thanks to Lemma  \ref{Th:existencePoly}. At this step we have the partial flag oper with regular singularities given by the polynomial $W$. At the same time, from electric frame degeration procedure, we know that tRS system produces Bethe equations for quivers with no relations $a_i=a_j\hbar^{c_{ij}}$ between singularities. That means that we have to pick the appropriate reduction of the Bethe equations for partial flag. This means that minors $\mathscr{D}_i$ are divisible by the appropriate multipliers of 
$W(z)$, thus producing the $QQ$-system with regular singularities given by $\Lambda_i$ with distinct roots )producing appropriate framing of quiver). Counting the degrees of $s_k$ we obtain that they are defined by the linking numbers of the quiver $Y^!_{\mathbf{u}, \mathbf{w}}$, thus producing the $\hbar$-oper $\hbar$Op$(Y^!_{\mathbf{u}, \mathbf{w}})$ and thus the isomorphism \eqref{mirisom}.


\end{proof}

\begin{Rem}
Essentially the 3d Mirror map from a point of view of tRS system gives different realizations of the algebra of functions on the family of the intersection of Lagrangian submanifolds $\{\rho_i=a_i\}\cap\mathcal{L}^e_{\xi}$ withe some degenerations 
$\{a_i=\hbar^{c_ij}a_j\}$ and 
$\{x_k=\hbar^{kl}\xi_l\}$. 
In one formulation it produces function on 
$\hbar$Op$(Y_{\mathbf{u}, \mathbf{w}})$ in the electric frame on the other hand it produces functions on $\hbar$Op$(Y^!_{\mathbf{u}, \mathbf{w}})$ in the magnetic frame.
\end{Rem}

\subsection{Examples of 3d Mirror Pairs.}
In this section, we will look at the first nontrivial examples of 3d Mirror symmetry studied in the literature in various contexts.
In the end, we take a look at the most important example, the self-mirror quiver $X_{k,l}$, which we will need in the next Section.
Since under the mirror map parameter $\hbar$ gets inverted, in what follows, for the brevity of exposition, we shall write magnetic relations with $\hbar$ and electric relations with $\hbar^{-1}$.

\subsubsection{Complete flag}
Let $X_\lambda^{\lambda^!}=\mathbb{F}Fl_n$. In this case, the linking numbers are $\lambda=\lambda^!=(1,\dots,1)$. In the electric frame 
$T=\text{diag}(\xi_1,\dots,\xi_n)$ and $M$ has a canonical form of the tRS Lax matrix \eqref{eq:LaxFullFormula}. The energy equations are thus complete (unrestricted) tRS relations
\begin{equation}\label{eq:tRSRelFull}
\sum_{\substack{\mathcal{I}\subset\{1,\dots,L\} \\ |\mathcal{I}|=k}}\prod_{\substack{i\in\mathcal{I} \\ j\notin\mathcal{I}}}\frac{\hbar\,{\mfa}_i - {\mfa}_j }{{\mfa}_i-{\mfa}_j}\prod\limits_{m\in\mathcal{I}}{\mfp}_m = e_k (\upxi_1,\dots,\upxi_n)\,,
\end{equation}
where $e_k$ are elementary symmetric polynomials.

In the magnetic frame, the tRS relations are
\begin{equation}\label{eq:tRSRelFullMag}
\sum_{\substack{\mathcal{I}\subset\{1,\dots,L\} \\ |\mathcal{I}|=k}}\prod_{\substack{i\in\mathcal{I} \\ j\notin\mathcal{I}}}\frac{{\xi}_i - 
\hbar\,{\xi}_j }{{\xi}_i-{\xi}_j}\prod\limits_{m\in\mathcal{I}}{p}_m = e_k (a_1,\dots,a_n)\,.
\end{equation}
The mirror map is straightforward
\begin{equation}
a_i = \upxi_i,\qquad \mfa_i=\xi_i\,,\qquad \mfp_i = p_i\,,\qquad i =1,\dots, n\,.
\end{equation}

\begin{figure}[!h]
\includegraphics[scale=0.25]{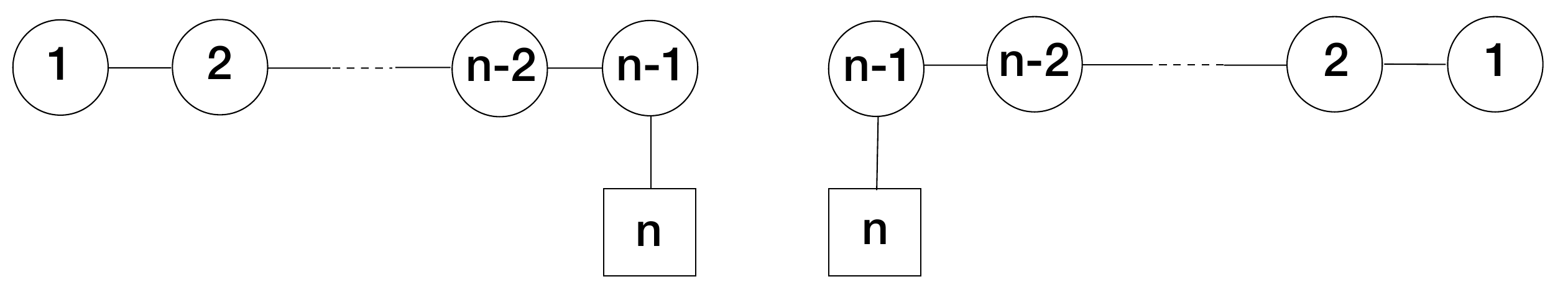}
\caption{$\mathbb{F}Fl_n$ and its 3d mirror dual quiver.}
\label{fig:Flagsmirror.pdf}
\end{figure}

Notice that the 3d mirror dual to $\mathbb{F}Fl_n$ is the same quiver variety that underwent an $\mathbb{Z}_2$ outer automorphism. 

\vskip.1in

Recall that in general, if quivers $X$ and $X^!$ are 3d mirror dual to each other we have in fact two dualities simultaneously -- electric frame of $X$ vs. magnetic frame of $X^!$ and magnetic frame of $X$ vs. electric frame of $X^!$. In the case of $X=F\mathbb{F}l_n$, as well as for all self-dual quivers, all four descriptions are equivalent.

\subsubsection{The Cotangent Bundle to the Grassmannian}
Let $\lambda = (k,n-k)$ and $\lambda^!=(1,\dots,1)$. It means that $\lambda^!$ is generic. Let us denote $X^{\lambda}=T^*Gr_{k,n}$. 
On the mirror side quiver $X_\lambda$ is the $A_{n-1}$ quiver with the following data
\begin{equation}
{\bf v}_1 = 1, {\bf v}_2 = 2\dots, {\bf v}_k = k,\dots, {\bf v}_{n-k} = k,\dots, {\bf v}_{n-2} = 2, {\bf v}_{n-1} = 1\,,
\end{equation}
and framing
\begin{equation}
{\bf w}_k ={\bf w}_{n-k} = 1\,,
\end{equation}
with all other framings trivial, see \figref{fig:GrMirr}.
\begin{figure}[h]
\includegraphics[scale=0.25]{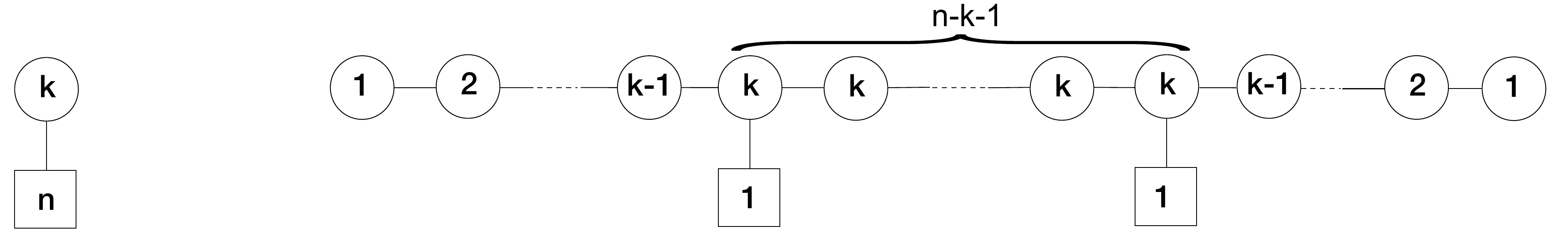}
\caption{$T^*Gr_{k,n}$ and its 3d mirror dual $A_{n-1}$ quiver.}
\label{fig:GrMirr}
\end{figure}

Below we list all four sets of relations for the tRS models in question

\begin{itemize}
\item Magnetic frame for $T^*Gr_{k,n}$. The relations are 
\begin{equation}\label{eq:MagneticGr}
\det(u-T)=\prod_{i=1}^n (u-a_i)\,.
\end{equation}
Matrix $T$ has Slodowy form  which is obtained from the unrestricted Lax matrix by imposing
$\xi_2 = \hbar \xi_1\,, \xi_3= \hbar \xi_2,\dots, \xi_{k}=\hbar \xi_{n-k}$ and $\xi_{n-k-2} = \hbar \xi_{n-k-1}\,,\dots, \xi_{n}=\hbar \xi_{n-1}$.

\item Electric frame for $T^*Gr_{k,n}$. The relations are 
\begin{equation}\label{eq:ElectricGr}
\det(u-M)=\prod_{i=1}^k (u-\hbar^{i-1}\upxi_1) \prod_{i=1}^{n-k} (u-\hbar^{i-1}\upxi_2)\,.
\end{equation}
Matrix $M$ is the unrestricted tRS $n$-body Lax matrix written in terms of $\mfa_i$ and $\mfp_i$.

\item Magnetic frame for $X_\lambda$.
The relations are 
\begin{equation}\label{eq:MagneticGrd}
\det(u-T)=\prod_{i=1}^k (u-\hbar^{i-1}a_1) \prod_{i=1}^{n-k} (u-\hbar^{i-1}a_2)\,,
\end{equation}
where $T$ is the $n$-body Lax matrix in terms of $\xi_i$ and $p_i$.

\item Electric frame for $X_\lambda$.
The relations are 
\begin{equation}\label{eq:ElectricGrd}
\det(u-M)=\prod_{i=1}^n (u-\upxi_i)\,,
\end{equation}
where $M$ has Slodowy form according to $\lambda$, namely $\mfa_2 = \hbar^{-1} \mfa_1\,, \mfa_3= \hbar^{-1} \mfa_2,\dots, \mfa_{k}=\hbar^{-1} \mfa_{n-k}$ and $\mfa_{n-k-2} = \hbar^{-1} \mfa_{n-k-1}\,,\dots, \mfa_{n}=\hbar^{-1} \mfa_{n-1}$ (see example \eqref{eq:TRSSlodowy25} for $T^*Gr_{2,5}$).

\end{itemize}

We can see that \eqref{eq:MagneticGr} coincides with \eqref{eq:ElectricGrd} and \eqref{eq:ElectricGr} coincides with \eqref{eq:MagneticGrd} via the 3d mirror map
$$
a_{i}\mapsto \upxi_i\,,\qquad \mathfrak{a}_a \mapsto \xi_a\,,\qquad \mathfrak{p}_a=p_a\,,\qquad \hbar\mapsto\hbar^{-1},
$$
where $a=1,\dots, n$ and $i=1,2$.

\subsubsection{Mirror for general $A_r$ Quiver}
Let $X_\lambda^{\lambda^!}$ be an $A_3$ quiver with labels ${\bf v}=(1,4,3)$ and ${\bf w}=(1,2,2)$. Then the linking numbers are $\lambda=(3,2,2,1,1)$ and $\lambda^!=(4,2,2,1)$. After swapping them and recalculating the dual labels we find the 3d mirror dual $A_4$ quiver $X^\lambda_{\lambda^!}$ with labels ${\bf v}^!=(1,1,2,1)$ and ${\bf w}^!=(1,0,2,1)$ (see \figref{fig:Amirrgen}).

\begin{figure}[!h]
\includegraphics[scale=0.47]{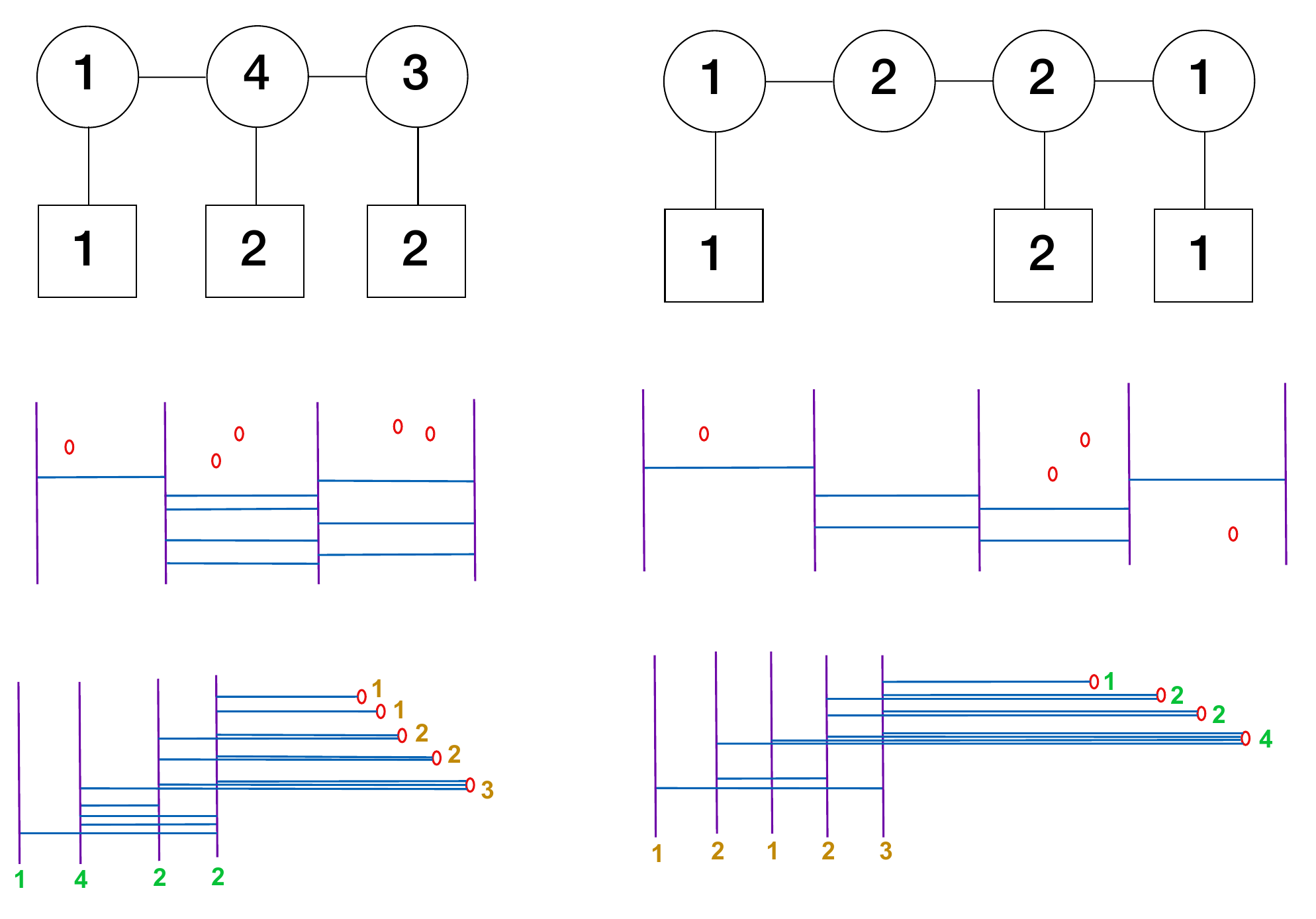}
\caption{3d mirror dual quivers with their labels.}
\label{fig:Amirrgen}
\end{figure}

In this example both partitions have sizes nine $|\lambda|=|\lambda^!|=9$ thus the tRS description of both quiver varieties involves a $9\times 9$ Lax matrix. The breakdown of the four frames goes as follows.

In the magnetic frame of $X_\lambda^{\lambda^!}$ (left of \figref{fig:Amirrgen}) we have relations 
\begin{equation}\label{eq:MagneticXgen}
\det(u-T)=(u-a_1)(u-\hbar a_1)(u-\hbar^2 a_1)\cdot (u-a_2)(u-\hbar a_2)\cdot (u-a_3)(u-\hbar a_3)\cdot(u-a_7)(u- a_8)(u- a_9)\,,
\end{equation}
where $T$ has Slodowy form for $\lambda^!$, namely in the presence of constraints 
$$
\xi_1=\hbar \xi_2 = \hbar^2 \xi_3 = \hbar^3 \xi_4\,,\quad \xi_5=\hbar \xi_6\,,\quad \xi_7=\hbar \xi_8\,.
$$

According to the mirror map \eqref{eq:MagneticXgen} also describes the electric frame of $X^\lambda_{\lambda^!}$ where $\xi_i$ are mapped to $\mfa_i$, $a_i$ are sent to $\upxi_i$, and $\hbar$ is inverted.

\vskip.1in
In the electric description of $X_\lambda^{\lambda^!}$ we get 
\begin{equation}\label{eq:ElectricXgen}
\det(u-M)=(u-\upxi_1)(u-\hbar^{-1} \upxi_1)(u-\hbar^{-2} \upxi_1)(u-\hbar^{-3} \upxi_1)\cdot (u-\upxi_5)(u-\hbar^{-1} \upxi_5)\cdot(u-\upxi_7)(u-\hbar^{-1} \upxi_7)(u- \upxi_9)\,,
\end{equation}
and where $M$ has Slodowy form for $\lambda$:
$$
\mfa_1=\hbar^{-1} \mfa_2 = \hbar^{-2} \mfa_3 = \hbar^{-3} \mfa_4\,,\quad \mfa_5=\hbar^{-1} \mfa_6\,,\quad \mfa_7=\hbar^{-1} \mfa_8\,.
$$

According to the mirror map \eqref{eq:ElectricXgen} also describes the magnetic frame of $X^\lambda_{\lambda^!}$ where $\upxi_i$ are mapped to $a_i$, $\mfa_i$ are sent to $x_i$, and $\hbar$ is inverted.

\subsubsection{Self-Dual Family $X_{k,l}$}
Consider quiver variety $X_{k,l}$ with K\"ahler parameters $\zeta_i=\frac{\xi_{i}}{\xi_{i+1}}$, equivariant parameters $a_k,\dots,a_{k+r-1}$ and $a_{k+r,1},\dots a_{k+r,k+1}$, and scaling weight of the cotangent fibers is $\hbar$ (see top of \figref{fig:mirroranew}). On the bottom of \figref{fig:mirroranew} we see quiver variety $X_{k,l}^!$ has equivariant parameters $\mathfrak{a}_{1,1,},\dots,\mathfrak{a}_{1,k+1}$ and $\mathfrak{a}_2\dots\mathfrak{a}_r$ as well as K\"ahler parameters $\mathfrak{z}_i=\frac{\upxi_i}{\upxi_{i+1}}$ and $\hbar^{!}$ for the weight of the $\mathbb{C}^\times$ action.
\begin{figure}[!h]
\includegraphics[scale=0.5]{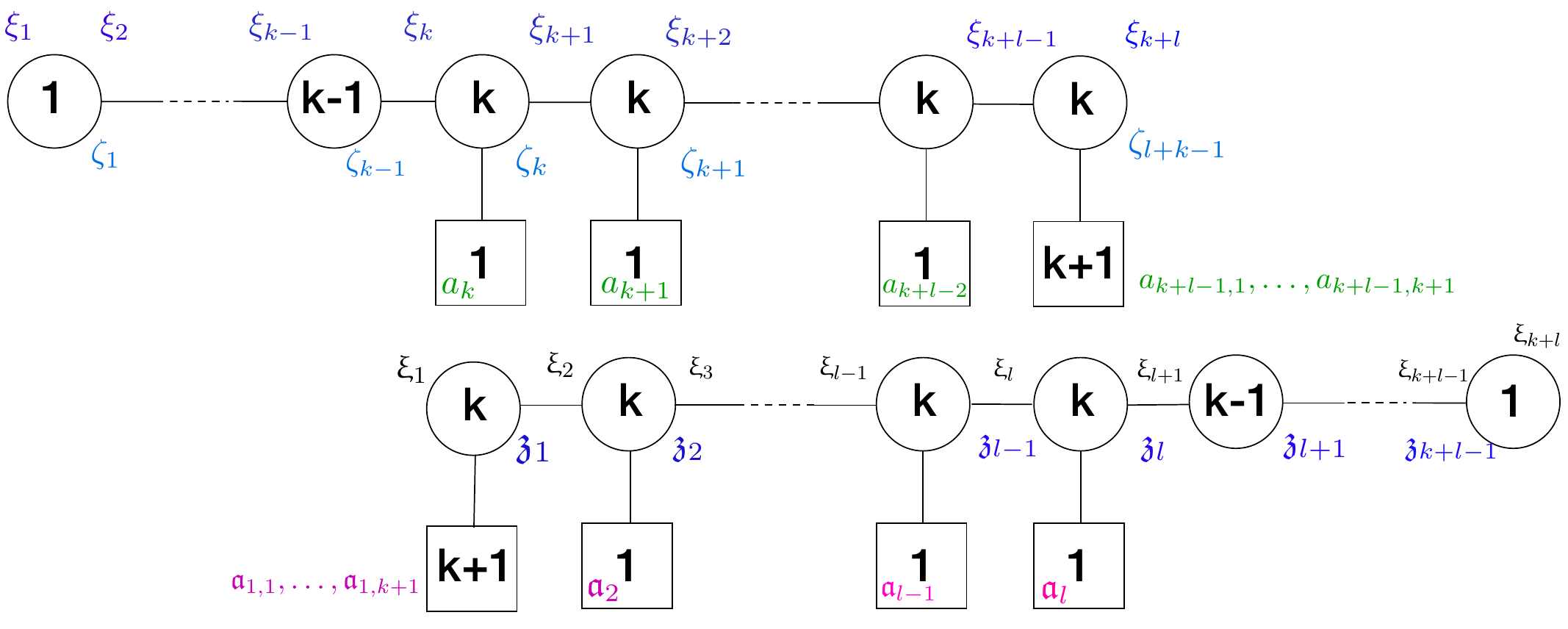}
\caption{$X_{k,l}$ (top) and its mirror $X_{k,l}^!$ (bottom).}
\label{fig:mirroranew}
\end{figure}

Thus we have the following Theorem.

\begin{Thm}\label{Th:MirrorXkr}
Under 3d mirror symmetry quiver varieties $X_{k,l}$ and $X_{k,l}^!$ are dual to each other. Moreover
\begin{equation}
K_T (X_{k,l}) \simeq K_{T^!} (X_{k,l}^!)\,,
\end{equation}
where the K\"ahler and equivariant tori parameters of $X_{k,l}$ and $X_{k,l}^!$ are mapped to each other as follows
\begin{align}\label{eq:MirrorMapTorus}
\upxi_j &= a_{k+l,j}\,,\quad j=1,\dots,k+1,\qquad \upxi_{k+i}= a_i\,,\quad i = 2,\dots,l\notag\\
\mathfrak{a}_{k+l,j}&=\xi_j\,,\quad j=1,\dots,k+1,\qquad  \mathfrak{a}_{i}=\xi_{k+i},\quad i=2,\dots, l\,,
\end{align}
as well as $\hbar^!=\hbar^{-1}$.
\end{Thm}

\begin{Rem}

Geometrically the conjugate momenta correspond to the following quantum K-theory classes of $X_{k,l}$ (see \cite{Koroteev:2017aa})
in the magnetic frame:
\begin{equation}
p_i = \prod_{m=i}^{k+r} \prod_{l=1}^{M_m} a_{m,l} \cdot \frac{Q^+_{i-1}(0)}{Q^+_{i}(0)}= \bigotimes_{m=i}^{k+r+1}W_m\circledast \widehat{\Lambda^{i-1}V_{i-1}}\circledast \widehat{\Lambda^i V^*_i}\,
\end{equation}
and in the electric frame:
\begin{equation}
p^a_{j} = (-\hbar)^k\prod_{l=1}^j \xi_l \cdot \frac{Q^+_j\left(\hbar a_{j}\right)}{Q^+_j\left(a_{j}\right)}\,,\quad j = r,\dots, k+l-1,\,\qquad  p^a_{k+l,j} = (-\hbar)^k\prod_{l=1}^{k+l} \xi_l \cdot \frac{Q^+_{k+l,j}\left(\hbar a_{k+l,j}\right)}{Q^+_{k+r,j}\left(a_{k+r,j}\right)}.
\end{equation}

As for the mirror dual variety $X_{k,l}^!$ we get the following formulae for the electric frame
\begin{equation}
\mathfrak{p}_{1,c} = (-\hbar^!)^{-k}\prod_{l=1}^j \upxi_1 \cdot \frac{Q^{+,!}_1\left(\hbar^! \mathfrak{a}_{1,c}\right)}{Q^{+,!}_1\left(\mathfrak{a}_{1,c}\right)}\,,\quad c=1,\dots,k+1\,,\qquad
\mathfrak{p}_{b} = (-\hbar^!)^{-k}\prod_{l=1}^b \upxi_l \cdot \frac{Q^{+,!}_b\left(\hbar^! \mathfrak{a}_{b}\right)}{Q^{+,!}_b\left(\mathfrak{a}_{b}\right)}\,,
\end{equation}
and for the magnetic frame
\begin{equation}
\mfp^{\xi}_i =\prod_{l=i}^r  \mathfrak{a}_{l} \cdot \frac{Q^{+,!}_{i-1}(0)}{Q^{+,!}_{i}(0)}=\bigotimes_{m=i}^{k+r+1}W^!_m\circledast \widehat{\Lambda^{i-1}V^!_{i-1}}\circledast \widehat{\Lambda^i V^{!,*}_i}\,,
\end{equation}
Upon 3d mirror symmetry $\hbar^{-1} = \hbar^!$.
\end{Rem}

\section{The $A_{\infty}$ Quivers and 3d Mirror self-duality of the Hilbert Scheme}\label{Sec:DirectLimit}

\subsection{$A_{\infty}$ quiver as a direct limit}
In this section we will use the algebras $K_T^{q}(X_{k,l})$ quivers and consider the appropriate reduction to produce the algebra  
$K_T^{q}(X_{k,\infty})$. Here $X_{k,\infty}$ is the 
$A_{\infty}$ quiver with one-dimensional framing for each vertex (see the top picture in \figref{fig:Ainfmirror}).

Consider the $QQ$-system for $X_{k,l}$ and impose the following conditions on some equivariant and K\"ahler parameters
\begin{equation}\label{eq:droptail}
\xi_1= \dots = \xi_{k}=\xi_{k+l}=0\,,\qquad a_{k+l-1, 1}= \dots  =a_{k+l-1, k+1}=0\,.
\end{equation}
Then the system becomes
\begin{align}\label{eq:QQXklTrunc}
Q^+_k(z) Q^-_k(\hbar z) &=z\, Q^+_{k-1}(\hbar z)Q^+_{k+1}(z)\,,\cr
\dots&\dots\cr
\xi_{k+i} Q^+_{k+i}(\hbar z) Q^-_{k+i}(z) - \xi_{k+i+1} Q^+_{k+i}(z) Q^-_{k+i}(\hbar z) &=\Lambda_{k+i} (z) Q^+_{k+i-1}(\hbar z)Q^+_{k+i+1}(z)\,,\\
\dots&\dots\cr s
\xi_{k+l-1} Q^+_{k+l-1}(\hbar z) Q^-_{k+l-1}(z) &=z^{k+1} Q^+_{k+l-2}(\hbar z)\,,\notag
\end{align}
for $i=2,\dots,l-2$. Notice that polynomials $Q^+_{k-1}(z), Q^{\pm}_{k}(z)$ and $Q^-_{k+l-1}$ can be found from the first equation in \eqref{eq:QQXklTrunc} and polynomial $Q^-_{k+l-1}(z)$ can be found from the last one. The rest of the equations look alike. The truncated $QQ$-system \eqref{eq:QQXklTrunc} is describes quiver in \figref{fig:xkltrunc}.
\begin{figure}[!h]
\includegraphics[scale=0.33]{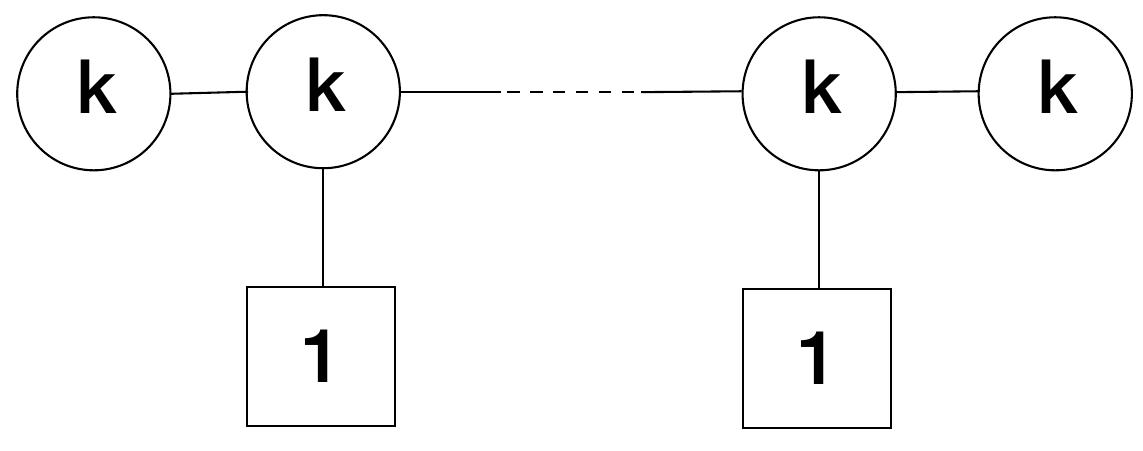}
\caption{Truncated $X_{k,l}$ Quiver.}
\label{fig:xkltrunc}
\end{figure}

We can define the reduced algebra as follows
\begin{eqnarray}
K^{{\rm red},q}_T(X_{k,l}):=\frac{S(a_{k+2},\dots a_{k+l-2};\xi_{k+2},\dots\xi_{k+l-2}, \hbar)(\{s_{i,k}\})}{\rm{(Bethe\, Ansatz\, for \,(\ref{eq:QQXklTrunc}))}}\,,
\end{eqnarray}
where we take quotient over the Bethe ansatz equations which correspond to the reduced $QQ$-system (\ref{eq:QQXklTrunc}).

This algebra remains self-dual, since the eliminated variables are mirror map to each other according to the homomorphism which was described in Theorem \ref{Th:MirrorXkr}. Moreover, this algebra is significantly simplified thanks to the reduced $QQ$-system \eqref{eq:QQXklTrunc}.

\begin{Rem}
The first $k-1$ equations of the $QQ$-system for $X_{k,l}$ become trivial in the limit when $\xi_1,\dots, \xi_k$ go to zero if the coefficients of the corresponding polynomials $\Lambda_i(z)$ are equal to $(\xi_i-\xi_{i+1})$. This limit corresponds to the classical equivariant K-theory for the full flag, which is a `tail' of $X_{k,l}$.
\end{Rem}

Thus, there is a natural homomorphism 
\begin{eqnarray}
i_k:K^{{\rm red},q}_T(X_{k,l})\to K^{{\rm red},q}_T(X_{k+1,l}).
\end{eqnarray}
Now one can take a direct limit of $K^{red,q}_T(X_{k,l})$ with respect to $l$.  
As a result we obtain the algebra, which can be characterized as follows:
\begin{eqnarray}
K^q_T(X_{k, \infty}):= \frac{S(\{a_i\},\{\xi_k\}, \hbar^{\pm 1})(s_{i,k})}{{\rm (Bethe~ equations)}},
\end{eqnarray}
where, as usual $\{s_{i,k}\}$ are the Bethe roots, i.e. roots of $Q_i^+(z)$, which satisfy Bethe equations, or equivalently the following infinite $QQ$-system: 
\begin{align}\label{eq:QQinfQform}
\xi_{i+1} \,Q^+_i(\hbar z)Q^-_i(z)-\xi_i\,  Q^+_i(z)Q^-_i(\hbar z)&=\Lambda_i (z) Q^+_{i-1}(\hbar z)Q^+_{i+1}(z)\,,\quad i\in\mathbb{Z}\,,
\end{align}
where $\Lambda_i(z)=z-a_i$.
Thus we arrive to the following theorem:

\begin{Thm}
The following is true:
\begin{enumerate}
\item[i)] $ K^q_T(X_{k, \infty})\cong \lim\limits_{\rightarrow}K^{red,q}_T(X_{k,l}).$
\item[ii)] The algebra $K^q_T(X_k, \infty)$ is 3d-mirror self-dual, i.e. 
\end{enumerate}
\begin{eqnarray}\label{infmirror}
K^q_T(X_{k, \infty})\vert_{\{\xi_{i}= \mathfrak{a}_{i}\}, \{a_{i}=\upxi_{i}\}, \hbar^!=\hbar^{-1}}\cong K^q_T(X^!_{k, \infty}),
\end{eqnarray}
so that the quivers $X_{k, \infty}$ and $X^!_{k, \infty}$ are given in \figref{fig:Ainfmirror}. 
\end{Thm}

\begin{figure}[!h]
\includegraphics[scale=0.4]{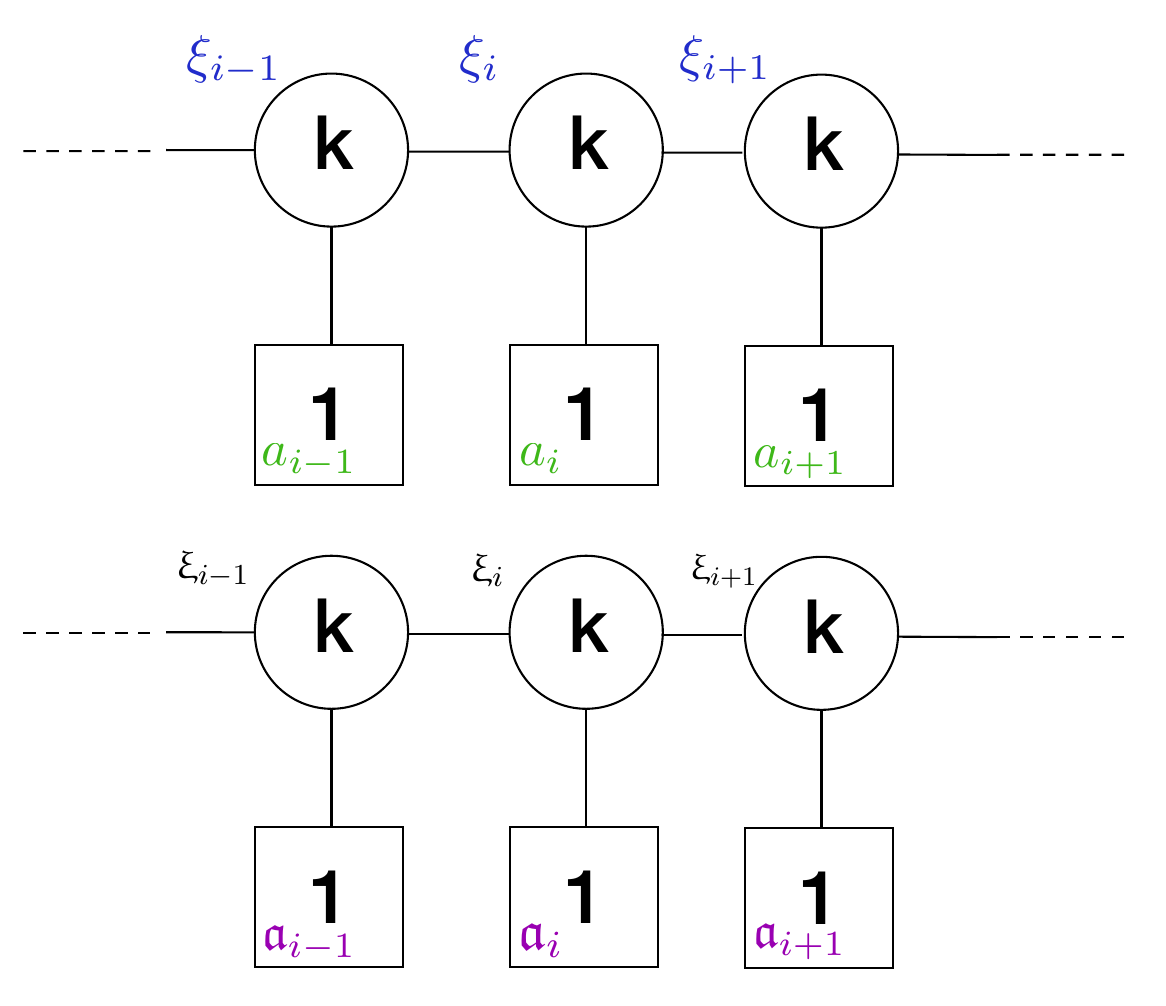}
\caption{$X_{k,\infty}$ and its mirror $X_{k,\infty}^!$.}
\label{fig:Ainfmirror}
\end{figure}

\vskip.1in

\textbf{Remark.}
As we have seen the nondegenerate solutions of the finite $QQ$-systems of rank $r$ were naturally related to the nondegenerate $Z$-twisted Miura $(SL(r+1), \hbar)$-opers. The infinite $QQ$-system can be related in a similar fashion to the Z-twisted Miura $(\overline{GL}(\infty), \hbar)$-opers, where $\overline{GL}(\infty)$ is a certain completion of $GL(\infty)$.  We won't discuss this object here and refer the reader to \cite{Koroteev:2020mxs} for more details.

\vskip.1in

Note that there is a natural action of the additive  group:
\begin{Prop}\label{symminf}
The algebra $K^q_T(X_{k, \infty})$ is invariant under the action of a 
symmetry group generated by the overall shift of vertices in $X_{k, \infty}$ quiver: 
\begin{eqnarray}
s_{i,k}\to s_{i+m,k}, \quad a_i\to a_{i+m}, \quad \xi_i\to \xi_{i+m}\,,\qquad m\in\mathbb{Z}\,.
\end{eqnarray}
The 3d Mirror map (\ref{infmirror}) is equivariant with respect to this symmetry.
\end{Prop}

\subsection{Periodic Boundary Conditions for the $QQ$-systems and 3d mirror self-duality for the Hilbert scheme}
The symmetry transformation from Proposition \ref{symminf} for the $QQ$-system associated with the quiver $X_k,\infty$ has a fixed point if the following periodic conditions are imposed:
\begin{eqnarray}\label{eq:xicqQ}
\xi_i = \xi_0\xi^i\,,\qquad Q^\pm_i(z)=\mathcal{Q}^\pm(t^iz)\,,\qquad \Lambda_i(t^iz) = {\xi_i} \Lambda(z)\,,
\end{eqnarray}
where $\Lambda(z)=z-a_0$ and $t\in\mathbb{C}^\times$ is the new parameter. 
We note that we introduced the constant coefficients in $Q^i_{\pm}$ and $\Lambda_i(z)$, which do not change the Bethe equations which produce the relations for the algebra $K_T(X_{k, \infty})$. 
The following Proposition holds.
 
\begin{Prop} 
Upon the periodic conditions \ref{eq:xicqQ}, the infinite $QQ$-system (\ref{eq:QQinfQform})  reduces to the following equation: 
\begin{equation}\label{eq:QQgl1}
\xi\mathcal{Q}^+(\hbar z)\mathcal{Q}^-( z)- \mathcal{Q}^+(z)\mathcal{Q}^-( \hbar z) = \Lambda (z)\mathcal{Q}^+(tz)\mathcal{Q}^+(t^{-1}\hbar z)\,.
\end{equation}
The condition \eqref{eq:xicqQ} translates into the following conditions:
\begin{equation}
a_i = a_0 t^{i}\,, \qquad s_{i,a} =s_at^{i} \,,\qquad a=1,\dots,k\,,\qquad  \Lambda_i(z)=\xi_0\Big(\frac{\xi}{t}\Big)^i(z-a_0)
\end{equation}
where $s_a$ are the roots of $\mathcal{Q}^+(z)$. 
\end{Prop}

\begin{Rem}
This reduction, on the level of Z-twisted Miura $(\overline{GL}(\infty), \hbar)$-opers leads to what we 
called toroidal $\hbar$-opers. The resulting Bethe equations, see below, appear in representation theory of toroidal 
algebra $\widehat{\widehat{\mathfrak{gl}}}(1)$. They have an explicit geometric realization, which we will talk briefly later.
\end{Rem}

We will refer to $K^q_T(X_{k,\infty})$ with imposed periodicity conditions as 
\begin{eqnarray}
K^{q,per}_T(X_{k,\infty}):=\frac{S(\xi, t, \hbar)(\{s_a\})}{(\rm{Bethe~ equations})},  
\end{eqnarray}
where Bethe equations read (for the Hilbert scheme $N=1$)
\begin{eqnarray}\label{eq:ADHMBethe}
\prod_{l=0}^{N-1}\frac{s_a-a_l}{\hbar^{-1} s_a-a_l}\cdot \prod_{\substack{b=1\\b\neq a}}^k \frac{t s_a-s_b}{t^{-1}s_a- s_b} \frac{\hbar^{-1} s_a-s_b}{\hbar s_a-s_b}\frac{t^{-1}\hbar s_a- s_b}{t \hbar^{-1} s_a-s_b} = \xi^{-1}\,,
\end{eqnarray}
for $\a=1,\dots,k$. These equations are produced by the roots of $\mathcal{Q}^{+}(z)$ in the $QQ$-system \eqref{eq:QQgl1}.

For the dual $X^!_{k,\infty}$ quiver with parameters $\mathfrak{a}_i$, $\upxi_i$ we have the following mapping upon 3d mirror symmetry:
\begin{equation}
\mathfrak{a}_i:=\xi_0\xi^i, \quad \upxi_i=a_0 t^i,\quad \hbar^!=\hbar^{-1}.
\end{equation}

Because the translations along the $X_{k, \infty}$ quiver act equivariantly with respect to the 3d mirror map, the periodicity condition for Bethe roots of the dual algebra transform into
$s^!_{i,a} =s^!_a\xi^{i}$.

Thus we obtain the following theorem.

\begin{Thm}\label{Th:Main}
The following two algebras are isomorphic:
\begin{eqnarray}
K^{q,per}_T(X^!_{k,\infty})\vert_{\{\mathfrak{a}_i=\xi_0\xi^i\}, \{\upxi_i=a_0 t^i\},\hbar^!=\hbar^{-1}}\cong K^{q,per}_T(X_{k,\infty}).
\end{eqnarray}
\end{Thm}

The algebra $K^{q,per}_T(X^!_{k,\infty})$ characterized by the $QQ$-system from (\ref{eq:QQgl1}) is known as the quantum K-theory of the quiver variety, corresponding to the quiver with one loop, one vertex and one-dimensional framing \figref{fig:QuivADHM} \cite{Koroteev:2020mxs}. This variety is also known as a Hilbert scheme of points in $\mathbb{C}^2$. 
\begin{figure}[!h]
\includegraphics[scale=0.37]{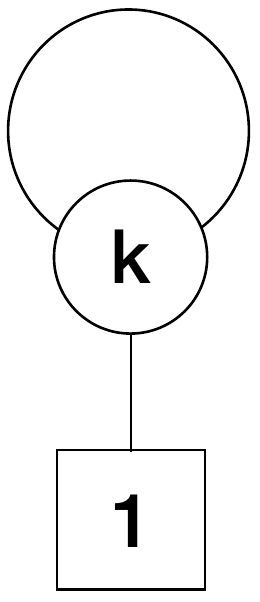}
\caption{Quiver description of Hilb$^k[\mathbb{C}^2]$.}
\label{fig:QuivADHM}
\end{figure}

\vskip.1in
Theorem \ref{Th:Main} can be understood as the 3d mirror self-duality of the quantum K-theory of the Hilbert scheme of points, which is the main result of this paper.

\vskip.1in

In the next section we will remind the reader of this variety and its generalizations, known as ADHM spaces.

\section{3d Mirror Symmetry for the $\mathcal{M}_{N,k}$ Quiver}\label{Sec:MirrorADHM}
In this section we discuss how the quiver variety $\mathcal{M}_{N,k}$, which describes the moduli space of sheaves of rank-$N$ $\mathbb{C}^2$, is transformed under 3d mirror symmetry.

\subsection{The ADHM Spaces}
Let $\mathcal{M}_N$ be the moduli space of rank-$N$ torsion-free sheaves on $\mathbb{P}^2$ with framing at infinity which forces its first Chern class to vanish. The second Chern class ranges over integers thus the space can be represented as a direct sum of disconnected components of all degrees 
$\mathcal{M}_{N} = \bigsqcup_k \mathcal{M}_{N,k}$ \cite{Feigin:2009aa, Negut:2012aa}.

There is an action of maximal torus $T_N:=\mathbb{C}_t^\times\times\mathbb{C}_\hbar^\times\times\mathbb{T}(GL(N))$ on each component $\mathcal{M}_{N,k}$. The first two $\mathbb{C}^\times$ factors act on $\mathbb{P}^2$, while the rest acts on the framing with equivariant parameters $\mathrm{a}_1,\dots \mathrm{a}_N$. The $T_N$-equivariant K-theory of $\mathcal{M}_{N,k}$ is formed by virtual equivariant vector bundles on $\mathcal{M}_{N,k}$. The space $K_{t,\hbar}(\mathcal{M}_{N,k})$ is a module over $\mathbb{C}[t^{\pm 1},\hbar^{\pm 1},\mathrm{a}_1^{\pm 1},\dots \mathrm{a}_N^{\pm 1}]$. 

The relations in the quantum K-theory of $\mathcal{M}_{N,k}$ are given by Bethe equations \eqref{eq:ADHMBethe}, where $\xi$ is the K\"ahler (quantum) parameter (see i.e. \cite{Koroteev:2018}, Section 5).

The space $\mathcal{M}_{N,k}$ can also be described as the moduli spaces of stable representations of the ADHM quiver (left in \figref{fig:MirrorPairsADHMN}).
For $N=1$ this quiver variety describes Hilbert schemes of $k$ points on $\mathbb{C}^2$ \figref{fig:QuivADHM}. The two descriptions in terms of sheaves and the ADHM are equivalent \cite{Nakajima:2014aa,Negut:thesis}.

\subsection{Approximation by A-type quivers}
Consider the following families of quiver varieties -- $X_{k,N,r}$ and $X_{k,N,r}^!$:
\begin{figure}[!h]
\includegraphics[scale=0.37]{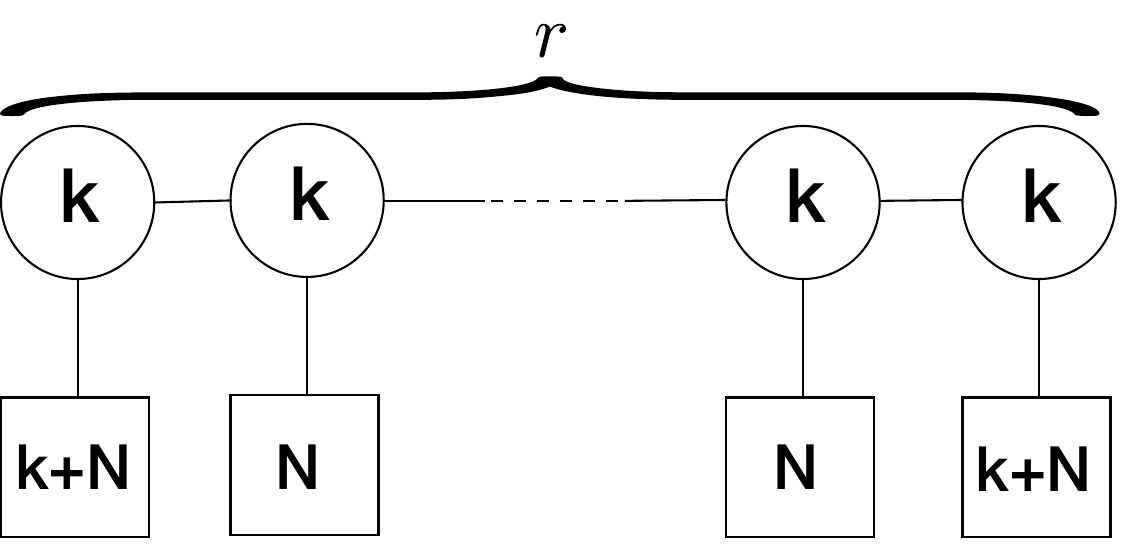}
\caption{The $X_{k,N,r}$ quiver family.}
\label{fig:Ainfmirror2}
\end{figure}

\begin{figure}[!h]
\hspace*{-1cm}\includegraphics[scale=0.5]{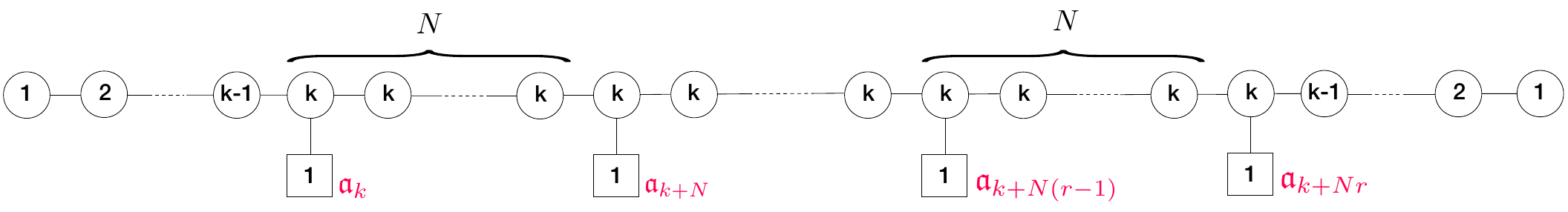}
\caption{The $X_{k,N,r}^!$ quiver family contains $r+1$ framing bundles separated by $N-1$ nodes.}
\label{fig:Binfmirror}
\end{figure}

The following statement we leave for the enthusiastic reader to prove. The proof goes along the way of the proof of Theorem \ref{Th:MirrorXkr}.

\begin{Thm}
Under 3d mirror symmetry quiver varieties $X_{k,N,r}$ and $X_{k,N,r}^!$ are dual to each other. 
\end{Thm}

\subsection{Approximation by A$_\infty$ quivers}
The next step is to take limits $r\to\infty$ for both $X_{k,N,r}$ and $X_{k,N,r}^!$ families along the lines of our derivation above.

\begin{figure}[!h]
\includegraphics[scale=0.4]{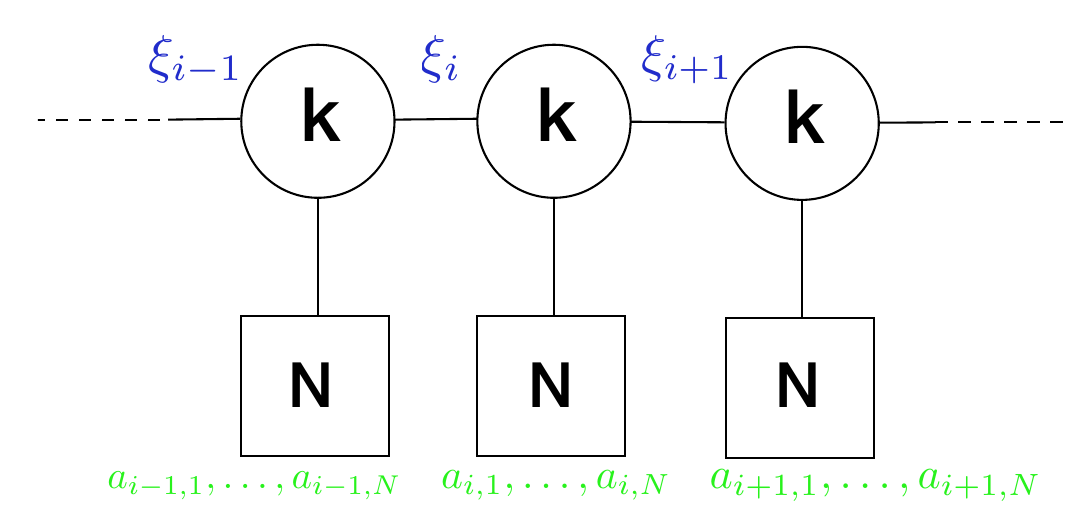}
\caption{The $X_{k,N,\infty}$ quiver family.}
\label{fig:Ainfmirrorinfty}
\end{figure}

\begin{figure}[!h]
\hspace*{-.7cm}\includegraphics[scale=0.5]{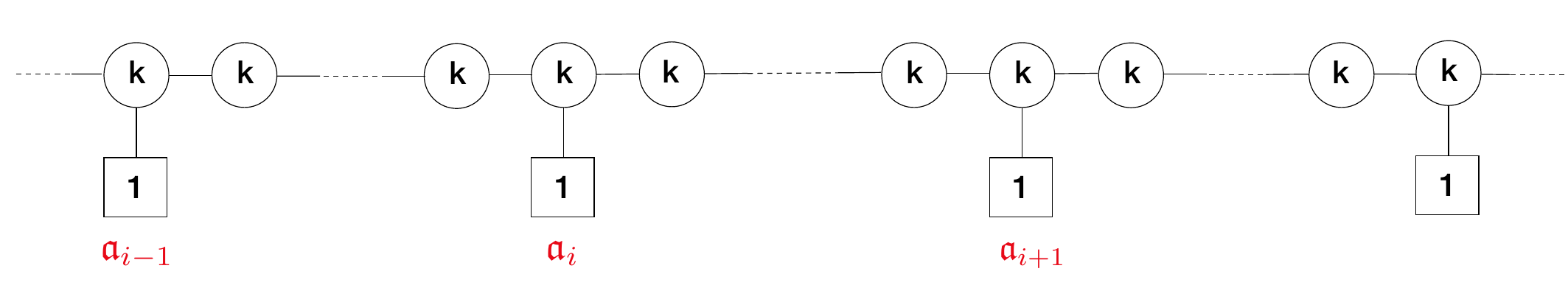}
\caption{The $X_{k,N,\infty}^!$ quiver family.}
\label{fig:Binfmirrorinfty}
\end{figure}

The  K\"ahler parameters of $X_{k,N,\infty}$ and $X_{k,N,\infty}^!$ are given by 
\begin{equation}
\zeta_i = \frac{\xi_{i+1}}{\xi_{i}}\,,\qquad \mathfrak{z}_{i,b} = \frac{\upxi_{i+1,b}}{\upxi_{i,b}}\,,\qquad i\in\mathbb{Z}\,,
\end{equation}
while the equivariant parameters are 
\begin{equation}
a_{j,1},\dots,a_{j,N}\,,\qquad \mathfrak{a}_{j}\,,\qquad j\in\mathbb{Z}\,,
\end{equation}
respectively (see Figures \ref{fig:Ainfmirrorinfty} and \ref{fig:Binfmirrorinfty}).

\begin{Thm}
Under 3d mirror symmetry quiver varieties $X_{k,N,\infty}$ and $X_{k,N,\infty}^!$  are dual to each other. Under the mirror map
\begin{equation}
\xi_i =\mathfrak{a_i}\,\qquad a_{i,b}=\upxi_{i,b}\,,\qquad i\in\mathbb{Z},\qquad b=1,\dots,N\,.
\end{equation}
\end{Thm}

\subsection{$\mathcal{M}_{N,k}$ and its Mirror Dual}
Finally we can impose periodic boundary conditions on the data of quivers $X_{k,N,\infty}$ and $X_{k,N,\infty}^!$ to obtain the desired mirror pair.
\begin{figure}[!h]
\includegraphics[scale=0.6]{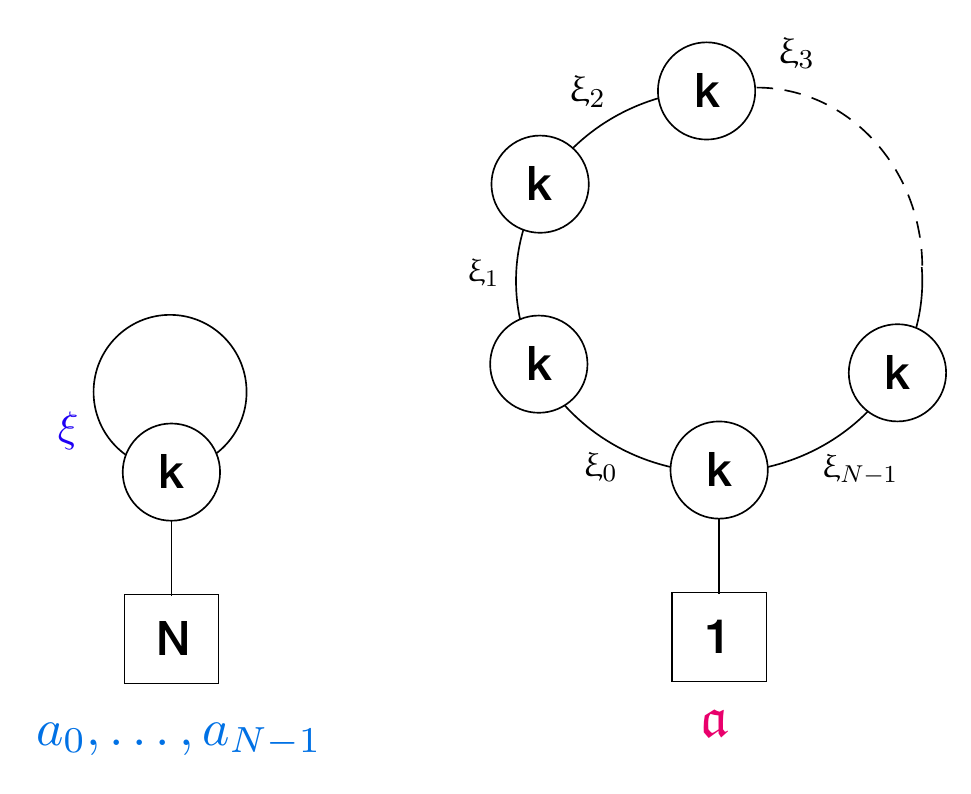}
\caption{The $\mathcal{M}_{N,k}$ quiver (left) and its 3d mirror dual $\widehat{A}_{N-1}$ quiver.}
\label{fig:MirrorPairsADHMN}
\end{figure}

The $QQ$-system for the magnetic frame of $X_{k,N,\infty}$ reads
\begin{align}\label{eq:QQinfQformN}
\xi_{i+1} \,Q^+_i(\hbar z)Q^-_i(z)-\xi_i\,  Q^+_i(z)Q^-_i(\hbar z)&=\Lambda_i (z) Q^+_{i-1}(\hbar z)Q^+_{i+1}(z)\,,\quad i\in\mathbb{Z}\,,
\end{align}
where $\Lambda_i(z)=c_i\prod_{b=0}^{N-1}(z-a_{i,b})$.
The periodic boundary conditions (see \cite{Koroteev:2020mxs}, section 10.3) 
\begin{equation}
\xi_i = \xi_0 \xi^i\,,\qquad a_{i,b}=a_b t^i\,, \qquad Q_i^\pm(z)=\mathcal{Q}^\pm(t^i z)\,, \qquad \Lambda_i(z)= \xi_0 \left(\frac{\xi}{t}\right)^{Ni} \Lambda(z)\,. 
\end{equation}
lead to the following equation
\begin{align}\label{eq:QQinfQformNper}
\xi \,\mathcal{Q}^+(\hbar z)\mathcal{Q}^-(z)-\mathcal{Q}^+(z)\mathcal{Q}^-(\hbar z)&=\Lambda (z) \mathcal{Q}^+(\hbar t^{-1}z)\mathcal{Q}^+(tz)\,,
\end{align}
where $\Lambda(z)=\prod_{b=0}^{N-1}(z-a_b)$ provided that we scaled Bethe roots as $s_{i,a}=s_a t^i$. The K\"ahler parameter of the $\mathcal{M}_{N,k}$ quiver $X_{N}$ is $\xi$ and the equivariant parameters are $a_1,\dots, a_N$.
Notice that when $N=1$ we arrive at \eqref{eq:QQgl1}.

\vskip.2in

Analogously we consider electric frame of the $X_{k,N,\infty}^!$ quiver. 
\begin{align}\label{eq:QQinfQformElN}
\upxi_{i+1} \,Q^{+,!}_i(\hbar^{-1} u)Q^{-,!}_i(u)-\upxi_i\,  Q^{+,!}_i(u)Q^{-,!}_i(\hbar^{-1} u)&=\Lambda^!_i (u) Q^{+,!}_{i-1}(\hbar^{-1} u)Q^{+,!}_{i+1}(u)\,,\quad i\in\mathbb{Z}\,,
\end{align}
where 
\begin{equation}
\Lambda^!_i(u)=
\begin{cases}
c_{i}(u-\mathfrak{a}_i)\,,&\quad i = k N\,,\quad k\in\mathbb{Z}\\
1\,,&\quad \text{otherwise}
\end{cases}
\end{equation}

After imposing periodic boundary conditions 
\begin{equation}
\upxi_{i+N-1} = \upxi_{i-1} t^{i+N-1}\,,\qquad \mathfrak{a_i} = \mathfrak{a}_0 \xi^{i+N}\,,\qquad Q^{\pm,!}_{i+N}(u)=\mathcal{Q}^{\pm,!}_i(\xi^N u)\,, \qquad \Lambda^!_i(u)= \upxi_i\left(\frac{t}{\xi}\right)^{Ni}\Lambda^!(u)\,,
\end{equation}
where $\Lambda^!(u)=u-\mathfrak{a}_0$.

The resulting $QQ$-system will contain $N$ coupled equations
\begin{align}
\upxi_{i+1} \mathcal{Q}^{+,!}_i(\hbar^{-1} u)\mathcal{Q}^{-,!}_i(u) - \upxi_i \mathcal{Q}^{+,!}_i(u)\mathcal{Q}^{-,!}_i(\hbar^{-1} u) &= \mathcal{Q}^{+,!}_{i-1}(\hbar^{-1} \xi^{-1} u)\mathcal{Q}^{-,!}_i(\xi u)\,,\qquad i =1,\dots, N-1\notag\\
t^N\upxi_{0} \mathcal{Q}^{+,!}_1(\hbar^{-1} u)\mathcal{Q}^{-,!}_1(u) - \upxi_{N-1} \mathcal{Q}^{+,!}_1(u)\mathcal{Q}^{-,!}_i(\hbar^{-1} u) &= \Lambda^!(u)\mathcal{Q}^{+,!}_{N}(\hbar^{-1} \xi^{-1} u)\mathcal{Q}^{-,!}_1(\xi u)\,,
\end{align}
where $\Lambda^!(u)=u-\mathfrak{a}_0$.
The K\"ahler parameters of the $\widehat{A}_{N-1}$ quiver $X_N^!$ are $\frac{\upxi_1}{\upxi_0},\dots,\frac{\upxi_{N-1}}{\upxi_{N-2}},\frac{\upxi_0}{\upxi_{N-1}}$, while its sole equivariant parameter is $t$.

Again we can see that for $N=1$ we get back the $QQ$-system for the Hilbert scheme.

\vskip.1in

We can formulate the final theorem
\begin{Thm}
Under 3d mirror symmetry the ADHM quiver varieties $X_{N}$ is dual to the $\widehat{A}_{N-1}$ quiver variety $X_N^!$ (right quiver in \figref{fig:MirrorPairsADHMN})
Upon this duality the torus parameters of $X_1$ and $X_1^!$ are mapped to each other as follows
\begin{equation}
\{\hbar,\xi,t,a_0,\dots,a_{N-1}\} \longleftrightarrow \{\hbar^{-1},t,\xi,\upxi_0,\dots,\upxi_{N-1}\}\,.
\end{equation}
\end{Thm}

\bibliography{cpn1}
\end{document}